\documentclass[a4paper,12pt]{article}

\usepackage{amsthm,mathtools,amsfonts}

\newtheorem{theorem}{Theorem}[section]
\newtheorem{lemma}[theorem]{Lemma}

\newtheorem{definition}[theorem]{Definition}
\newtheorem{corollary}[theorem]{Corollary}
\newtheorem{remark}[theorem]{Remark}
\newtheorem{assumption}[theorem]{Assumption}

\usepackage{listings}
\usepackage{appendix}
\usepackage{graphicx}
\usepackage{wrapfig}
\usepackage{pdfpages}
\usepackage{psfrag}
\usepackage{rotating}
\usepackage{subfig}
\usepackage{mathptmx}
\usepackage{epstopdf} 
\usepackage{mathpazo} 

\usepackage{hyperref}       
      
\begin{document}

\title{A Compressive Sampling Approach To Adaptive Multi-Resolution Approximation of Differential Equations With Random Inputs}
\author{Behrooz Azarkhalili$^*$}
\date{}

\maketitle

\let\oldthefootnote\thefootnote
\renewcommand{\thefootnote}{\fnsymbol{footnote}}
\footnotetext[1]{ Department of Mathematical Science, Sharif University of Technology, Tehran, Iran. \\Email: \url{b_azarkhalili@alum.sharif.edu}}
\let\thefootnote\oldthefootnote


\begin{abstract}
In this paper, a novel method to adaptively approximate the solution to stochastic differential equations, which is based on compressive sampling and sparse recovery, is introduced. The proposed method consider the problem of sparse recovery with respect to multi-wavelet basis (MWB) from a small number of random samples to approximate the solution to problems. To illustrate the robustness of developed method, three benchmark problems are studied and main statistical features of solutions such as the variance and the mean of solutions obtained by proposed method are compared with the ones obtained from Monte Carlo simulations.

\end{abstract}


\section{Introduction}

Recently, spectral representation methods based on the polynomial chaos (PC) decomposition [1-2] provide a robust, efficient, and practical alternative approach instead of Monte Carlo methods to be utilized in uncertainty quantification. Today, spectral representations have been employed to the many engineering and science fields. For instance, spectral methods can be utilized in fluid mechanics [3-7], thermal engineering [8], random oscillations [9], and porous media [10-11].

\subsection{Multi-Resolution Analysis (MRA)}
 
Although spectral PC methods are efficient methods to solve interest problems, they have limited accuracy when the solutions lack sufficient regularity and smoothness [12]. One of the main limitations occurs when the solutions have sharp gradient or discontinuity with respect to random input data [13]; specifically, the accumulated errors and the Gibbs-type phenomena cause spectral PC method become inefficient. To overcome such restrictions, a few recent work focus on multi-resolution analysis (MRA) and multi-wavelet basis (MWB) method [13-17].In this approach, one can increase the accuracy of constructed MRA scheme by increasing both the polynomial order ($p$-refinement) and resolution level ($h$-refinement) [13-17]. In several previous studies the concepts of generalized polynomial chaos (GPC) combined with employing piecewise continuous polynomial functions to decompose the random data and the solution into Haar wavelets [13-17]. MRA let us to localize decomposition and achieve highly accurate results than general PC methods. Although MWB allows us to decompose random variables into spectral PC expansion the same as other orthonormal basis spectral methods, several fundamental differences exist between multi-wavelet expansion and other PC expansions, three main of which are explained as follows [17]:

\begin{enumerate}

\item
In spectral PC methods, orthonormal polynomials selected specially so that when the problem satisfies adequate smoothness and regularity conditions, we can expect an "infinite-order" mean-squared convergence. It means that the mean squared error of the approximation approaches zero when the number of terms in the spectral representation approaches infinity; however, this type of convergence rate cannot be achieved in multi-wavelet expansion since MWB is localized.

\item
Comparing with global basis spectral representations, Multi-wavelet representation leads to localized decompositions which provides the possibility of a highly accurate behavior; however, with slower rate of convergence. One can expect that in situations where the responses of the problem show a localized sharp gradient or a discontinuous variation, the wavelet decomposition become more efficient than a spectral expansion, whose convergence can be destroyed due to the Gibbs-type phenomena and accumulated error.

\item 
The most distinctive feature of MWB deals with the products of piecewise constant functions. The product of two polynomials having degree at most $n$ does not belong to space of polynomials having degree at most $n$. In contrast, One can show that multi-wavelet truncated basis constructs a space which is close under multiplication of any two (so any number) functions.  Therefore, for problems showing sharp gradients or discontinuity, MWB expansion decreases the possibility of accumulating error [13-17]. Hence, Studies show MRA can be a practical approach to the problems lacking sufficient regularity or smoothness [13-19].

\end{enumerate}
In previous studies, $intrusive$ methods such as stochastic Galerkin method, were exploited [13-19] to discretize and perform block-partitioning adaptively. Intrusive methods are based on the solution to the system of governing equations for the spectral coefficients in the PC representation of the solution.

Several previous schemes to adaptively solve stochastic differential equations (SDEs) was constructed based on stochastic Galerkin method [14-19]. In general, stochastic Galerkin method, as an intrusive one, has high computational cost since it involves high dimensional integration to evaluate the coefficients in the approximated solution.

In this paper, we aim to develop the $non$-$intrusive$ $solution$-$adaptive$ uncertainty propagation scheme to extract any local structures in the solution in order to decrease computational cost. Since non-intrusive methods are based on the direct random sampling of the problem solutions, they only require a deterministic solver. In fact, we can do sampling procedure employing any legacy code for the deterministic problem as a black box. In addition, this feature of non-intrusive approach allow us to do sampling process in parallel with deterministic simulations. Since the computational cost of non-intrusive approach can be scaled by the number of deterministic model simulations  performed to construct the approximation, the computational cost of non-intrusive approach can be large when the underlying deterministic scheme has a large computational cost. To overcome this problem, which is called curse of dimensionality and means that the number of model simulations increases exponentially with the number of independent random variables in the problem, decreasing the computational cost of non-intrusive approaches is one of the main subjects of recent researches in uncertainty quantification.

\subsection{Compressive Sensing via $\ell_1$-minimization} 
 
In the present paper, we focus on a special case in which the quantity of interest at stochastic level is $sparse$; it can be accurately expanded to only few terms in terms of spectral MWB representation. It will be shown that when the problem in MWB representation is sparse, the new approach considerably decreases the computational cost.

In the present work, $compressive$ $sampling$ and $sparse$ $recovery$ methodologies are exploited to develop theorems for the sparse recovery of $bounded$ $orthonormal$ $systems$ and then utilizing them, a sparse spectral MWB representation of solution to sparse SDEs is constructed.

Compressive sensing (CS) is an efficient method  to recover sparse signals which is based on two main definitions:  $Recoverability$ and $Stability$. Recoverability describes the fact that a sparse signal of length $m$ can be recovered by far fewer than $m$ measurements via $\ell_1$-minimization or other recovery schemes. Specifically, recoverability answer to the questions of which types of measurement matrices and recovery methods ensure exact recovery of all $k$-sparse signals (those having exactly $k$-nonzero entries) and how many measurements are sufficient to guarantee such a recovery. In other perspective, stability specifies the accuracy and robustness of a recovery scheme to recover sparse signals in noisy measurements and inaccurate or imperfect informations [20-21].

\subsection{Our Contribution} 
This study present an innovative non-intrusive method to solve SDEs applying uniform sparse recovery (which means when sparse recovery conditions are satisfied, the new scheme can recover all sparse signals) of MWB spectral representations via  $\ell_1$-minimization. Present method, the same as all non-intrusive methods, employ deterministic solvers directly. In fact, it is based on compressive sampling via $\ell_1$-minimization as a non-intrusive method to overcome complexity of implementation and decrease computational cost. We also employ the MWB representation the same as previous works [13-17] to decompose random process data to achieve more accurate results when the solution to interest problems have sharp variation or discontinuity. Our main goal in the present method is exploiting CS as a stochastic non-intrusive method to obtain the approximated  values of unknown coefficient with high accuracy in the sparse multi-wavelet PC expansion.

The following characterizes our method:

\begin{enumerate}
\item
{\bf Non-intrusive}: The present method is based on the direct random sampling of the solution to SDEs which means that sampling is done by exploiting any available codes for the deterministic problem as a black box. 

\item
{\bf Adaptive}: The proposed method identifies the importance of dimensions at the stochastic level. The dimension of a problem is the number of stochastic input parameters in it. Additionally, present approach adaptively identifies the local structure of the solution which is the most important feature of the present method.

\item
{\bf Provably Convergent}: We will obtain probabilistic bounds on the approximation error which show the stability and convergence of the method. To find the bounds, a number of theorems and lemmas in sparse recovery of bounded orthonormal systems are expressed and modified in order to develop the required theorems for MWB system.

\item
{\bf Applicable to Problems with Non-smooth Solution}: Present approach can be used in the problems having sharp gradient or  discontinuity, and thus lack sufficient regularity and smoothness. Furthermore, it can be applicable in the problems with high dimensional random inputs. However, high-dimensional problems are not considered in this paper.

\end{enumerate}

This paper is organized as follows: In section $2$, MRA is described and several definitions on spectral PC and multi-wavelet (MW) methods is reviewed to expand the solution to SDEs in terms of spectral MWB representation. In section 3, we introduce our method to solve SDEs via sparse approximation applying the extended version of available theorems and definitions of bounded orthonormal systems for MWB system; then, we derive probabilistic bounds of errors which shows stability and convergence of proposed methods. The problem of interest is formulated in section 4.1. In section 4.2, we describe the strategy for adaptive partitioning. Two SDEs and one stochastic discontinuous function with random inputs are studied in section 5 as the numerical problems lacking sufficient smoothness to analyse the convergence, accuracy and efficiency of the proposed method. In the last section the conclusions and the way to improve the present method is discussed.

\section{MWB Expansion of Stochastic Process}

In this section, we introduce some preliminary definitions in the probability context. First, we construct generalized polynomial chaos (GPC) and combine GPC with MWB to generate required global and local basis exploited to expand the solution to the problems into this frame. Next, the combination of local multi-wavelet and local Legendre basis is employed to construct local desired basis.

\subsection{Mathematical Preliminaries}

Assume a triple $(\Omega,\mathcal{F},\mathcal{P})$ is a complete probability space where $\Omega$ is the sample space, $\mathcal{F}$ is the $\sigma$-algebra of subsets (events) of $\Omega$  and $\mathcal{P}$ is the probability measure on it. For this probability space, real-valued random variable $u$ is defined as a function which maps the probability space to the real line or subset of the real line and is often denoted by $u(\zeta)$, where $\zeta$ denotes association with a probability space and is utilized to emphasize randomness in $u$. Finally, $\mathbf{P}(x)$ denotes the probability that $\mathcal{P}(\zeta\leq x)$. 

\begin{remark}
\
If $\mathbf{P}(x)$ is continuously increasing function of $x$ defined on $[a,b]$ such that $-\infty \leq a < b \leq  \infty $, $\mathbf{P}(a)=0$ and $\mathbf{P}(b)=1$, the probability density function (pdf) $\mathcal{P}$ is defined on $(a,b)$ by:

\begin{equation}
\mathcal{P}=
\begin{cases}
\displaystyle\frac{d\mathbf{P}}{dx}& x \in \mathrm{(a, b)} \\
0 &\mathrm{otherwise}
\end{cases}
\end{equation}

Based on the assumed properties, there is one-to-one mapping

\begin{equation}
\omega \in[0,1] \to x: x=\mathbf{P}^{-1}(\omega)\in [a,b]
\end{equation}

\medskip

Furthermore, if $\omega$ is a uniformly distributed random variable on $[0,1]$, it can be easily shown that $\mathbf{P}^{-1}(\omega)$ is a random variable on $(a,b)$ having the same distribution as $\zeta$.


\end{remark}

\begin{remark}

The space of all second-order random variables is denoted by $L_2(\Omega,\mathcal{F},\mathcal{P})$ and is defined as the Hilbert space constructed by all random variables equipped with the mean square norm:

\begin{equation}
\|u\|_{L_2(\Omega)}=\mathbb{E} (u^2)^{\frac{1}{2}}< \infty
\end{equation}

\end{remark}

The following of this section, GPC and MWB expansions of the second-order random variables are described. Here, we focus on a second-order random process in $L_2$ spaces and aim to express the solution to SDEs in terms of Fourier-like series which are convergent with respect to the norm associated with the corresponding inner product. First, GPC expansion and then MWB expansion is discussed.

\subsection{Generalized Polynomial Chaos (GPC)}

At first, Weiner exploited Hermite polynomials in terms of Gaussian random variables as a basis to form the original PC [2]. Next, Cameron and Martin [24] proved that such a choice of basis, which is applied to define the random process, can lead to an expansion which is convergent in mean-square norm to the original process [9]. Expanding the process in terms of Hermite polynomials, we have:  

\begin{equation}
u(\zeta)=a_0 H_0 + \sum_{i_1=1}^{\infty} a_{i_1} H_{1}(\zeta_{i_1})+ \sum_{i_1=1}^{\infty} \dotsb \sum_{i_n=1}^{i_{n-1}} a_{i_n} H_{n}(\zeta_{i_1}, \cdots, \zeta_{i_n})
\end{equation}

where $H_{n}(\zeta_{i_1}, \cdots, \zeta_{i_n})$ denote the Hermit polynomial of order $n$ in terms of the multi-dimensional independent standard Gaussian random variables $\zeta=(\zeta_{i_1}, \cdots, \zeta_{i_n})$ with zero mean and unit variance. For notation convenience, we rewrite this expansion as:

\begin{equation}
u(\zeta)=\sum_{k=0}^{\infty} u_k \Psi_k(\zeta) 
\end{equation}

where there is a one-to-one relation between $H_{n}(\zeta_{i_1}, \cdots, \zeta_{i_n})$ and $\Psi_n$.

\begin{remark}
The general expression of the Hermite polynomials is given as:

\begin{equation}
 H_{n}(\zeta_{i_1}, \cdots, \zeta_{i_n})=\frac{1}{2}\exp (\zeta \zeta^T)(-1)^n \frac{\partial^n \exp(\frac{-1}{2}\zeta \zeta^T)}{\partial\zeta_{i_1} \dotsb \partial \zeta_{i_n}}
\end{equation}

\end{remark}

In order to consider general random inputs,GPC uses the orthonormal polynomials from the Askey scheme [25] as the basis in $L_2(\Omega)$. Hence, the GPC of $u(\zeta)$ can be expressed as:

\begin{equation}
u(\zeta)=a_0 \Gamma_0 + \sum_{i_1=1}^{\infty} a_{i_1} \Gamma_{1}(\zeta_{i_1})+ \sum_{i_1=1}^{\infty} \dotsb \sum_{i_n=1}^{i_{n-1}} a_{i_n} \Gamma_{n}(\zeta_{i_1}, \cdots, \zeta_{i_n})
\end{equation}

where $\Gamma_{n}(\zeta_{i_1}, \cdots, \zeta_{i_n})$ denotes the Wiener-Askey polynomial chaos of order $n$ in terms of the uncorrelated random vector $\zeta=(\zeta_{i_1}, \cdots, \zeta_{i_n})$ Again, for notation convenience, we rewrite the aforementioned expansion as:

\begin{equation}
u(\zeta)=\sum_{k=0}^{\infty} u_k \Psi_k(\zeta)
\end{equation} 

such that $\Psi_n$ has a one-to-one relationship with an Askey polynomial $\Gamma_{n}(\zeta_{i_1}, \cdots, \zeta_{i_n})$. Now, a truncated form of the expansion after $P$ terms  can be expressed as:

\begin{equation}
u(\zeta)=\sum_{k=0}^{P} u_k \Psi_k(\zeta)
\end{equation}

\begin{remark}
The choice of orthonormal polynomials is determined by the probability distribution of random variables using Askey scheme [25].
\end{remark}

Since GPC is a global basis spectral representation, its convergence could be damaged when the solution to problems show localized sharp gradients or discontinuous variations. In contrast, MW representation leads to localized decomposition, Therefore, it can show highly accurate behavior in such types of problems. This is why we are motivated to employ MWB as an alternative basis when the solution to interest problems lack sufficient smoothness or regularity.

\subsection{Multi-Resolution Analysis}

In this section we construct MRA scheme can be applied as a powerful tool in both intrusive or non-intrusive methods. We start this section by developing PC expansions based on Haar wavelets, and extend it to MWB functions. 

\subsubsection{Multi-Resolution Decomposition of $L_2([0,1])$}

For $n_0={0,1,\cdots}$ and $k={0,1,\cdots}$, define the space $\mathcal{V}^{n_0}_k$ of piecewise polynomial continuous functions such that

\begin{equation}
\mathcal{V}^{n_0}_k \equiv \{f: f|_{(2^{-k}l,2^-{k}(l+1))}\ \in \mathbf{P}_{n_0} \ \mbox{for}  \ 0 \leq l \leq 2^k-1\ \mbox{and} \ f: f|_ {\mathbf{R}\setminus[0,1]}=0\}
\end{equation} 

where $f|_{\mathbf{S}}$ means the restriction of $f$ onto the set $\mathbf{S}$ and $\mathbf{P}_{n_0}$ denotes the space of polynomials of degree at most $n_0$. It can be shown that $\mathbf{Dim}({\mathcal{V}^{n_0}_k})=(n_0+1)2^k $ and $\mathcal{V}^{n_0}_0 \subset \mathcal{V}^{n_0}_1 \dotsb  \subset \mathcal{V}^{n_0}_k \subset \dotsb $. 

Now, according to [14-17] and [26-27], we define

\begin{equation}
\mathcal{V}^{n_0}=\overline{\bigcup_{k\geq 0}}\mathcal{V}^{n_0}_k .
\end{equation}

It is easy to show that $\mathcal{V}^{n_0}$ is dense in $L_2([0,1])$ with respect to the $ L_2 $-norm $\|f\|_2=\mathbb{E}(f^2)^\frac{1}{2}$ where 

\begin{equation}
\mathbb{E}(f g)=\int_0^1 f(\zeta) g(\zeta) d\zeta
\end{equation}

We denote the subspace $\mathcal{W}^{n_0}_k$ as the orthonormal complement of $\mathcal{V}^{n_0}_k$, i.e,

\begin{equation}
\begin{split}
&\mathcal{V}^{n_0}_k \oplus \mathcal{W}^{n_0}_k=\mathcal{V}^{n_0}_{k+1} \\
&\mathcal{W}^{n_0}_k \perp \mathcal{V}^{n_0}_k
\end{split}
\end{equation}

which yields:

\begin{equation}
\mathcal{V}^{n_0}_k \oplus_{k\geq 0} \mathcal{W}^{n_0}_{k}=L_2([0,1]) .  
\end{equation}

\subsubsection{Multi-Wavelet Basis (MWB)}
As mentioned before, GPC as well as methods based on smooth functions cannot often describe the solution behavior of problems having sharp gradient, steep variation, discontinuity or bifurcation with respect to random inputs.
One of the most important features of MWB is its power to model the behavior of such problems, which causes major difficulties GPC is utilized [13-17].

Now, we are ready to construct the orthonormal basis $\{\Psi_i\}_{i=0}^{n_0}$ of  $\mathcal{W}^{n_0}_0$. Here, $\{\Psi_i\}_{i=0}^{n_0}$ are piecewise polynomial functions with degree at most  $n_0$ satisfying the orthonormality condition, so we have [14-17]:

\begin{equation}
\mathbb{E}(\Psi_i\Psi_j)=\delta_{ij}\hspace{16pt} 0 \leq i,j \leq n_0
\end{equation}

Since $\mathcal{W}^{n_0}_{0}$ is the orthonormal complement of $\mathcal{V}^{n_0}_{0}$, the first $n_0+1$ moments of the $\Psi_i$ vanish: 

\begin{equation}
\mathbb{E}(\Psi_i \zeta^j)=0 \hspace{16pt} 0 \leq i,j \leq n_0
\end{equation}

equations (15) and (16) give the system of linear equations whose solution results unknown coefficients in $\{\Psi_i\}^{n_0}_{i=0}$. 

Now, we construct the orthonormal basis for  $\mathcal{W}^{n_0}_{k}$ by utilizing multi-wavelets $\Psi^{k}_{j,l}$, the translated and dilated version of original $\Psi_i$'s. $\Psi^{k}_{j,l}$ can be expressed as:

\begin{equation}
\Psi^{k}_{j,l}=2^{\frac{k}{2}}\Psi_j(2^k\zeta-l)\mathbb{I}([2^{-k}l,2^{-k}(l+1)]) \hspace{16pt} 0 \leq j \leq n_0 \hspace{8pt} \mbox{and} \hspace{8pt} 0 \leq l \leq 2^k-1
\end{equation}

Orthonormality of $\Psi_i$'s results:

\begin{equation}
\mathbb{E}(\Psi^{k}_{j,l}\Psi^{k\prime}_{j\prime,l\prime})=\delta_{kk\prime} \delta_{jj\prime} \delta_{ll\prime}
\end{equation}  

We also construct the basis ${\Phi_i}^{k-1}_{i=0}$ for $\mathcal{V}^{k}_0$. To do this, we employ orthonormal Legendre polynomial on $[-1,1]$ and then rescale it over $[0,1]$. Assume $L_i$ denotes the orthonormal Legendre polynomial of degree $i$ on $[-1,1]$; We define

\begin{equation}
\Phi_i(\zeta)=(\frac{1}{\sqrt{2}})L_i(2\zeta-1)
\end{equation}

so

\begin{equation}
\mathbb{E}(\phi_i\phi_j)=\delta_{ij} \hspace{16pt} 0 \leq i,j \leq n_0
\end{equation}

Now, one can construct the polynomials $\Phi^k_{i,l}$, which are the translated and dilated version the polynomials $\Phi_i$'s as:

\begin{equation}
\Phi^{k}_{j,l}=2^{\frac{k}{2}}\Phi_j(2^k\zeta-l)\mathbb{I}([2^{-k}l,2^{-k}(l+1)]) \hspace{16pt} 0 \leq i \leq n_0 \hspace{8pt} \mbox{and} \hspace{8pt} 0 \leq l \leq 2^k-1
\end{equation}

Hence, the space $\mathcal{V}^{n_0}_{k}$ whose dimension is $2^k(n_0+1)$ is spanned by $\Phi^k_{i,l}$ [14-17] and [26-27].  

\subsubsection{MW Expansion}

Consider the stochastic function $f \in L_2([0,1])$, so $f$ can be approximated by using the basis which is formerly constructed.  

assume $f^{n_0,n_1}$ is the projection of $f$ onto the space $\mathcal{V}^{n_0}_{n_1}$, so 

\begin{equation}
f^{n_0}_{n_1}=\sum _{l=0}^{2^{n_1}-1}\sum_{i=0}^{n_0}\mathbb{E}(\Phi^{n_1}_{i,l}f)\Phi^{n_1}_{i,l}(\zeta)
\end{equation}

Since it can be shown that $\mathcal{V}^{n_0}_{n_1}=\mathcal{V}^{n_0}_{0}\oplus_{1\leq k\leq n1} \mathcal{W}^{n_0}_{k}$ so $f^{n_0,n_1}$ can have another representation with respect to $\Psi^{k}_{i,l}$

\begin{equation}
f^{n_0}_{n_1}=f^{n_0}_{0}+\sum _{k=0}^{{n_1}-1}\sum_{i=0}^{2^k-1}(\sum_{i=0}^{n_0} g^{k}_{i,l} \Psi^{k}_{i,l}(\zeta))
\end{equation} where the 

\begin{equation}
g^{k}_{i,l}=\mathbb{E}((f^{n_0}_{k+1}-f^{n_0}_{k})\Psi^{k}_{i,l})
\end{equation}

Define $\delta^{n_0}_{n_1}\equiv \mathbb{E}(|f-f^{n_0}_{n_1}|^2)$ thus convergence of the expansion can be determined by the decreasing of $\delta^{n_0,n_1}$ with respect to increasing polynomial order $n_0$ ($p$-convergence) and increasing resolution level $n_1$ ($h$-convergence).  

\subsection{Multi-Wavelet Basis Construction}
 Prior to constructing local basis expansion, some vital assumptions to easily implement and regularize our approach are expressed.

\begin{assumption}
\
\begin{enumerate}

\item
Assume $\Omega=[a_1,b_1] \times \dotsb \times [a_n,b_n]$ is the space of random parameters. Since by change of variable $\xi_i=(b_i-a_i)(\zeta_i)+a_i$ we can transform the space  $\Omega_s=[0,1]^n$ to $\Omega$ with out losing generality, we assume $\Omega=\Omega_s=[0,1]^n$ with the random parameter $\zeta=(\zeta_1,\cdots, \zeta_n)$ in this paper. 

\item
In addition, we assume that parameters $\{\zeta_i\}_{i=1}^n$ have the uniform probability distribution; therefore, $\zeta$ have the uniform probability distribution on $\Omega=[0,1]^n$. 

\end{enumerate}

\end{assumption}

\subsubsection{Global Basis Construction}
To construct our desired basis, we first construct the global expansion basis for $\Omega=[0,1]^n$. Define the $\mathbf{I}_{n_0,n}:= \{\iota=(\iota_1,\cdots,\iota_n)\in {\mathbf{N}^n_0}:\ \|\iota\|_1 \leq n_0 \hspace{4pt} ,\hspace{4pt} \|\iota\|_0 \leq n\}$. For $\zeta$ we construct the global projection basis as [14-17]:

\begin{equation}
\mathcal{B}_p\equiv \{\Phi_{\iota \in {\mathbf{I}_{n_0,n}}} (\zeta_1,\cdots, \zeta_n)=  \prod_{j=1}^n \phi_{\iota_j}(\zeta_j)  \}
\end{equation}

and the detail direction basis 

\begin{equation}
\mathcal{B}^j_1\equiv \{\Psi_i(\zeta_j), 0\leq i \leq n_0  \} \hspace{12pt} : \hspace{12pt} 0\leq j \leq n_0
\end{equation}

The complete global expansion basis can be expressed as:

\begin{equation}
\bigcup^n_{j=1} \mathcal{B}^j_1(\Omega) \bigcup \mathcal{B}(\Omega)=\mathcal{B}_p(\Omega) 
\end{equation}

and the multi-dimensional process ${u}(\zeta)$ can be  approximately expressed as:

\begin{equation}
u_{n_0}(\zeta)=\sum_{\iota \in {\mathbf{I}_{n_0,n}}} u_{\iota} \Phi_{\iota}(\zeta)+ \sum_{j=1}^n \sum_{i=0}^{n_0}u_{j,i}\Psi_i(\zeta_j)
\end{equation}

In fact, the global basis $\mathcal{B}(\Omega)$ consists of the union of the rescaled Legendre polynomials basis $\mathcal{B}_p(\Omega)$ and first level detail basis  $\{\mathcal{B}^j_1(\Omega)\}_{j=1}^n.$ 
\medskip

consequently, the global expectation of $u$ can be expressed as:

\begin{equation}
\mathbb{E}(u)=u_{(0,\cdots,0)}
\end{equation}

and its global variance can be approximated as:

\begin{equation}
\sigma^2_{\Omega}(u)=({\hat{\sigma}_{\Omega}})^2+\sum_{j=1}^n (\sigma^j_{\Omega})^2
\end{equation}

where 

\begin{equation}
({\hat{\sigma}_{\Omega}})^2 \equiv\sum_{\iota \in {\mathbf{I}^0_{n_0,n}}} (u_{\iota})^2
\end{equation}

here $\mathbf{I}^0_{n_0,n}:=\mathbf{I}_{n_0,n}\setminus \{(0,\cdots,0)\}$,

and

\begin{equation}
(\sigma^j_{\Omega})^2\equiv  \sum_{j=1}^n \sum_{i=0}^{n_0}(u_{j,i})^2
\end{equation}

\begin{remark}
The local expansion for random parameter $\zeta$ on partition $\Omega_m=[a^m_1,b^m_1]\times \cdots \times  [a^m_n,b^m_n] \subset \Omega$ can be immediately derived by $\Omega_m$ instead of $\Omega$ in the above relations  respectively .
\end{remark}

\subsubsection{Local Basis Construction}
In this section we construct local basis expansion utilizing definitions and notations to localize the general basis. This local basis is employed for adaptive partitioning strategy in next sections.
 
\begin{definition}
Consider $I\in\mathbf{N}$, we say I is index level $k$ and also write $\mathcal{L}(I)=k$ if and only if $(k-1)$ is the largest power of $2$ in $I$ i.e.,

\begin{equation}
2^{k-1}\leq I < 2^k
\end{equation} 
so there exists $1\leq l \leq 2^{k-1} $ that $I=2^{k-1}+l-1$.
\end{definition}

\begin{definition}
Consider $I=2^{k-1}+l-1$ $(1\leq l \leq 2^{k-1} )$; therefore  $\mathcal{L}(I)=k$ according to the above definition. Now, we define the interval $\mathcal{I}(I)$ as:

\begin{equation}
\mathcal{I}(I)=[2^{-(k-1)}(l-1),2^{-(k-1)}l]
\end{equation}

\end{definition}

\begin{definition} 
We can immediately extend the above definitions to a multi-dimensional index. Consider multi-dimensional index $I=(I_1,\cdots,I_n)$ so that 
 
\begin{equation}
 I_j=2^{k_j-1}+l_j-1 \hspace{12pt} 1\leq l \leq 2^{k_j-1} 
\end{equation}

so we have $\mathcal{L}(I_j)=k_j$ and $\mathcal{I}(I_j)=[2^{-(k_j-1)}(l-1),2^{-(k_j-1)}l]$, now we define the index level $I$ as:

\begin{equation}
\mathcal{L}(I)=(\mathcal{L}(I_1),\cdots,\mathcal{L}(I_n))=(k_1,\cdots, k_n)
\end{equation}

and 

\begin{equation}
\mathbf{Vol^I}:=\mathcal{I}(I)=\prod_{j=1}^n \mathcal{I}(I_j)
\end{equation}

\end{definition}
Now, we can form the local expansion basis. Consider multidimensional index $I=(I_1,\cdots, I_n)$,  so we have $\mathcal{I}(I)=\prod_{j=1}^n \mathcal{I}(I_j)$ and $\mathcal{L}(I)=(k_1,\cdots, k_n)$. Define the local expansion basis for $I$ as:

\begin{equation}
\Psi_{j,I}(\zeta)=\prod_{i=1}^n \Psi^{k_i}_{j,l_i}(\zeta)
\end{equation}

and 

\begin{equation}
\Phi_{j,I}(zeta)=\prod_{i=1}^n \Phi^{k_i}_{j,l_i}(\zeta)
\end{equation}

or in general case,
 
\begin{equation}
\Phi_{\iota,I} (\zeta)=\prod_{i=1}^n \prod_{j=1}^n \phi^{k_i}_{\iota_j,l_i}(\zeta_j)  
\end{equation} 

so the local projection basis can be expressed as:

\begin{equation}
\mathcal{B}^I_{p}\equiv \{\Phi_{\iota,I} (\zeta)=\prod_{i=1}^n \prod_{j=1}^n \phi^{k_i}_{\iota_j,l_i}(\zeta_j) \hspace{4pt}: \hspace{4pt}\iota \in {\mathbf{I}_{n_0,n}} \}
\end{equation}

and the detail direction basis can be expressed as:

\begin{equation}
\mathcal{B}^{j,I}_1\equiv \{\Psi_i(\zeta_j), 0\leq i \leq n_0  \} \hspace{12pt} : \hspace{12pt} 0\leq j \leq n_0
\end{equation}

Consequently, the complete local expansion basis can be expressed as:

\begin{equation}
\mathcal{B}^I(\Omega)=\mathcal{B}^I_p(\Omega) \bigcup \bigcup^n_{j=1} \mathcal{B}^{j,I}_1(\Omega)
\end{equation}

and the local basis expansion for multi-dimensional parameter $\zeta$ can be expressed as:

\begin{equation}
u_{n_0,I}(\zeta)=\sum_{\iota \in {\mathbf{I}_{n_0,n}}} u_{\iota,I} \Phi_{\iota,I}(\zeta)+ \sum_{j=1}^n \sum_{i=0}^{n_0}u_{j,i,I}\Psi_{i,I}(\zeta_j)
\end{equation}
and its local mean and variance can be approximated as:

\begin{equation}
\left\{
	\begin{array}{ll}
	\mathbb{E}_I(u)=u_{(0,\cdots,0),I} \\
		\sigma^2_{I}(u)=({\hat{\sigma}_{I}})^2+\sum_{j=1}^n (\sigma^j_{I})^2 \hspace{72pt} 
	\end{array}
\right.
\end{equation}

where 

\begin{equation}
({\hat{\sigma}_{I}})^2 \equiv\sum_{\iota \in {\mathbf{I}^0_{n_0,n}}} (u_{\iota,I})^2
\end{equation}

here $\mathbf{I}^0_{n_0,n}=\mathbf{I}_{n_0,n}\setminus \{(0,\cdots,0)\}$

and

\begin{equation}
(\sigma^j_{I})^2\equiv  \sum_{j=1}^n \sum_{i=0}^{n_0}(u_{j,i,I})^2
\end{equation}

\begin{remark}
The global or local expansion for random vector $u(\zeta)$ exists for random process $u(\Re,\zeta)$.
Indeed, we have 

\begin{equation}
u_{n_0}(\Re,\zeta)=\sum_{\iota \in {\mathbf{I}_{n_0,n}}} u_{\iota}(\Re) \Phi_{\iota}(\zeta)+ \sum_{j=1}^n \sum_{i=0}^{n_0}u_{j,i}(\Re)\Psi_i(\zeta_j)
\end{equation}

and 

\begin{equation}
u_{n_0,I}(\Re,\zeta)=\sum_{\iota \in {\mathbf{I}_{n_0,n}}} u_{\iota,I}(\Re) \Phi_{\iota,I}(\zeta)+ \sum_{j=1}^n \sum_{i=0}^{n_0}u_{j,i,I}(\Re)\Psi_{i,I}(\zeta_j)
\end{equation}

for general and local expansion of random process $u(\Re,\zeta)$, respectively.

\end{remark}

In the next sections, we apply these notations for adaptive partitioning.


\section{Compressive Sampling via $\ell_1$ Minimization}
  A fundamental question in MRA is how to compute unknown coefficients in MWB expansion. Le Maitre  and co-workers [14-17] have evaluated those with stochastic Galerkin projection method as an intrusive method. This approach evaluates accurate (even exact) values for unknown coefficients, however, it involves high-dimensional integration and has to use high dimensional integration methods such as Smolyak sparse grid [27-29], adaptive sparse grid integration [30-32] or other available high-dimensional integration methods. All these methods have high computational cost and time to attain high accuracy. So it seems non-intrusive approach which evaluates coefficients using surrogate methods instead of high-dimensional integration can achieve same or even higher accuracy having very low computational cost and time.

One of the efficient methodology to evaluate unknown coefficients in multi-wavelet basis expansion is compressive sensing which is so called compressive sampling. This part gives a mathematical introduction to compressive sampling focusing on sparse recovery using $\ell_1$-minimization and structured random matrices. We emphasize on developing techniques in order to obtain probabilistic estimates for condition numbers of structured random matrices. Estimates of this type are important in providing conditions that ensure exact or approximate recovery of sparse vectors employing $\ell_1$-minimization [33-35].

In this section we develop the ideas and definitions related to compressive sensing via $\ell_1$-minimization. We employ restricted isometry property (RIP) for a preconditioned version of matrix whose entries are the MW basis evaluated at sample points chosen from our sampling measure to establish main results. 

At first, results are established by proving RIP for preconditioned random Legendre matrix. We then extend the results to a large class of orthonormal polynomials system. Finally, we develop theorems concern stability of proposed compressive sensing scheme regarding sampling measure; as a result, the condition that sample points $\{x_j\}$ should be drawn from the special measure (Chebyshev measure) can be neglected. 

\subsection{Sparsity}

Suppose $x\in\mathbf{R}^m$. In mathematical perspective, we say that $x$ is $s$-sparse if it has $s$ or fewer non-zero coordinates:

\begin{equation}
x \in \mathbf{R}^m, \|x\|_0=\#\{i,x_i\not= 0\}\leq s\ll m 
\end{equation}

One can show that $\|.\|_0$ is a quasi-norm where the usual $p$-norm is defined by:

\begin{equation}
\|x\|_p=(\sum_{i=1}^m |x_i|^p)^{1/p}
: (1\leq p< \infty) 
\end{equation}

and $\|.\|_{\infty}=\max{|x_i|}$. One can also show 

\begin{equation}
\|x\|_0=lim_{p \to 0}{\|x\|_p}^p
\end{equation}

It can be shown from the theories of convex optimization that if $J(x)$ is convex function the following equation has a unique well-defined solution [20, 36-40].

\begin{equation}
\mathbf{P}_J{}:\min_x {J(x)}\hspace{8pt} \mbox{Subject to}\hspace{8pt} Ax=b
\end{equation} 

When we substitute $J(x)$ with $\|x\|_p$, since $\ell_p$-norms $(1\leq p< \infty)$ are convex, if we define 

\begin{equation}
\mathbf{P}_p{}:\min_x {\|x\|_p}\hspace{8pt} \mbox{Subject to}\hspace{8pt} Ax=b \hspace{8pt} (1\leq p< \infty)
\end{equation}

$\mathbf{P}_p$ consequently has a unique solution.

\bigskip

Now, we attempt to find the sparsest solution, i.e., the solution with the fewest nonzero entries. Indeed, if we define $\mathbf{P}_0$ by substituting $J(x)=\|x\|_0$ in (53), we seek the solution to $\mathbf{P}_0$.
Although it seems there are many similarities between $\mathbf{P}_0$ and $\mathbf{P}_p$, however, they are completely different. The quasi-norm $\|.\|_0$ is a highly discontinuous function and its behavior is very complex. Hence, there is no guaranty for uniqueness and global optimality that can be achieved from convex optimization for the solution to $\mathbf{P}_p$.
There are some fundamental questions which help us to analyse $\mathbf{P}_0$. 

\begin{enumerate}
\item
When is the solution to $\mathbf{P}_0$ unique?
\item
How can one obtain the good approximation from this unique solution?
\item
What is the accuracy and efficiency of these approximated solutions?
\end{enumerate}

We address the questions respectively in this section.

\subsection{Pursuit Algorithms}

Since the direct approach to solve $\mathbf{P}_0$ is numerically intractable, we need the alternative approaches to solve it. There are two all-important algorithms to approximate solution to $\mathbf{P}_0$ or find its exact solution under specific conditions. The first applies relaxation methods and relies on an optimization problem that can be solved using convex programming and convex optimization algorithms, and the second approach employs fast greedy algorithms. Here, we just discuss "Basis Pursuit", a method that utilizes a convex optimization to solve $\mathbf{P}_0$ and leave the greedy algorithms.

\subsubsection{Basis Pursuit}

Donoho and his collaborators have shown in [39-40] that for the specific matrices $A$, $\mathbf{P}_0$ is equivalent with its $\mathbf{P}_1$ relaxation. The problem $\mathbf{P}_1$ is a linear programming (LP) problem and can be solved employing the interior-point method, simplex method, or other techniques.

\begin{definition}[\textbf{Ristricted Isometry Property (RIP)}]

We say the matrix $A$ has the restricted isometry property (s, $\delta$) if for all $s$-sparse vector $x$ we have 

\begin{equation}
(1-\delta)\|x\|_2\leq \|Ax\|_2\leq (1+\delta)\|x\|_2
\end{equation}

We define the restricted isometry constant  $\delta_s$ as the smallest value for $\delta$ so that (55) holds for all $s$-sparse matrices. When $\delta_s$ is small in (55), it means that every subset of $s$ or fewer columns of $A$ is approximately orthonormal. Indeed, we have $\delta_s=0$ when $A$ is an orthonormal matrix.
\end{definition}

Cand\`{e}s and Tao have shown in [41-42] that for measurement matrices satisfying the restricted isometry property, basis pursuit recovers all sparse signals exactly. It is expressed as following theorem. 

\begin{theorem}[\textbf{Sparse Recovery Under RIP [41]}]
Suppose the measurement matrix $A\in \mathbf{R}^{n\times m}$ satisfy the RIP(3s, 0.2). Then every s-sparse vector $x\in \mathbf{R}^m$ can be exactly recovered from its measurements $Ax=b$ as a unique solution to the linear optimization $\mathbf{P_1}$.  
\end{theorem}

\begin{remark}
In [42] Cand\`{e}s has shown that RIP(3s, 0.2) can be weakened with RIP(2s, $\sqrt{2}-1$).   
\end{remark}

\begin{remark}
This guarantee is uniform which means for any arbitrary measurement matrix $A$ satisfying conditions of Theorem 3.2, Basis pursuit recovers all sparse signals.     
\end{remark}

So we find a way to be sure that in specific cases, we have uniqueness for $\mathbf{P}_0$. Now we should answer to the next question. Since $\mathbf{P}_0$ is combinatorial optimization and numerically impractical because  $\mathbf{P}_0$ is NP-Hard in general, so the direct approach to solve $\mathbf{P}_0$ seems impossible. Hence,We introduce methods which can approximate the solution to  $\mathbf{P}_0$ under specific conditions.

\begin{remark}
The signals and measurements are noisy in practice. It is clear that in noisy cases we cannot employ $\mathbf{P_1}$. Cand\`{e}s, Tau and Romberg in [43] developed a new version of Basis Pursuit recovering noisy sparse signals. We express their results in the following theorem [42-43, 45-46]. 
\end{remark}

\begin{theorem}[\textbf{Sparse Recovery for RIP Matrices}]
Suppose the restricted isometry constants $\delta_{2s}$ of a matrix $A\in \mathbf{R}^{n\times m}$ satisfies $\delta_{2s}\leq {\frac{3}{4+\sqrt{6}}}$. 
Let $x\in \mathbf{R}^m$ and assume noisy measurement matrix $Y=Ax+\eta$ is given with $\|\eta\|_2\leq \epsilon$. Suppose $x^*$ be the solution to 

\begin{equation}
\min_x {\|x\|_1}\hspace{8pt} \mbox{Subject to}\hspace{8pt} \|Az-b\|_2\leq \epsilon \hspace{8pt}
\end{equation}

then

\begin{equation}
\|x-x^*\|_2 \leq C_1\frac{\sigma_s(x)_1}{\sqrt{s}}+ D_1\epsilon
\end{equation}

and

\begin{equation}
\|x-x^*\|_1 \leq C_2 \sigma_s(x)_1+ D_2\sqrt{s}\epsilon
\end{equation}

Where $\sigma_s(x)_p=\inf_{\{z:\|z\|_0\leq s\}}\|x-z\|_p$
and constants $C_1, C_2, D_1, D_2 >0$ depend only on $\delta_{2s}$.
\end{theorem}

\subsubsection{Bounded Orthonormal System}
Suppose $(\mathbf{R}^d,D,\nu)$ is a complete probability space and $\{\Psi_i\}_1^n $ is an orthonormal system of real function on $D$, i.e.,

\begin{equation}
\int_D \Psi_i(x) \Psi_j(x) \ dx=\delta_{j, k} \hspace{12pt} k, j\in \{1,2,\cdots,m\}
\end{equation}

We call system $\{\Psi_i\}_{i=1}^n $ a $Bounded$  $Orthonormal$   $System$ if it is uniformly bounded, i.e, 

\begin{equation}
\sup_{1\leq i\leq m}\|\Psi_i\|_{\infty}=\sup_{1\leq i\leq m}\sup_{x\in D}|\Psi_i(x)\ \leq K 
\end{equation}

for some $K\geq 1$

\begin{theorem}[\textbf{RIP for Bounded Orthonormal Systems [33-34]}]
Assume $(\mathbf{R}^d,D,\nu)$ is a complete probability space and $\{\Psi_i\}_{i=1}^m $ is an orthonormal bounded system with uniform bound $K\geq 1$ on $D$. Suppose the matrix $\Psi \in \mathbf{R}^{n\times m}$ with entries 

\begin{equation}
\Psi_{j, k}=\Psi_k(x_j)\hspace{16pt} k\in\{1,2,\cdots,m\} \hspace{8pt} and \hspace{8pt} j\in\{1,2,\cdots,n\}
\end{equation}

is formed by \emph{i.i.d} samples $\{x_j\}$ drawn from the measure $\nu$. If 

\begin{equation}
n\geq \frac{C}{\delta^2} (K^2 s\ln^3s\ln m)
\end{equation}

Then, with Probability at least $1-m^{-\gamma \ln^3s}$, the restricted isometry constant $\delta_s$ of $\frac{1}{\sqrt{n}}\Psi$ satisfies $\delta_s\leq \delta$ where the constant $C$ and $\gamma$ are global. We can imply (62) by the simpler lower bound 

\begin{equation}
n\geq \frac{C}{\delta^2} (K^2 s\ln^4m)
\end{equation} 

\end{theorem}

Now we can develop theory concerning sparse recovery of orthonormal bounded system. First, we express the theorem which is beneficial to find $k$ in the boundedness condition (60).

\subsubsection{Uniform Bound for Orthonormal Bounded System}
We should find upper uniform bound $K$ to evaluate the minimum required samples for sparse recovery employing $\ell_1$-minimization. Since smaller $k$ results in smaller required sample for sparse recovery, we should find the smallest upper bound for $K$ in the boundedness condition (60). 
\bigskip

Following theorem from Szeg\"{o} [47] evaluate the smallest value for specific classes of orthogonal (or orthonormal) polynomials.

\begin{theorem}
Suppose $\nu(x)$ is a non-decreasing weight function on $[a,b]$, $b$ is finite, and $\{p_k\}$ is the corresponding orthonormal polynomials. Assume $P_k(x)=\nu(x)^{\frac{1}{2}} |p_k(x)|$, then $P_k$ attain its maximum value on $[a,b]$ for $x=b$
\end{theorem}

\begin{remark}
Theorem 3.5 can be extended for any subinterval $[x_0,b]$ of $[a,b]$ if $w(x)$ is non-decreasing. 
\end{remark}

\begin{corollary}
Consider general probability measure $d\zeta(x)=\xi(x) dx$ 
on [-1,1] and corresponding orthonormal polynomials $\{p_k\}$. Suppose $P_k(x)=\xi(x)^{\frac{1}{2}} |p_k(x)|$, then $P_k$ attain its maximum value on $[-1,1]$ for $x=1$.
\end{corollary}

\begin{remark}
The $L_{\infty}$-norm of orthonormal Legendre polynomial $L_k(x)$ is $\|L_k\|_{\infty}=|L_k(1)|=|L_k(-1)|=(k+\frac{1}{2})^{\frac{1}{2}}$, hence

\begin{equation}
\sup_{1\leq i\leq n}\|L_i\|_{\infty}=\sup_{1\leq i\leq n}\sup_{x\in D}|L_i(x)| \leq (n+\frac{1}{2})^{\frac{1}{2}} 
\end{equation}

\end{remark}

\begin{theorem}[\textbf{Recovery of Sparse Legendre Basis [33-35]}]
Consider the Legendre matrix $\Phi_{n\times m}$ with entries $\Phi_{j , k}=L_{k-1}(x_j)$ where $L_i(x)$ is the $i^{th}$ orthonormal Legendre polynomial in $[-1,1]$ and $n$ sampling points $\{x_j\}$ are drawn independently at random from the Chebyshev measure. Furthermore, consider the diagonal pre-conditioner matrix $W$ whose entries are $W(j,j)={(\frac{\pi}{2})}^\frac{1}{2}()1-x_j^2)^\frac{1}{4}$, and assume

\begin{equation}
n\geq C s\ln^3 s\ln{m}
\end{equation} 

Then with probability at least $1-m^{-\gamma\ln^3 s}$, the following result can be established for all polynomial $f(x)=\sum_{k=0}^{m-1}{c_k L_k(x)}$.

\bigskip
Assume noisy sample values $Y=\Phi c+\eta=(f(x_1)+\eta_1,\cdots,f(x_n)+\eta_n)$ are observed and $\|W\eta\|_{\infty}\leq \epsilon$ Then the best $s$-term or $s$-sparse approximation $c^*$ of the coefficient vector $c=(c_0,c_1,\cdots,c_{m-1})$ can be recovered by the following $\ell_1$-minimization problem:

\begin{equation}
\min_c {\|c\|_1}\hspace{8pt} \mbox{Subject to}\hspace{8pt} \|W\Phi c-WY\|_2\leq \sqrt{n}\epsilon \hspace{8pt}
\end{equation}

such that

\begin{equation}
\|c-c^*\|_2 \leq C_1\frac{\sigma_s(c)_1}{\sqrt{s}}+ D_1\epsilon
\end{equation}

and
\begin{equation}
\|c-c^*\|_1 \leq C_2\sigma_s(c)_1+ D_2\sqrt{s}\epsilon
\end{equation}

The constants $C_1, C_2, D_1, D_2 >0$ depend only on $\delta_{2s}$.

\end{theorem}

\begin{remark}
\

\begin{enumerate}
\item
In the absence of noise $(\epsilon=0)$ and if c is exactly $s$-sparse$(\sigma_s(x)_1=0)$ the exact recovery can be attained via

\begin{equation}
\min_c {\|c\|_1}\hspace{8pt} \mbox{Subject to}\hspace{8pt} \Phi c=Y
\end{equation}

\item
It can be shown that $ \|W\eta\|_{\infty}\leq \epsilon $ if $ \|\eta\|_{\infty}\leq \epsilon $.
\end{enumerate}
\end{remark}

\begin{proof}
To prove theorem 3.12 we employ following Lemma.

\begin{lemma}[\textbf{Szeg\"{o} [34-35, 47, ]}]
***For orthonormal Legendre polynomials $\{L_k(x)\}_{k\geq 1}$ we have
$(1-x^2)^{\frac{1}{4}} |L_n(x)|<{(\frac{2e}{\pi})}^{\frac{1}{2}}(\frac{2n+1}{2n+2})^{\frac{1}{2}}<2\pi^\frac{-1}{2} $ for all $n$ and $x\in[-1,1]$. Moreover, the constant $2\pi^{\frac{-1}{2}}$ cannot be substituted with a smaller value in the inequality.
\end{lemma}

Now, consider the functions $Q_k(x)=(\frac{\pi}{2})^{\frac{1}{2}}(1-x^2)^{\frac{1}{4}} (L_k(x))$ matrix $\Psi$ whose entries are $\Psi_{j,k}=Q_{k-1}(x_j) $, we can readily result that $\Psi=W\Phi$. So we have:

\begin{enumerate}
\item
Since $\int_{-1}^1 (\pi)^{-1} Q_k(x) Q_j(x) (1-x^2)^{\frac{-1}{2}} dx=0 $, the system $\{Q_k\}$ is an orthonormal system.

\item
Applying Lemma 3.14, the orthonormal system $\{Q_k\}$ is bounded since $\|Q_k\|_{\infty} \leq \sqrt{2}$.

\end{enumerate}

Hence, the system $\{ Q_k\}$ is orthonormal bounded system with uniform bound $K=\sqrt{2}$. Regarding  Theorem 3.7 when
$n\geq \frac{C}{\delta^2} (K^2 s\ln^3s\ln m)$ the matrix $\frac{1}{\sqrt{n}}\Psi=\frac{1}{\sqrt{n}}W\Phi$ has the restricted isometry property. Applying Theorem 3.6 for noisy sample measurements $\frac{1}{\sqrt{n}}WY$ where $Y=(f(x_1)+\eta_1,\cdots,f(x_n)+\eta_n)$, and considering the fact that $\frac{1}{\sqrt{n}} \|W\eta\|_2 \leq \|W\eta\|_{\infty}\leq \epsilon$, Theorem 3.12 is established.  

\end{proof}

\begin{remark}
for $n$-dimensional orthonormal Legendre polynomials one can show $K=2^{\frac{n}{2}}$ and apply Theorem 3.7 with $K=2^{\frac{n}{2}}$ to obtain the theorem the same as Theorem 3.12 for $n$-dimensional Legendre polynomials.
\end{remark}

\begin{proof}
Assume $\Phi_{\iota}=\prod_{j=1}^n \Phi_{\iota_j}$ and define $Q_{\iota}=\prod_{j=1}^n Q_{\iota_j}$, so $\{Q_{\iota}\}$  are orthonormal polynomials. In addition, according to Theorem 3.8, one can show that $\
\|Q_{\iota}\|_{\infty}=\prod_{j=1}^n \|Q_{\iota_j}\|_{\infty}$ so $K=\|Q_{\iota}\|_{\infty}=2^{\frac{n}{2}}$
\end{proof}

\begin{remark}[\textbf{Generalization of Remark 3.15}]
Remark 3.15 can be generalized for the systems of arbitrary $n$-dimensional orthonormal polynomials. 
\end{remark}

\begin{remark}
Theorem 3.12 can be extended for arbitrary weight function $\vartheta$ on $[-1,1]$ and the orthonormal polynomial $\{p_k\}$ with respect to $\vartheta.$ Indeed, we can generalize Theorem 3.12 provided we have mild continuity condition on $\vartheta$, and therefore, the sparse recovery result of Theorem 3.12 can be extended to more general cases, while in all cases, the sampling points are drawn from the Chebyshev measure.
\end{remark}

The following lemma is the main idea to generalize Theorem 3.12.

\begin{lemma}
Assume $\vartheta$ is a weight function on $[-1,1]$ and define $g_\vartheta (\alpha)=\vartheta(\cos \alpha)|\sin\alpha|.$ In addition, assume that $g_\vartheta$ satisfies \emph{Lipschitz-Dini} condition, i.e., 

\begin{equation}
|g_\vartheta(\alpha+\epsilon)-g_\vartheta(\alpha)|\leq M |\ln(\epsilon)|^{-(1+\lambda)} \hspace{8pt} \forall \alpha \in[0,2\pi) , \epsilon >0 
\end{equation}

for some constants $M , \lambda >0.$ Suppose $\{p_k ,k\in{\mathbf{N}_0\}}$ where $\mathbf{N}_0=\mathbf{N}\cup\{0\}$ be the respective orthonormal polynomial system. In this condition, there is a positive constant $K_\vartheta$ which depends only on the weight function $\vartheta$ and 

\begin{equation}
(1-x^2)^{\frac{1}{4}} \vartheta(x)^{\frac{1}{2}} |p_j(x)| \leq K_{\vartheta} \hspace{8pt} \forall j\in \mathbf{N},x\in [-1,1]
\end{equation}

\end{lemma}

\begin{remark}
The Jacobi Polynomials $p_n^{(\alpha , \beta)}$, which are orthonormal with respect to the weight function $\vartheta(x)=(1-x)^\alpha(1+x)^\beta $, satisfy the \emph{Lipschits-Dini} condition for $\alpha ,\beta \geq \frac{-1}{2}$. In special cases, one can attain the Legendre polynomials with $\alpha=\beta=0$ and Chebyshev polynomials with $\alpha=\beta=\frac{-1}{2}$ [34].
\end{remark}

We entitle the polynomials satisfying Lemma 3.18 \emph{Lipschits-Dini} orthonormal system. So for \emph{Lipschits-Dini} orthonormal system, , Theorem 3.15 holds which expresses RIP-approximate for \emph{Lipschits-Dini} orthonormal system.

\begin{theorem}[\textbf{RIP for \emph{Lipschits-Dini} Orthonormal Systems[34]}]
Consider a positive weight function $\vartheta$ on $[-1,1]$ obeying Lemma 3.18, and assume the \emph{Lipschits-Dini} orthonormal system $\{P_k\}$ with respect to probability measure $d\upsilon(x)=c_\vartheta \vartheta(x)dx$ on $[-1,1]$ where $ c_\vartheta =(\int_1^1 \vartheta(x)\,dx)^{-1} $. 

\bigskip
Now, consider the diagonal pre-conditioner matrix $W$ whose entries are $W(j,j)=(c_\vartheta\pi)^{\frac{1}{2}}(1-x_j^2)^{\frac{1}{4}}\vartheta(x_j)^{\frac{1}{2}}$ and assume $\Phi$ is the matrix whose entries are $\Phi(j,k)=p_{k-1}(x_j) $. Assume that $n$ sampling points $\{x_j\}$ are chosen independently at random from the Chebyshev measure, and consider the $\Psi_{n\times m}=W\Phi$. Finally, assume

\begin{equation}
n\geq \frac{C}{\delta^2} (K^2 s\ln^3s\ln m)
\end{equation}

Then, with probability exceeding $1-m^{-\gamma \ln^3s}$the restricted isometry constant of the matrix $\frac{1}{\sqrt{n}}\Psi$ satisfies $\delta_s\leq \delta$. The constant $C$ depends only on weight function $\vartheta$ and constant $\gamma$ is global. 

\end{theorem}

\begin{proof}
Define $Q_k(x)=(c_\vartheta\pi)^{\frac{1}{2}}(1-x_j^2)^{\frac{1}{4}}\vartheta(x_j)^{\frac{1}{2}}p_k(x)\  \mbox{and}\ \Psi_{j,k}=Q_{k-1}(x_j)$. According to Lemma 1.16 the system $\{Q_k(x)\}$ is uniformly bounded on $[-1,1]$ and $\|Q_k\|_{\infty}\leq (c_\vartheta\pi)^{\frac{1}{2}}K_\vartheta$. 

\bigskip
Furthermore, $\int_{-1}^1 (\pi)^{-1} Q_k(x) Q_j(x) (1-x^2)^{\frac{-1}{2}} dx$ $=\int_{-1}^1 c_\vartheta p_k(x) p_j(x) \vartheta dx=0$. Hence, the system $\{Q_k\}$ is orthonormal bounded system with corresponding matrix $\Psi$ in Theorem 3.20 which is constructed from samples $\{x_j\}$
chosen from the Chebyshev measure. Theorem 3.7 expresses that matrix $\frac{1}{\sqrt{n}}\Psi$ has the restricted isometry property as explained.

\end{proof}

\begin{corollary} 
Consider the \emph{Lipschits-Dini} orthonormal system $\{P_k\}$ associated to the measure $\vartheta$. Also, consider the matrix $\Psi=W\Phi$ as defined in theorem 3.20 and assume all conditions of Theorem 3.20 still hold.
Then with probability at least $1-m^{-\gamma\ln^3 s}$ the following result can be established for all polynomials $f(x)=\sum_{k=0}^{m-1}{c_k p_k(x)}$. 

\bigskip
Assume noisy sample values $Y=\Phi c+\eta=(f(x_1)+\eta_1,\cdots,f(x_n)+\eta_n)$ are observed and $\|W\eta\|_{\infty}\leq \epsilon$. Then the best $s$-term or $s$-sparse approximation $c^*$ of the coefficient vector $c=(c_0,c_1,\cdots,c_{m-1})$ can be recovered by following $\ell_1$-minimization problem:

\begin{equation}
\min_c {\|c\|_1}\hspace{8pt} \mbox{Subject to}\hspace{8pt} \|W\Phi c-WY\|_2\leq \sqrt{n} \epsilon \hspace{8pt}
\end{equation}

such that

\begin{equation}
\|c-c^*\|_2 \leq C_1\frac{\sigma_s(c)_1}{\sqrt{s}}+ D_1\epsilon
\end{equation}

and

\begin{equation}
\|c-c^*\|_1 \leq C_2 \sigma_s(c)_1+ D_2\sqrt{s}\epsilon
\end{equation}

The constants $C_1, C_2, D_1, D_2, \gamma>0$ are global [34].
\end{corollary}
\subsection{ Stability Regarding Sampling Measure}
One can show that the condition that sample points $\{x_j\}$ should be drawn from the Chebyshev measure can be weakened. The following theorem expresses the stability concerning sampling measure.

\begin{theorem} [\textbf{Stability about Sampling Measure [34]}]
We can suppose sampling point $\{x_j\}$ are chosen from a more general probability measure $d\zeta(x)=\xi(x) dx$ on $[-1,1]$ ($\int_{-1}^1 \xi(x)\, dx=1$) so that $\xi(x)\geq C (1-x^2)^{\frac{-1}{2}}$, instead of Chebyshev measure. Now, consider weight function and associated \emph{Lipschits-Dini} orthonormal system $\{P_k\}$. Again, all previous theorems are still confirmed. However, one can show that the smallest value for $K$ in the boundedness condition (60) can be earned by Chebyshev measure and Chebyshev polynomials. 
\end{theorem}

\begin{proof}
considering Lemma 3.18 the functions $Q_k(x)=(c_\vartheta\pi)^ {\frac{1}{2}}\xi(x))^{\frac{-1}{2}}\vartheta(x)^{\frac{1}{2}} p_k(x)$ form a bounded orthonormal system with respect to the probability measure $\xi(x)dx$. Hence, all former theorems are still true. 
\end{proof}

\subsection{Sparse Recovery of Local and Global MWB  Expansion via Compressive Sensing}
According to equation (28), consider the local MWB expansion 

\begin{equation}
u_I(\mathbf{x})=u_{n_0,I}(\mathbf{x})+\epsilon_{P}
\end{equation}

\begin{equation*}
u_{n_0,I}(\mathbf{x})=\sum_{\iota \in {\mathbf{I}_{n_0,n}}} u_{\iota,I} \Phi_{\iota,I}(\mathbf{x})+ \sum_{j=1}^n \sum_{i=0}^{n_0}u_{j,i,I}\Psi_{i,I}(x_j)
\end{equation*}

\begin{equation}
=\sum_{\iota \in {\mathbf{\Lambda}_{n_0,n}}} u_{\iota,I} \widehat{\Psi}_{\iota,I}(\mathbf{x})
\end{equation}

where $P$ is cardinality of $\mathbf{\Lambda}_{n_0,n}$ which can be expressed as:

\begin{equation}
 P=\mathbf{Card}(\mathbf{\Lambda}_{n_0,n})=\mathbf{Card}(\mathbf{I}_{n_0,n})+n(n_0+1)= \frac{(n+n_0)!}{n!(n_0)!}
+n(n_0+1)
\end{equation} 

We develop the methodology to evaluate the unknown  coefficients $\{u_{\iota,I}\}$ utilizing compressive sampling. We define another form of important definition to compute the unknown coefficients.

\begin{definition}[\textbf{Sparsity of Basis Expansion [48]}]
The local expansion $u_{n_0}(\mathbf{x})$, and so $u_{n_0,I}(\mathbf{x})$ for every $I$, is called (nearly) sparse if only a small fraction of coefficients $\{u_{\iota,I}\}$ are dominant and have the contribution in the statistics of $u_{n_0,I}$. 
\end{definition}

If $u_{n_0}(\mathbf{x})$, and so $u_{n_0,I}(\mathbf{x})$ for every $I$, is sparse the coefficient of  $\{u_{\iota,I}\}$ can be evaluated via compressive sampling using $n_s\ll P$ random sampling  $\{u_{n_0,I}(\mathbf{x}_i)\}_{i=1}^{n_s}$.

Assume $\{\mathbf{x}_i\}_{i=1}^{n_s}$ are i.i.d samples of the $n$-dimensional cube $[0,1]^n$ and  $\{u_{n_0,I}(\mathbf{x}_i)\}_{i=1}^{n_s}$ are the corresponding realizations of stochastic function $u_{n_0}(\mathbf{x})$.
Note that according to Theorems 3.20, the i.i.d samples  $\{\mathbf{x}_i\}_{i=1}^{n_s}$ can be chosen from an arbitrary sampling measure on $[0,1]^n$, and independent from probability distribution of $\mathbf{x}$. Consider the matrix ${\mathbf{\Psi}}^I\in \mathbf{R}^{n\times P}$, whose entries are:

\begin{equation}
{\mathbf{\Psi}}^I_{i,j}=\widehat{\Psi}_{j,I}(\mathbf{x}_i) \hspace{8pt} 1\leq i \leq n \hspace{4pt} , \hspace{4pt} 1\leq j \leq P
\end{equation}

and 

\begin{equation*}
\mathbf{U}_I=({u_I}(\mathbf{x}_1),\cdots,{u}_I(\mathbf{x}_n))^T
\end{equation*} 

\begin{equation}
\mathbf{u}_{n_0,I}=({u}_{1,I},\cdots,{u}_{P,I})^T
\end{equation}

\begin{equation}
\mathbf{\mathcal{E}}_P=(\epsilon_p(\mathbf{x}_1),\cdots,\epsilon_p(\mathbf{x}_n))^T
\end{equation}

Now, we can express equation(76) as:

\begin{equation}
\mathbf{U}_I={\mathbf{\Psi}}^I\mathbf{u}_{n_0,I}+\mathbf{\mathcal{E}}_P
\end{equation}

According to Definition (3.23), if $\{\widehat{\Psi}_{\iota,I}\}$ is an orthonormal bounded system, Theorem 3.12 and Theorem 3.20 express that we can recover  the coefficient vector $\mathbf{u}_{n_0}$ from equation(82). We prove it in the next section.

\subsubsection{Uniform Bound for Multi-Wavelet Basis (MWB)}

According to equation (28) the multi-dimensional process ${u}(\zeta)$, can be approximately expressed as:

\begin{equation}
u_{n_0}(\zeta)=\sum_{\iota \in {\mathbf{I}_{n_0,n}}} u_{\iota} \Phi_{\iota}(\zeta)+ \sum_{j=1}^n \sum_{i=0}^{n_0}u_{j,i}\Psi_i(\zeta_j)
\end{equation}

where $\{\Phi_{\iota}\}_{i
\in{\iota}}$ are rescaled Legendre polynomials and $\{\Psi_i\}_{i=0}^{n_0}$ are first level detail of MWB. Based on our construction, all $\{\Phi_{\iota}\}_{i
\in{\iota}}$ and $\{\Psi_i\}_{i=0}^{n_0}$ except $\Psi_0$ are continuous ($\Psi_0$ is discontinuous at $ \zeta=\frac{1}{2} $); hence, employing Theorem 3.8, we conclude that if

\begin{equation}
K_{\Phi}=\sup_{i\in \iota}\|\Phi_i\|_{\infty}=\sup_{i\in\iota}\sup_{\zeta\in [0,1]}|\Phi_i(\zeta)| 
\end{equation}

and

\begin{equation}
K_{\Psi}=\sup_{0\leq i\leq n_0}\|\Psi_i\|_{\infty}=\sup_{1\leq i\leq n}\sup_{\zeta\in [0,1]}|\Psi_i(\zeta)|
\end{equation}

we will have $K=\max\{K_{\phi},K_{\Psi}\}$ . Based on Remark 3.15 we know $K_{\Phi} < 2^{\frac{n}{2}}$ and it can be shown (by Theorem 3.8) that 

\begin{equation}
K_{\Psi} =\max\{ \sup_{1\leq i\leq n}\{|\Psi_i(0)|,|\Psi_i(1)|\}, \Psi_0(\frac{1}{2})\}
\end{equation}
 
Proceeding a few computational steps, we result that
 
 \begin{equation}
 K_{\Psi}=\Psi_0(\frac{1}{2})=n
 \end{equation}

and consequently, the upper bound for MWB is $K=\max\{K_{\phi},K_{\Psi}\}<2^{\frac{n}{2}}$. Therefore, the appropriate upper bound for MWB should be $K=2^{\frac{n}{2}}$ in Theorem 3.20 and Corollary 3.21.

\section{Adaptive Sparse Approximation of Solution to SDEs}

\subsection{Problem Formulation}

Assume $(\Omega,\mathcal{F},\mathcal{P})$ is a complete probability space where $\Omega$ is the sample space, $\mathcal{F}\subset 2^{\Omega}$ is the $\sigma$-algebra of subsets (events) of $\Omega$,  and $\mathcal{P}$ is the probability measure on it.

\bigskip
Define $\mathcal{D}$ as a subset of $\mathbf{R}^d$ with boundary $\partial\mathcal{D}$. Consider $\mathcal{O} $ and $\mathcal{H} $ as operators on $\mathcal{D}$ and  $\partial{\mathcal{D}}$ respectively and $\mathcal{O}$ may depend on $\mathbf{x}\in\Omega$. We seek the solution to following general problem: find $u:[0,T]\times\overline{\mathcal{D}}\times\Omega\to \mathbf{R}$ so that the following equation $\mathcal{P}$-almost sure holds in $\Omega$:
\begin{equation}
\left\{
	\begin{array}{ll}
		\mathcal{O}(t,\mathbf{y},\mathbf{x};u)=f(t,\mathbf{y},\mathbf{x}) \hspace{8pt} (t,\mathbf{y})\in [o,T]\times \mathcal{D} \\
		\mathcal{H}(t,\mathbf{y},\mathbf{x};u)=g(t,\mathbf{y},\mathbf{x}) \hspace{8pt} (t,\mathbf{y})\in [o,T]\times \partial\mathcal{D}
	\end{array}
\right.
\end{equation}

\begin{assumption}
In order to apply the compressive sampling method and achieve the sparse recovery of approximated solution $u_{n_0}$ of $u$, $u_{n_0}$ should be sparse based on Definition (3.23).
\end{assumption}

Define $u_I(\mathbf{x})=\left. u(t,\mathbf{y},\mathbf{x})\right|_{(\mathbf{y}=\mathbf{y}_0 , \mathbf{t}=\mathbf{t}_0,I\subset\Omega)}$ and suppose $\{u_I(\mathbf{x}_i )\}_{i=1}^{n_s}$ be the solutions of equation (110) (when $t=t_0 \ \mbox{and} \ \mathbf{y}=\mathbf{y}_0$) for sampling points $\{\mathbf{x}_i\}_{i=1}^{n_s}$. We can solve equation (82) to extract the coefficients $\{u_{j,n_0}\}_{j=1}^P$ and approximated solution $u_{n_0,I}(\mathbf{x})$ of $u_I(\mathbf{x})$.

\bigskip
Now, we develop the adaptive strategy to obtain more $h$-refinement since $p$-refinement can be  easily achieved by increasing the order of polynomial in MWB expansion. To obtain more $h$-refinement, we should increase the multi-resolution level.

\subsection{Adaptive Partitioning of Random Parameter Space}

Adaptive methods help to reduce the CPU and time cost by improving the quality of the representation locally where needed. In this section, we introduce a local refinement method based on local MWB expansion.(equations (28) and (44)) and determination of a local resolution level.

\subsubsection{Adaptive Strategy}

Suppose that the current partition $\Omega$ involves $N_{sp}$ sub-partitions. On each sub-partitions $\Omega_i$ the random process is expressed with the local basis $\mathcal{B}(\Omega_i)$ where the coefficient computed through compressive sampling methods as in the previous sections. To decide whether the given sub-partition $\Omega_i$ needs more refinement and specify which stochastic directions need this refinement, we use the following criteria [14-17]:

\begin{equation}
\frac{(\sigma^j_{\Omega_i})^2}{(\sigma^2_{\Omega_i})}\geq \epsilon (\mathbf{Vol}^i)
\end{equation}

where $\epsilon (\mathbf{Vol}^i)$ is the threshold function defined as:

\begin{equation}
\epsilon (\mathbf{Vol}^i)=\frac{C}{\sqrt{\mathbf{Vol}^i}}
\end{equation}

\begin{remark}
***The above criteria compares the energy of the one dimensional details along the $j^{th}$ dimension with the local variance of the solution.
\end{remark}

When the quality is satisfied, the sub-partition $\Omega_i$  should be refined along $j^{th}$ dimension. So, a new partition of $\Omega$ is formed by splitting $\Omega_i$ to smaller parts. For example, when the criteria is satisfied only in a single dimension $j$, $\Omega_i=[a^i_1,b^i_1] \times \cdots \times [a^i_n, b^i_n]$ should be refined by producing two new sub-partition $\Omega_{i_1}$ and $\Omega_{i_2}$ denoted by:

\begin{equation}
\left\{
	\begin{array}{ll}
		\Omega_{i_1}=[a^{i_1}_1,b^{i_1}_1] \times \cdots \times [a^{i_1}_n, b^{i_1}_n]=[a^{i}_1,b^{i}_1] \times \cdots \times [a^{i}_j,\frac{a^{i}_j+b^{i}_j}{2}] \times \cdots \times [a^{i}_n, b^{i}_n]\\
		\Omega_{i_2}=[a^{i_2}_1,b^{i_2}_1] \times \cdots \times [a^{i_2}_n, b^{i_2}_n]=[a^{i}_1,b^{i}_1] \times \cdots \times [\frac{a^{i}_j+b^{i}_j}{2},b^{i}_j] \times \cdots \times [a^{i}_n, b^{i}_n]
		\end{array}
\right.
\end{equation}

Now local expansion of the process should be computed on new sub-partitions before we continue to analyse and determine whether more refinement is needed. This procedure should be repeated until convergence is met.

\begin{remark}
It is important that we should do computations just for newly created sub-partitions during refinement process since the expansion is local.
\end{remark}

\section{Numerical Tests}

In this section, we examine the developed method by a number of benchmark problems to evaluate first two statistical moments. The first problem is a simple function having line singularity, which is our benchmark problem to test the accuracy and convergence rate of the proposed method in the presence of sharp discontinuity. The second test problem is a two dimensional surface absorption problem, which is a stochastic ODE, and the third problem is one dimensional diffusion problem, which is categorized as a stochastic PDE. Here, dimension means the number of random inputs. The solution to first problem exhibits discontinuity, and the solution to two last problems shows sharp variations with respect to its random inputs, and thus can be suitable benchmark problems for our developed method to show its capabilities in the presence of sharp gradients or discontinuity.

To have a better understanding concerning the accuracy and convergence rate of the adaptive strategy, for all problems, we present means and variances evaluated by equation (45), mean squared error for Monte Carlo realizations and approximated values of function, and finally the number of sub-partitions for decreasing values of  $C$ in following tables. Additionally, we plot the absolute value of differences and relative errors of means and variances evaluated by equation(45) and Monte Carlo simulations, probability density function come from Monte Carlo simulations, and finally, distribution of sampling points on the partitions for different values of $C$ for each fixed $n_s$ and $n_0$.

\begin{remark}
Since our convergence is $\ell_2$-norm convergence, the best criteria to show the accuracy and robustness of the proposed method is "Mean Squared Error" (MSE), which is evaluated by  the following formula:

\begin{equation}
MSE=\frac{1}{n}\sum_{i=1}^n (Y_i-\hat{Y}_i)^2
\end{equation}

where $\{\hat{Y_i}\}_{i=1}^n$ is a vector of n predictions obtained by the proposed method and $\{Y_i\}_{i=1}^n$ is the vector of true values obtained by Monte Carlo simulation.

\medskip

In fact, as MSE convergence to zero when $C$ decreases, both mean and variance obtained by the proposed method converge to their true values. 

\end{remark}

Finally, we discuss CPU cost as an important factor in computational cost. Since we utilize non-intrusive method, the computational cost can be scaled by the number of deterministic solver simulations. Hence, the total computational cost for $N_{sb}$ sub-partitions and $n_s$ can be scaled by:

\begin{equation}
\mbox{TCC}\approx n_s.N_{sb}.SCC
\end{equation}

where TCC describes the total computational cost and SCC expresses the computational cost of deterministic solver for each simulation. To reduce the TCC, one should improve available deterministic solver by efficiently implementing the code to solve problems. Here, sparse recovery has no important impact on TCC as it involves solving sparse linear systems. Since we extract approximated solution and unknown coefficients in MWB expansion for refined sub-partitions  only, the method can highly decrease the required memory. 
\pagebreak

\subsection{A Simple Function with Line Singularity}

We first consider the performance of the proposed method for the following simple function with line singularity. In this problem, the solution exhibits sharp discontinuity in the parameter space. We are interested in verifying the convergence of both the mean and variance to their true values utilizing the proposed method for this problem. The test function is: 

\begin{equation}
F(\zeta_1, \zeta_2)=
\left\{
	\begin{array}{ll}
		\sin(\zeta_1)\sin(\zeta_2) \hspace{8pt} (\zeta_1,\zeta_2)\in [0.5,1]\times [0.5,1] \\
		0 \hspace{72pt} otherwise
	\end{array}
\right.
\end{equation}

where $\zeta_1$ and $\zeta_2$ have the uniform distribution on [0,1].
\medskip

\begin{figure}[h!]
 \includegraphics[width=4in]{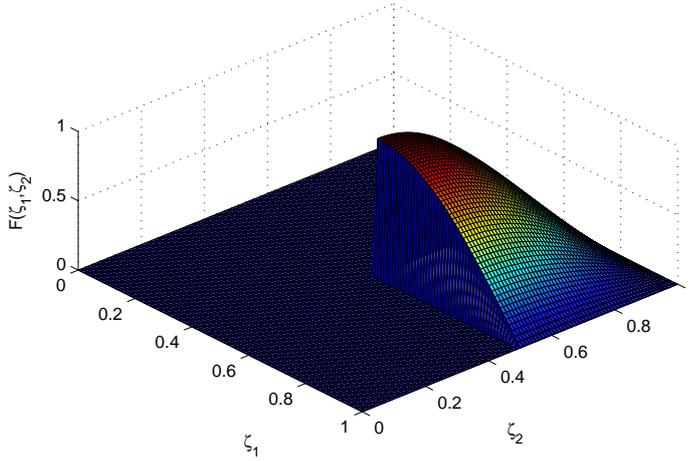}
\caption{Surface response of F ($\zeta_1$, $\zeta_2$)}
\end{figure}

Table (1) describes the quantities evaluated by the proposed method. As they illustrate, decreasing the values of $C$, the approximated values of both mean and variance converge to the true values, but in a lower rate compared to $MSE$. In fact, one should notice the fact that we obtain a sharp decrease in MSE values showing that the algorithm captures the discontinuity  accurately along the element boundaries.

\begin{table}[h!]\footnotesize
\caption{Evaluated values of $N_{sb}$, $\mathbb{E}(F)$, $\sigma^2(F)$, and MSE for different values of  $n_0$, $n_s$, and $C$.}

\medskip
\centering  
\subfloat[][Evaluated values of $N_{sb}$, $\mathbb{E}(F)$, $\sigma^2(F)$, and MSE for \\
 $n_0=2$, $n_s=8$.]{\begin{tabular}{c c c c c}  
\hline           

$C$ & $N_{sb}$ & $\mathbb{E}(\rho)$& $\sigma^2(\rho)$ & MSE \\ [0.5ex] 
\hline                  
$5\times 10^{-2}$ & 5 & 0.101544559 & 0.05274064 & 0.16377104 $\times 10^{-3}$ \\
$1\times 10^{-2}$ & 7 & 0.10043727 & 0.05199891 & 0.037920051$\times 10^{-3}$ \\
$5\times 10^{-3}$ & 15 & 0.10059612 & 0.05160273 & 0.02198535$\times 10^{-3}$ \\
$1\times 10^{-3}$ & 26 & 0.10227871 & 0.05274955 & 0.00389531$\times 10^{-3}$ \\
$5\times 10^{-4}$ & 52 & 0.10220550 & 0.05256251 & 0.00042613$\times 10^{-3}$ \\ [1ex]      

\hline 

\end{tabular}
}

\bigskip

\subfloat[][Evaluated values of $N_{sb}$, $\mathbb{E}(F)$, $\sigma^2(F)$, and MSE for \\
 $n_0=3$, $n_s=12$.]{\begin{tabular}{c c c c c} 
 \hline           
 
 $C$ & $N_{sb}$ & $\mathbb{E}(\rho)$& $\sigma^2(\rho)$ & MSE \\ [0.5ex] 
 \hline                  
 $5\times 10 ^{-2}$ & 4 & 0.10054470 & 0.05170200 &   0.93335605 $\times 10^{-3}$ \\ 
 $1\times 10^{-2}$ & 7 & 0.10090933 & 0.05177892 & 0.03124397 $\times 10^{-3}$ \\
 $5\times 10^{-3}$ & 8 & 0.10155312 & 0.05183906 & 0.00728255 $\times 10^{-3}$ \\
 $1\times 10^{-3}$ & 14 & 0.10120334 & 0.05329833 & 0.00105085 $\times 10^{-3}$ \\
 $5\times 10^{-4}$ & 21 & 0.101763329 & 0.05244912 & 0.00053608 $\times 10^{-3}$  \\ [1ex]      
 
 \hline 
 
\end{tabular}}
 
\end{table}

Figures (2-6) show the statistics and the partitions of random parameter space which is partitioned by adaptive strategy for different values of $C$, thus the problem has two stochastic dimensions $\zeta=(\zeta_1,\zeta_2)$ with $\zeta_1$, $\zeta_2$ both uniformly distributed on [0,1]. As the figures suggest, in order to have a suitable approximation of solutions, the number of samples should be increased when the polynomial order is increased.

\begin{figure}
\subfloat[][$C=5\times 10^{-2}$]{
\includegraphics[width=0.5\textwidth]{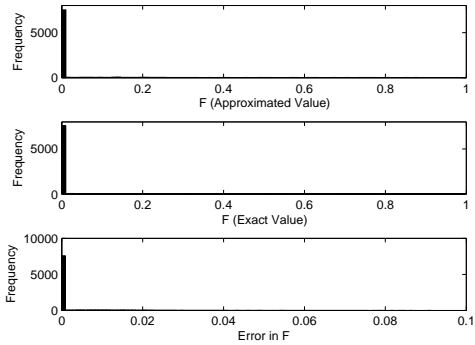}
\label{fig:subfig1}}
\qquad 
\subfloat[][$C=1\times 10^{-2}$]{
\includegraphics[width=0.5\textwidth]{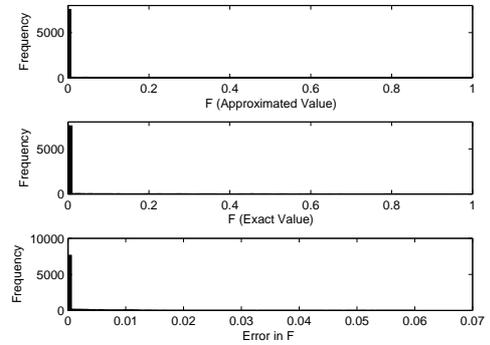}
\label{fig:subfig2}}
\\

\subfloat[][$C=5\times 10^{-3}$]{
\includegraphics[width=0.5\textwidth]{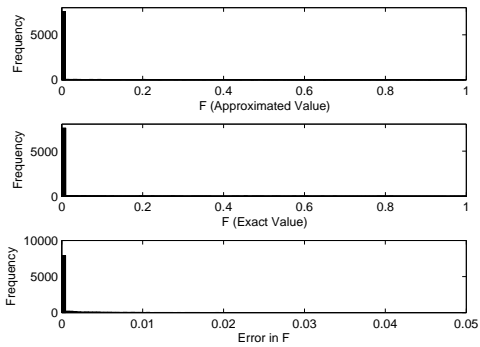}
\label{fig:subfig3}}
\qquad
\subfloat[][$C=1\times 10^{-3}$]{
\includegraphics[width=0.5\textwidth]{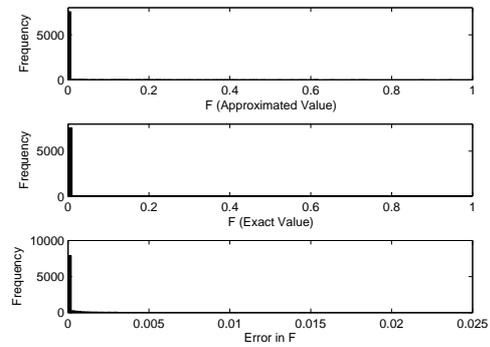}
\label{fig:subfig4}}
\\

\subfloat[][$C=5\times 10^{-4}$]{
\includegraphics[width=0.5\textwidth]{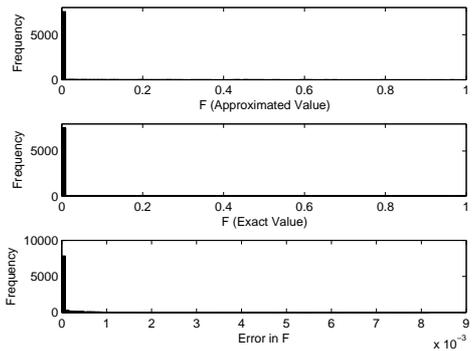}
\label{fig:subfig4}}

\tiny
\caption{Distribution of function exact values evaluated by Monte Carlo simulation, function approximated evaluated by proposed method, and absolute value of error between them for $n=2, n_0=2, \mbox{ and } n_s=8.$}
\label{fig:globfig}
\end{figure}

\begin{figure}
\subfloat[][$C=5\times 10^{-2}$]{
\includegraphics[width=0.45\textwidth]{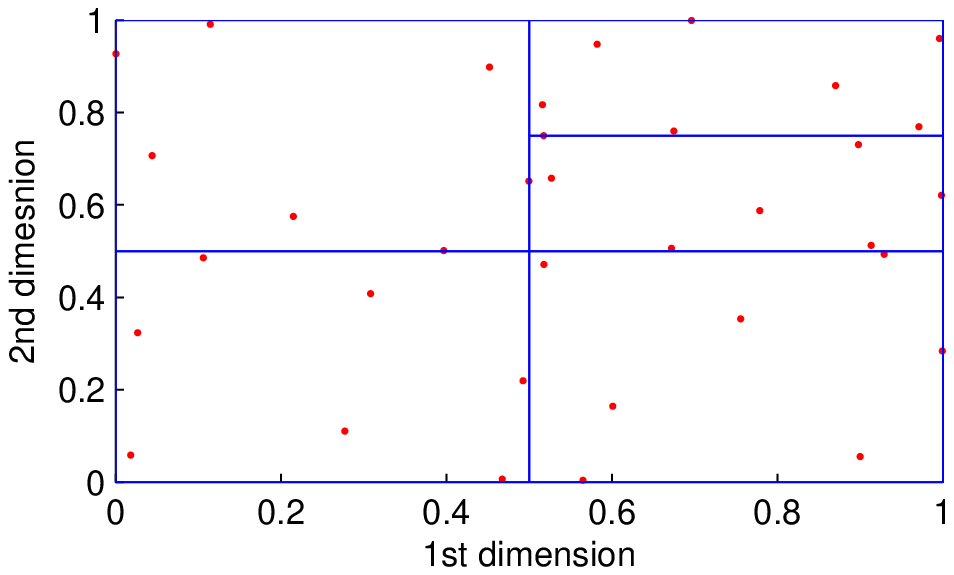}
\label{fig:subfig1}}
\qquad 
\subfloat[][$C=1\times 10^{-2}$]{
\includegraphics[width=0.45\textwidth]{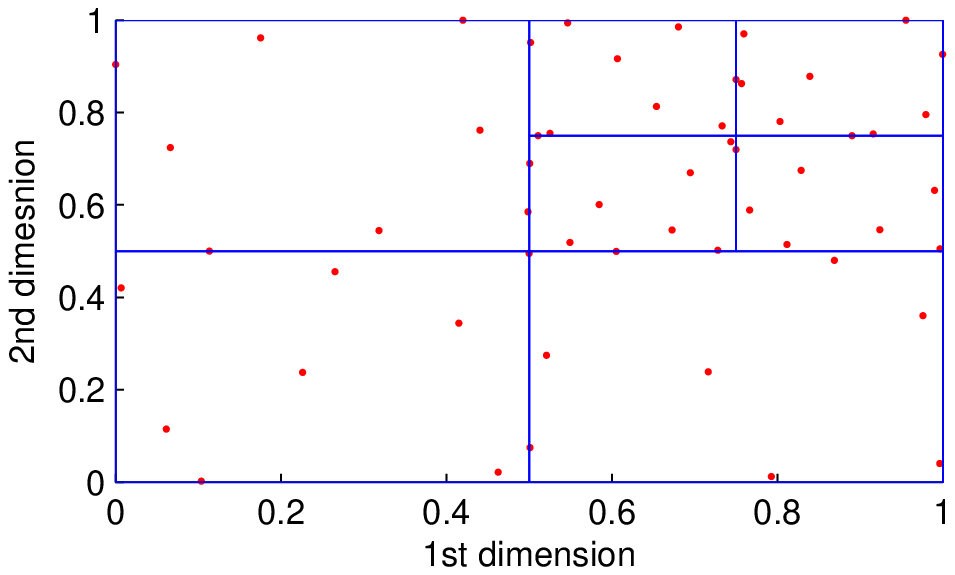}
\label{fig:subfig2}}
\\

\subfloat[][$C=5\times 10^{-3}$]{
\includegraphics[width=0.45\textwidth]{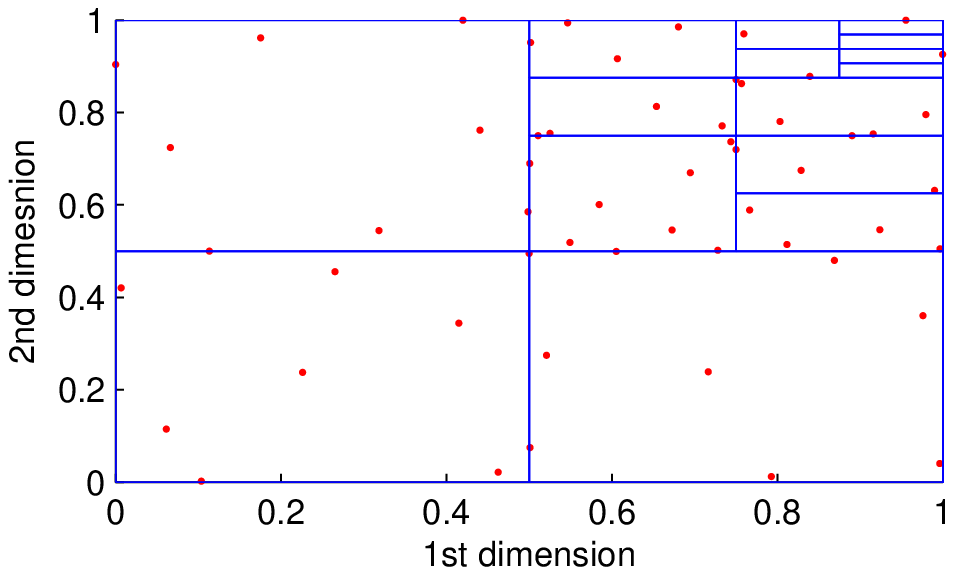}
\label{fig:subfig3}}
\qquad
\subfloat[][$C=1\times 10^{-3}$]{
\includegraphics[width=0.45\textwidth]{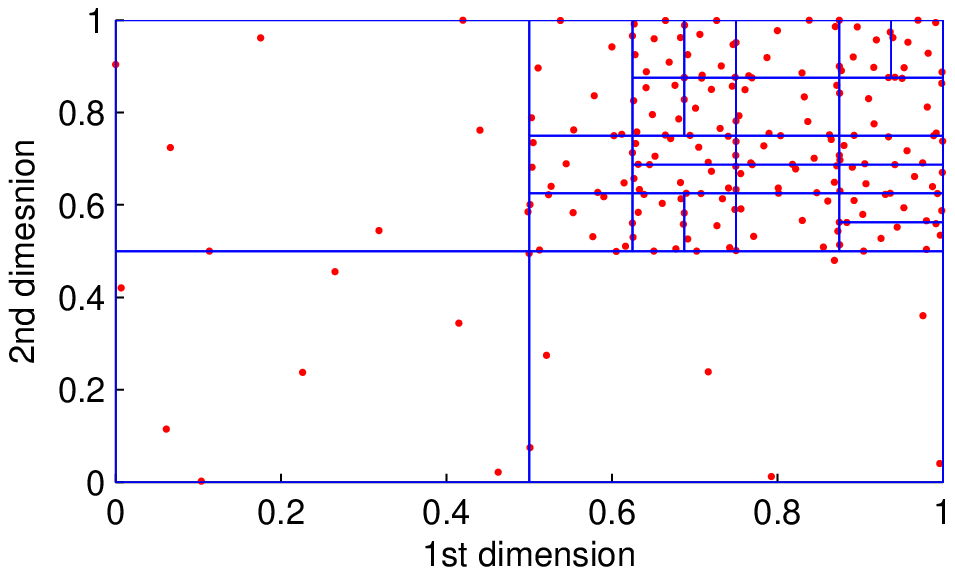}
\label{fig:subfig4}}
\\

\subfloat[][$C=5\times 10^{-4}$]{
\includegraphics[width=0.45\textwidth]{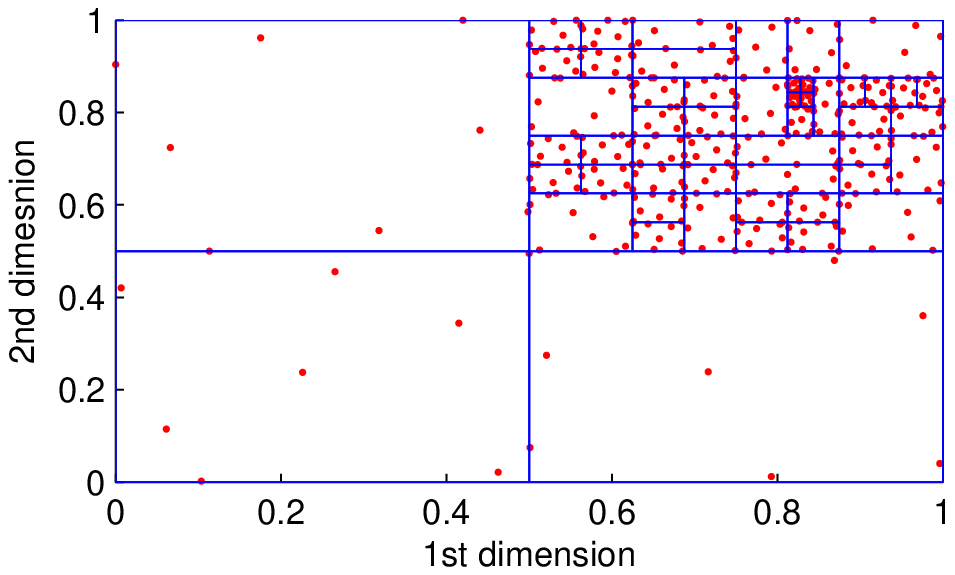}
\label{fig:subfig4}}

\caption{Partition of the random parameter space with sample points utilized for $n=2, n_0=2, \mbox{ and } n_s=8.$}
\label{fig:globfig}
\end{figure}

\begin{figure}
\subfloat[][$C=5\times 10^{-2}$]{
\includegraphics[width=0.5\textwidth]{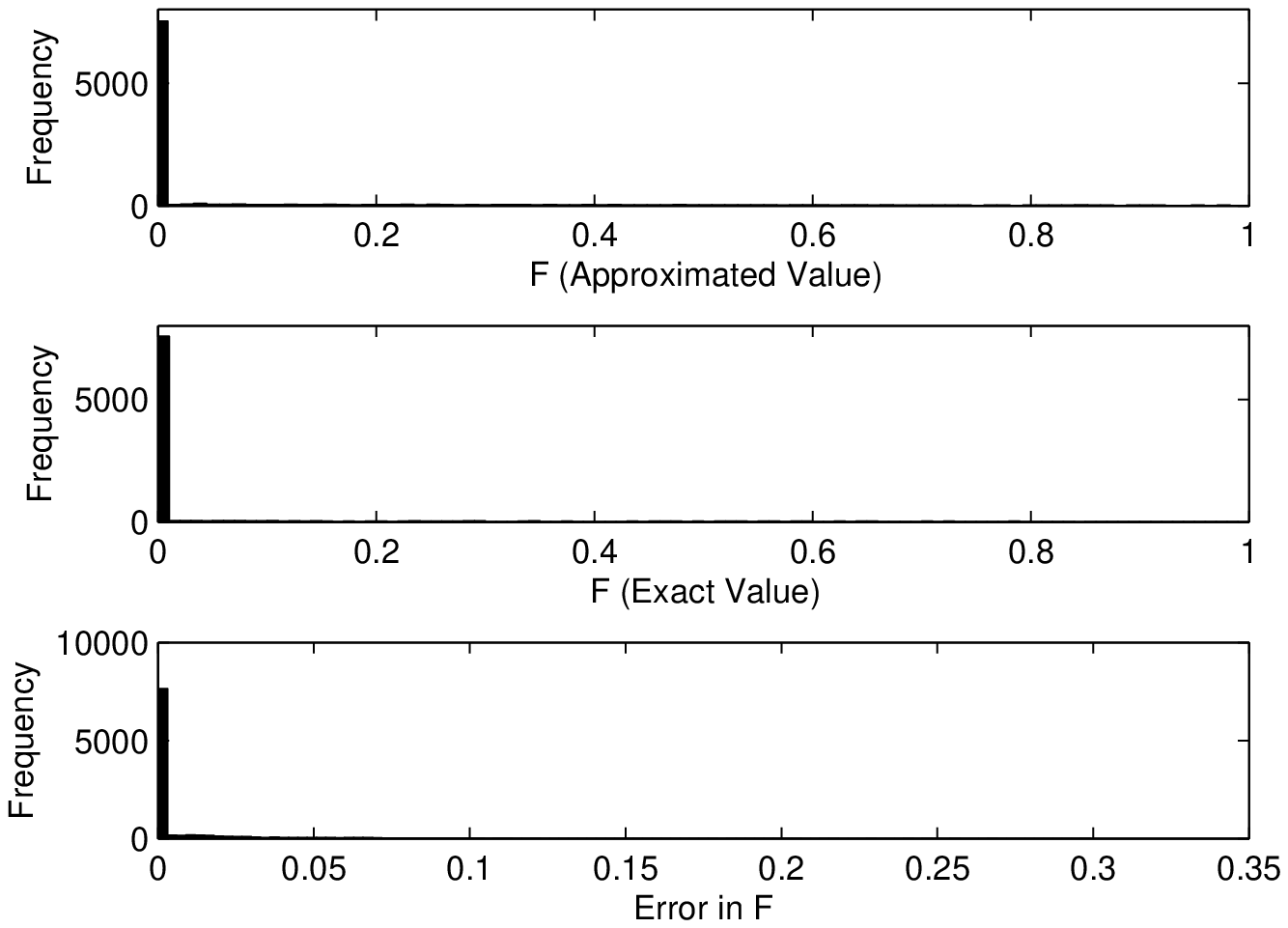}
\label{fig:subfig1}}
\qquad 
\subfloat[][$C=1\times 10^{-2}$]{
\includegraphics[width=0.5\textwidth]{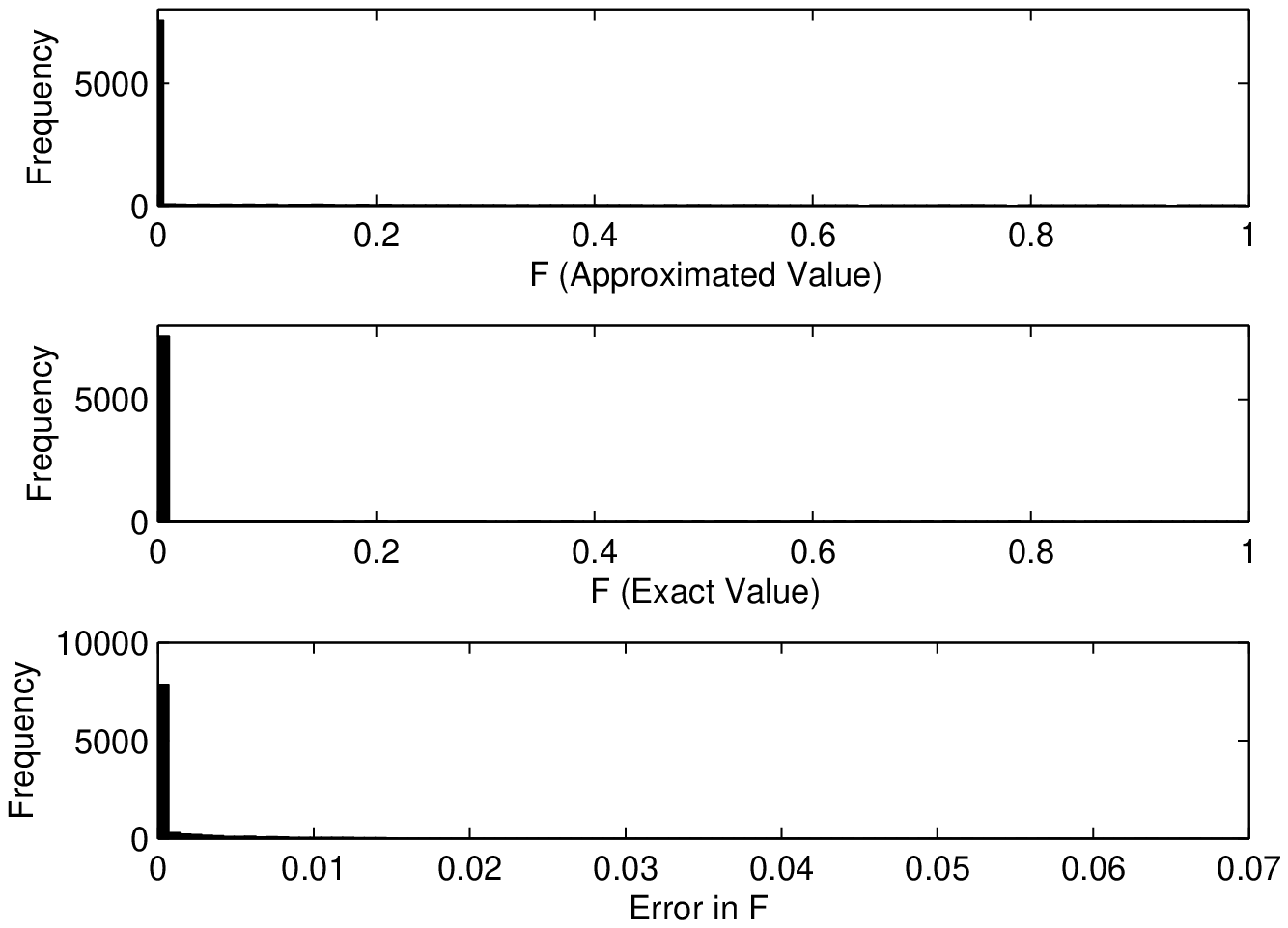}
\label{fig:subfig2}}
\\

\subfloat[][$C=5\times 10^{-3}$]{
\includegraphics[width=0.5\textwidth]{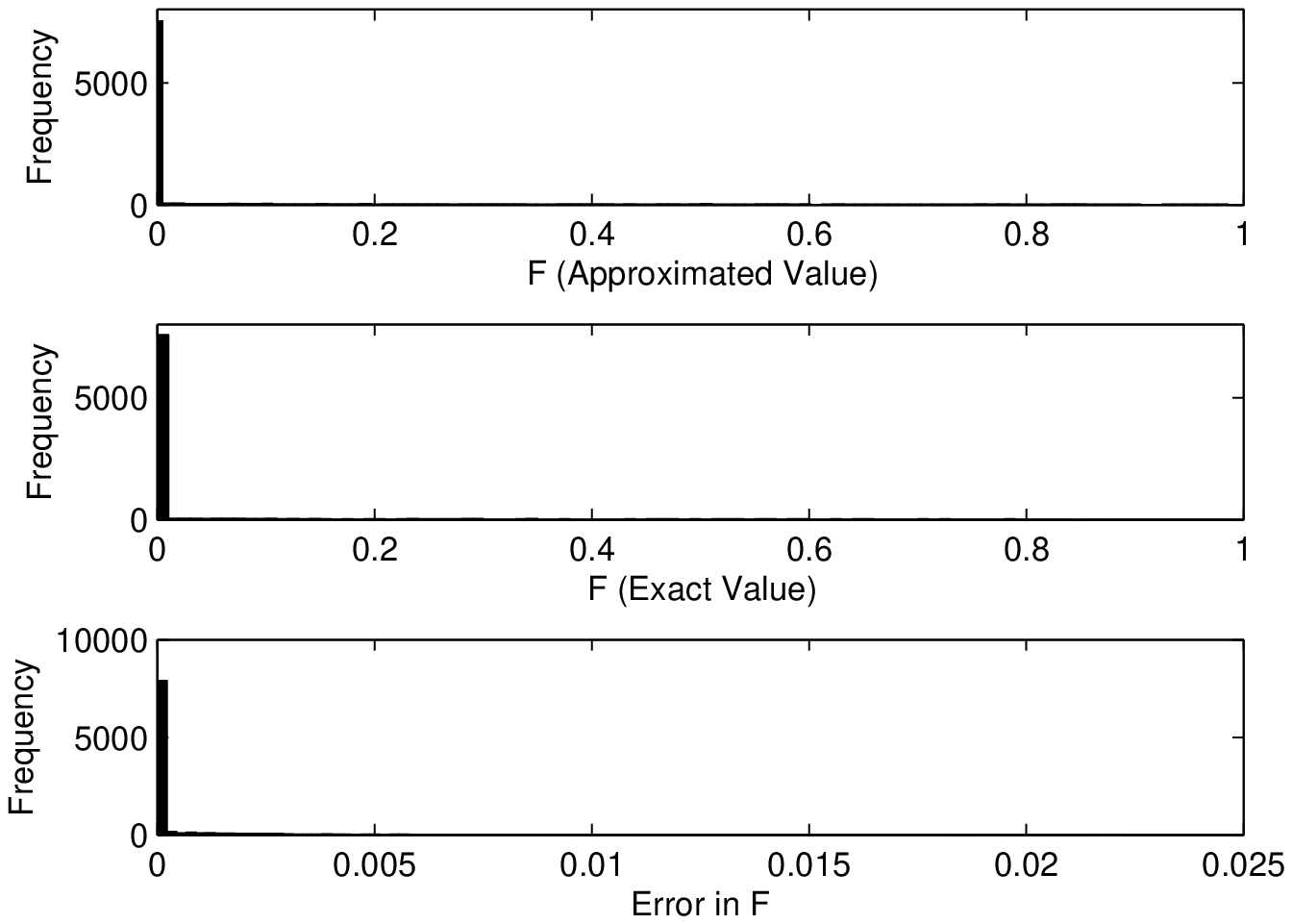}
\label{fig:subfig3}}
\qquad
\subfloat[][$C=1\times 10^{-3}$]{
\includegraphics[width=0.5\textwidth]{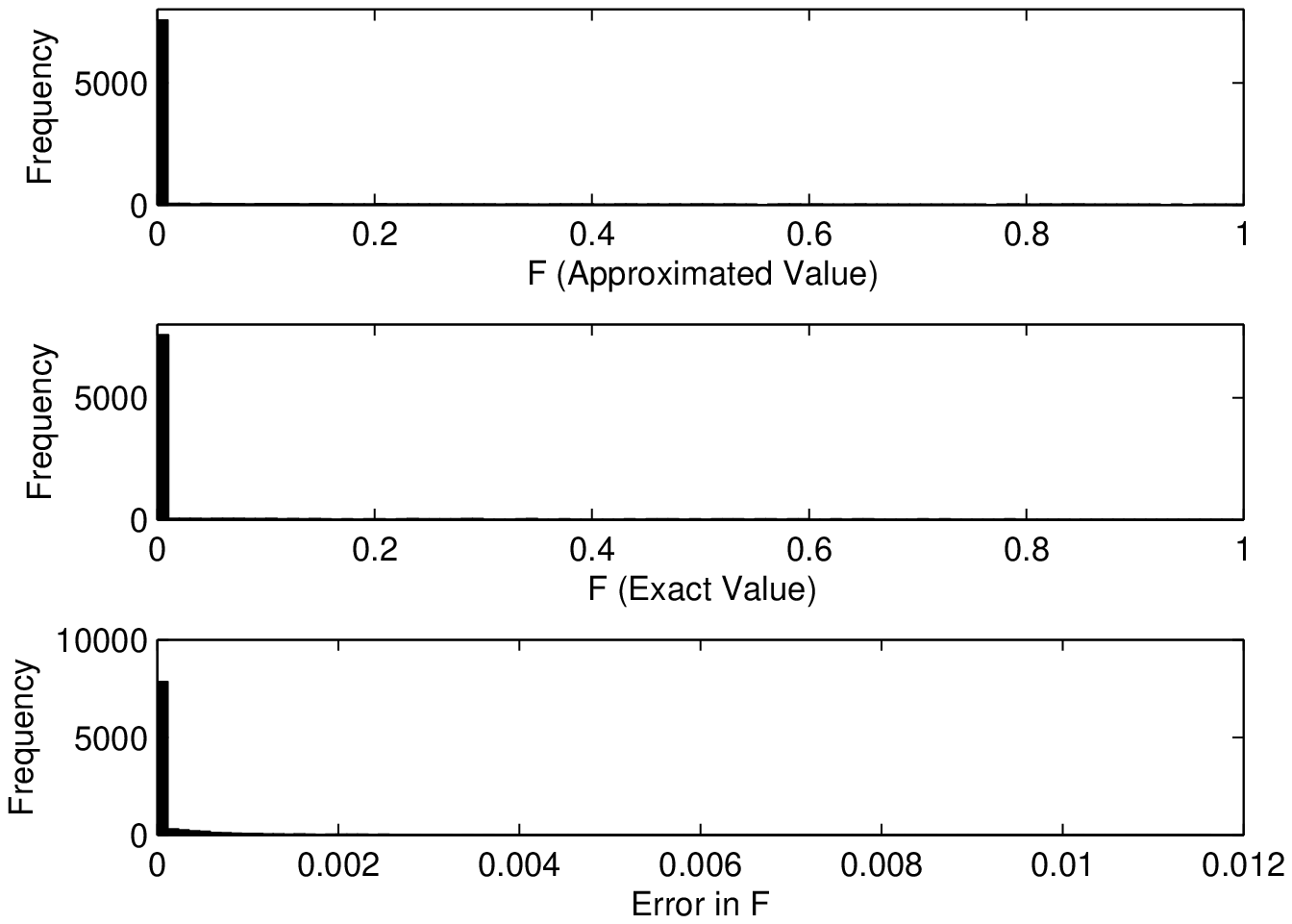}
\label{fig:subfig4}}
\\

\subfloat[][$C=5\times 10^{-4}$]{
\includegraphics[width=0.5\textwidth]{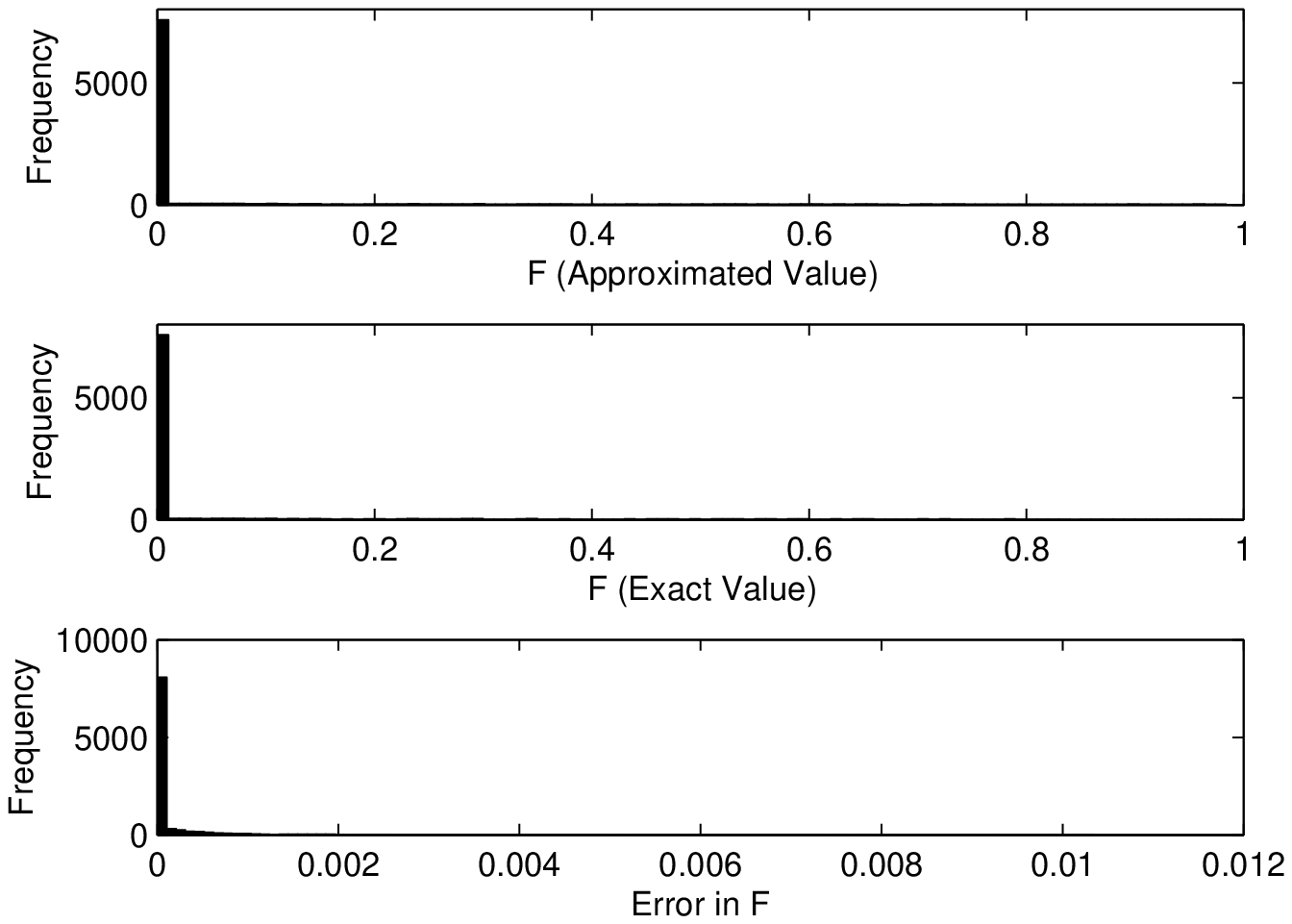}
\label{fig:subfig4}}

\caption{Distribution of function exact values evaluated by Monte Carlo simulation, function approximated evaluated by proposed method, and absolute value of error between them for $n=2, n_0=3, \mbox{ and } n_s=12.$}
\label{fig:globfig}
\end{figure}

\begin{figure}
\subfloat[][$C=5\times 10^{-2}$]{
\includegraphics[width=0.39\textwidth]{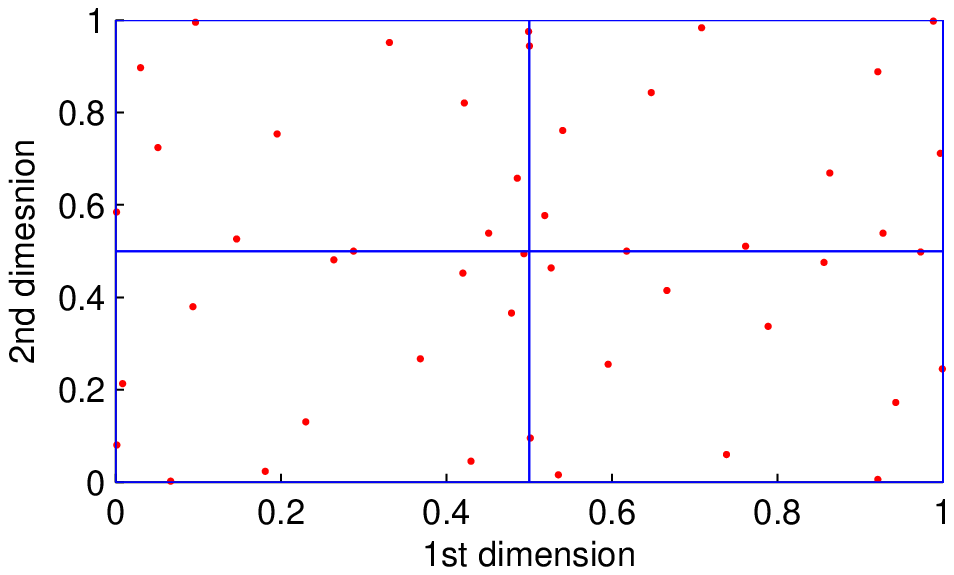}
\label{fig:subfig1}}
\qquad 
\subfloat[][$C=1\times 10^{-2}$]{
\includegraphics[width=0.39\textwidth]{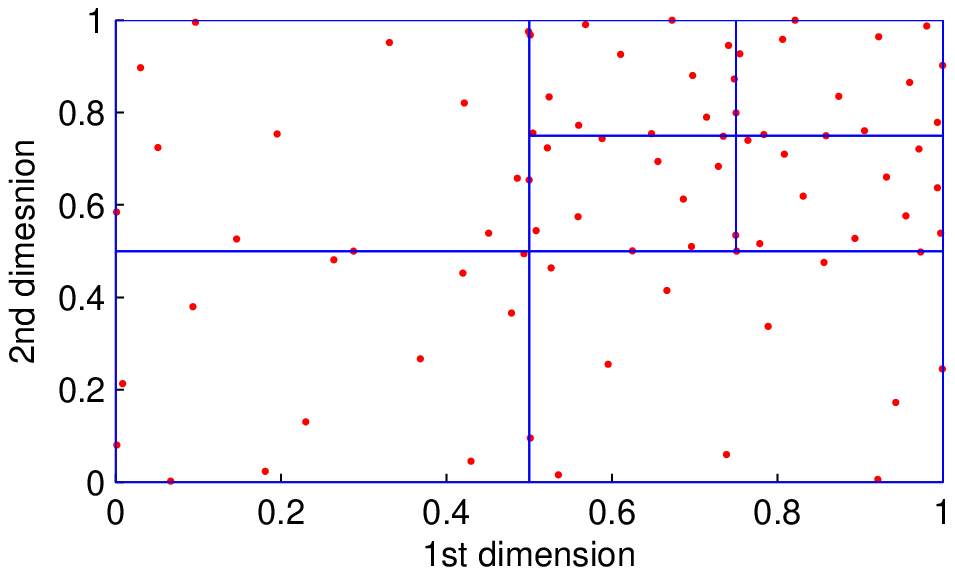}
\label{fig:subfig2}}
\\

\subfloat[][$C=5\times 10^{-3}$]{
\includegraphics[width=0.39\textwidth]{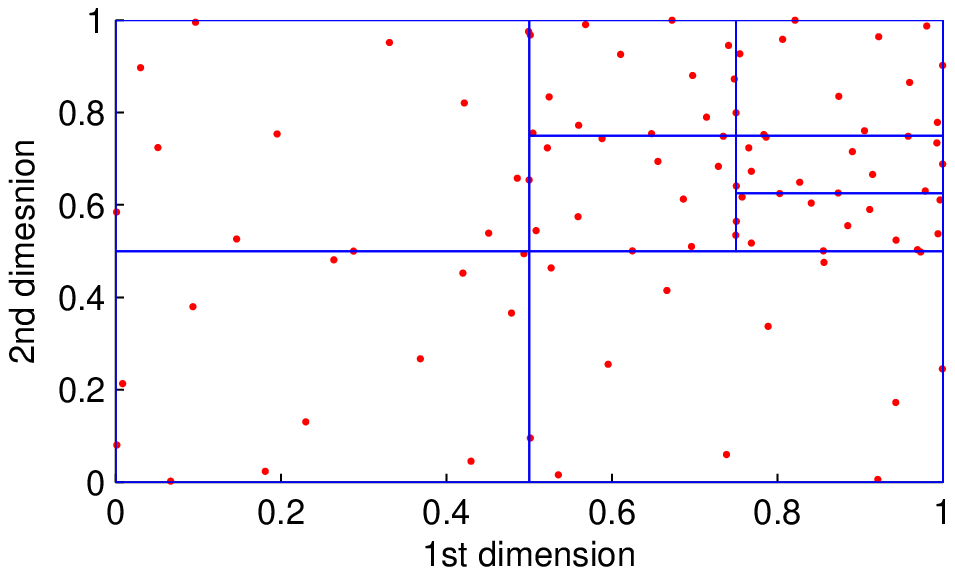}
\label{fig:subfig3}}
\qquad
\subfloat[][$C=1\times 10^{-3}$]{
\includegraphics[width=0.39\textwidth]{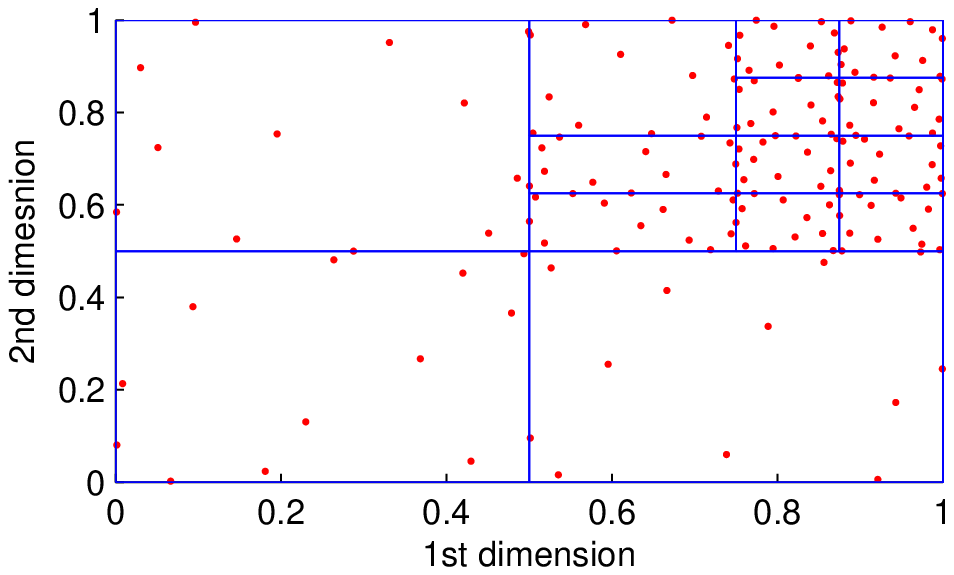}
\label{fig:subfig4}}
\\

\subfloat[][$C=5\times 10^{-4}$]{
\includegraphics[width=0.39\textwidth]{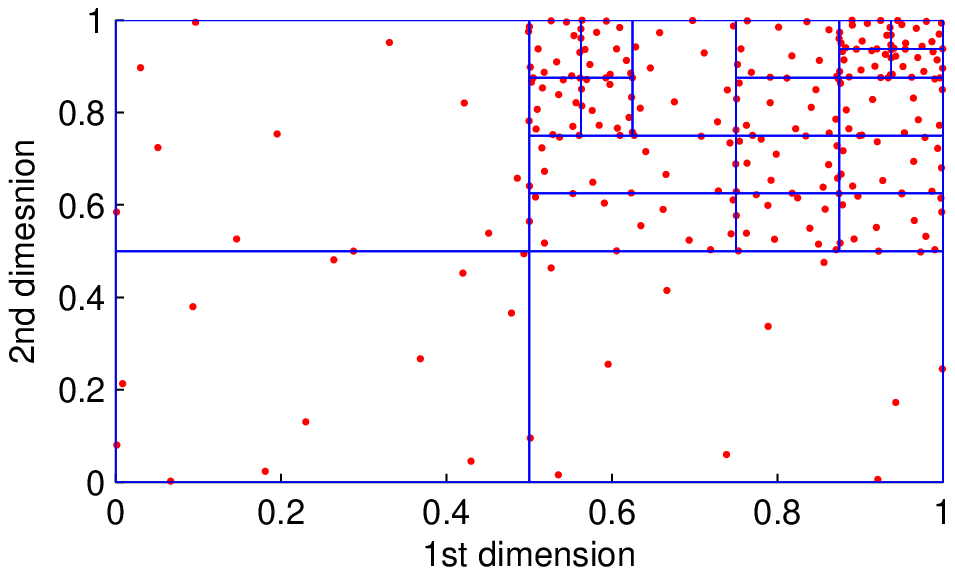}
\label{fig:subfig4}}

\caption{Partition of the random parameter space with sample points utilized for $n=2, n_0=3, \mbox{ and } n_s=12.$}
\label{fig:globfig}
\end{figure}

\begin{figure}
\subfloat[][Relative erros of variance]{
\includegraphics[width=0.45\textwidth]{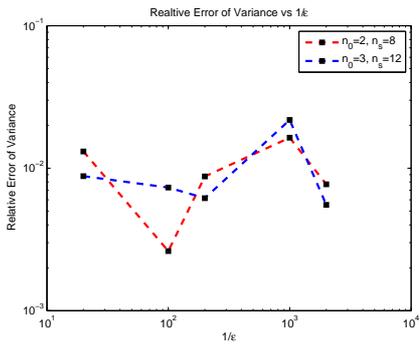}
\label{fig:subfig1}}
\qquad 
\subfloat[][Relative erros of mean]{
\includegraphics[width=0.45\textwidth]{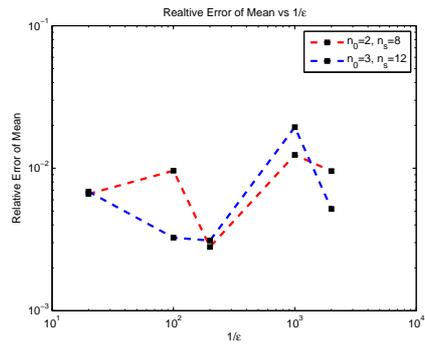}
\label{fig:subfig2}}
\\

\subfloat[][Error values (MSE)]{
\includegraphics[width=0.45\textwidth]{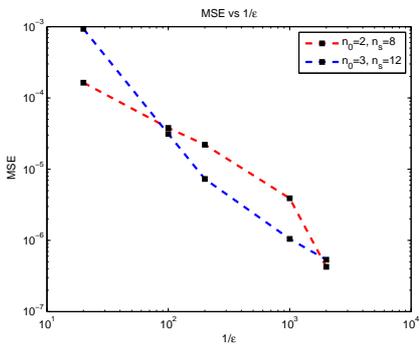}
\label{fig:subfig3}}

\caption{Relative error values of both mean and variance evaluated by Monte Carlo simulation and proposed method, and error (MSE) for all cases.}
\label{fig:globfig}
\end{figure}

\newpage
\subsection{Surface Absorption}  
We examine the proposed method on the following ODE considered by Makeev and his co-authors [49] to analyse the stability domain and dynamical behavior of kinetics of $\mathbf{NO}$ reduction by $\mathbf{H_2}$ on $\mathbf{Pt}$ and $\mathbf{Rh}$ surfaces:

\begin{equation}
\left\{
	\begin{array}{ll}
\frac{d\rho}{dt}= \alpha(1-\rho)-\gamma\rho+\beta\rho(1-\rho)^2 \\
\rho(t=0)=\rho_0
\end{array}
\right.
\end{equation}

Equation (95) expresses the time-progress of the absorbed coverage $\rho\in[0,1]$ of $\mathbf{NO}$ with respect to given input parameters. In this problem, $\alpha$ and $\gamma$ are the surface absorption and  desorption rates, respectively. Also $\beta$ and $\rho_0$ are respectively the recombination (or reaction) rate and initial covering  taken as our bifurcation parameters. This problem has one or two fixed point depending on the value of recombination rate $\beta$ whereas shows the smooth dependency with respect to the other parameters. We focus on the statistics of the solution $\rho$ at $t=1$ with considering the uncertainty in the initial covering $\rho_0$ and recombination rate $\beta$. Similar to the previous work [14], $\rho_0$ and $\beta$ are chosen as uncertain parameters such that $\rho_0$ is uniformly distributed on $[0,1]$ and $\beta$ is uniformly distributed on $[0,20]$. The values of $\gamma$ and $\alpha$ are set to $.01$ and $1$ respectively. We use ODE-45 function in Matlab to solve this ODE with time step $t=0.001$.

\begin{figure}
\includegraphics[width=4in]{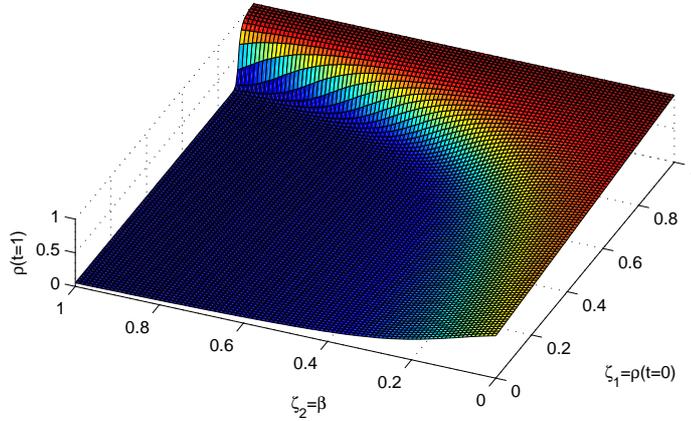}
\label{fig:subfig1}
\caption{Surface response of $\rho$ (t=1))}
\end{figure}

Table (2) explains our desired quantities evaluated by our method. One can see the convergence of mean and variance evaluated via present method to true values by comparing with mean and variance values evaluated by Monte Carlo simulations with 1000000 realizations which are .   

\begin{table}[]\footnotesize
\caption{Evaluated values of $N_{sb}$, $\mathbb{E}(\rho)$, $\sigma^2(\rho)$, and MSE for different values of  $n_0$, $n_s$, and $C$.}

\medskip
\centering  
\subfloat[][Evaluated values of $N_{sb}$, $\mathbb{E}(\rho)$, $\sigma^2(\rho)$, and MSE for \\
 $n_0=3$, $n_s=13$.]{\begin{tabular}{c c c c c}  
\hline           

$C$ & $N_{sb}$ & $\mathbb{E}(\rho)$& $\sigma^2(\rho)$ & MSE \\ [0.5ex] 
\hline                  
$5\times 10 ^{-2}$ & 4 & 0.37318715 & 0.10799163 & 0.01237100 \\ 
$1\times 10^{-2}$ & 18 & 0.35071422 & 0.11005503 & 0.00040281 \\
$5\times 10^{-3}$ & 24 & 0.35043779 & 0.10961185 & 0.00014461 \\
$1\times 10^{-3}$ & 69 & 0.35027543 & 0.11026816 & 0.00001921 \\
$5\times 10^{-4}$ & 102 & 0.35032569 & 0.11027245 & 0.00000123 \\
$1\times 10^{-4}$ & 192 &  0.35037020 & 0.11027245 & 0.00000028 \\ [1ex]      

\hline 

\end{tabular}
}

\bigskip

\subfloat[][Evaluated values of $N_{sb}$, $\mathbb{E}(\rho)$, $\sigma^2(\rho)$, and MSE for \\
 $n_0=3$, $n_s=14$.]{\begin{tabular}{c c c c c} 
 \hline           
 
 $C$ & $N_{sb}$ & $\mathbb{E}(\rho)$& $\sigma^2(\rho)$ & MSE \\ [0.5ex] 
 \hline                  
 $5\times 10 ^{-2}$ & 7 & 0.35071422 & 0.11005503 & 0.00369282 \\ 
 $1\times 10^{-2}$ & 16 & 0.35142326 & 0.10755043 & 0.00119106 \\
 $5\times 10^{-3}$ & 29 & 0.35042216 & 0.11010832 & 0.00005763 \\
 $1\times 10^{-3}$ & 69 & 0.35032293 & 0.11026043 & 0.00000546 \\
 $5\times 10^{-4}$ & 97 & 0.35033357 & 0.11025450 & 0.00000264 \\
 $1\times 10^{-4}$ & 210 &  0.35038897 & 0.11032818 & 0.00000007 \\ [1ex]      
 
 \hline 
 
 \end{tabular}}
 
 \bigskip

 \subfloat[][Evaluated values of $N_{sb}$, $\mathbb{E}(\rho)$, $\sigma^2(\rho)$, and MSE for \\
  $n_0=4$, $n_s=16$.]{\begin{tabular}{c c c c c} 
  \hline             
  
  $C$ & $N_{sb}$ & $\mathbb{E}(\rho)$& $\sigma^2(\rho)$ & MSE \\ [0.5ex] 
  \hline                  
  $5\times 10 ^{-2}$ & 6 &  0.34423843 & 0.10530110 & 0.00163459 \\ 
  $1\times 10^{-2}$ & 9 & 0.34920926 & 0.10735340 & 0.00065797 \\
  $5\times 10^{-3}$ & 15 & 0.34886888 & 0.10933321 & 0.00012769 \\
  $1\times 10^{-3}$ & 48 & 0.35005009 & 0.10933321 & 0.00001088 \\
  $5\times 10^{-4}$ & 68 & 0.35025497 & 0.10933321 & 0.00000334 \\
  $1\times 10^{-4}$ & 154 &  0.35043888 & 0.10933321 & 0.00000024 \\ [1ex]      
  
  \hline 
  
  \end{tabular}}
 
 \bigskip
 
 \subfloat[][Evaluated values of $N_{sb}$, $\mathbb{E}(\rho)$, $\sigma^2(\rho)$, and MSE for \\
   $n_0=4$, $n_s=20$.]{\begin{tabular}{c c c c c} 
   \hline             
   
   $C$ & $N_{sb}$ & $\mathbb{E}(\rho)$& $\sigma^2(\rho)$ & MSE \\ [0.5ex] 
   \hline                  
   $5\times 10 ^{-2}$ & 5 & 0.34912794 & 0.11001077 & 0.00533151 \\ 
   $1\times 10^{-2}$ & 15 & 0.34944424 & 0.10984353 & 0.00037115 \\
   $5\times 10^{-3}$ & 24 & 0.34999732 & 0.10983560 & 0.00007672 \\
   $1\times 10^{-3}$ & 54 & 0.35039355 & 0.11030063 & 0.00000438 \\
   $5\times 10^{-4}$ & 68 & 0.35036849 &  0.11031724 & 0.00000136 \\
   $1\times 10^{-4}$ & 146 &  0.35039224 &  0.11032860 & 0.00000007 \\ [1ex]      
   
   \hline 
   
   \end{tabular}}

\end{table}

Figure (8-12) show the statistics, the partition of random parameter space, which is partitioned by adaptive strategy. Here, the problem has two stochastic dimensions, $\zeta=(\zeta_1,\zeta_2)$ with $\zeta_1$ uniformly distributed on $ [0,1] $ and $\zeta_2$ uniformly distributed on $ [0,20] $.

\begin{figure}
\subfloat[][$C=5\times 10^{-2}$]{
\includegraphics[width=0.5\textwidth]{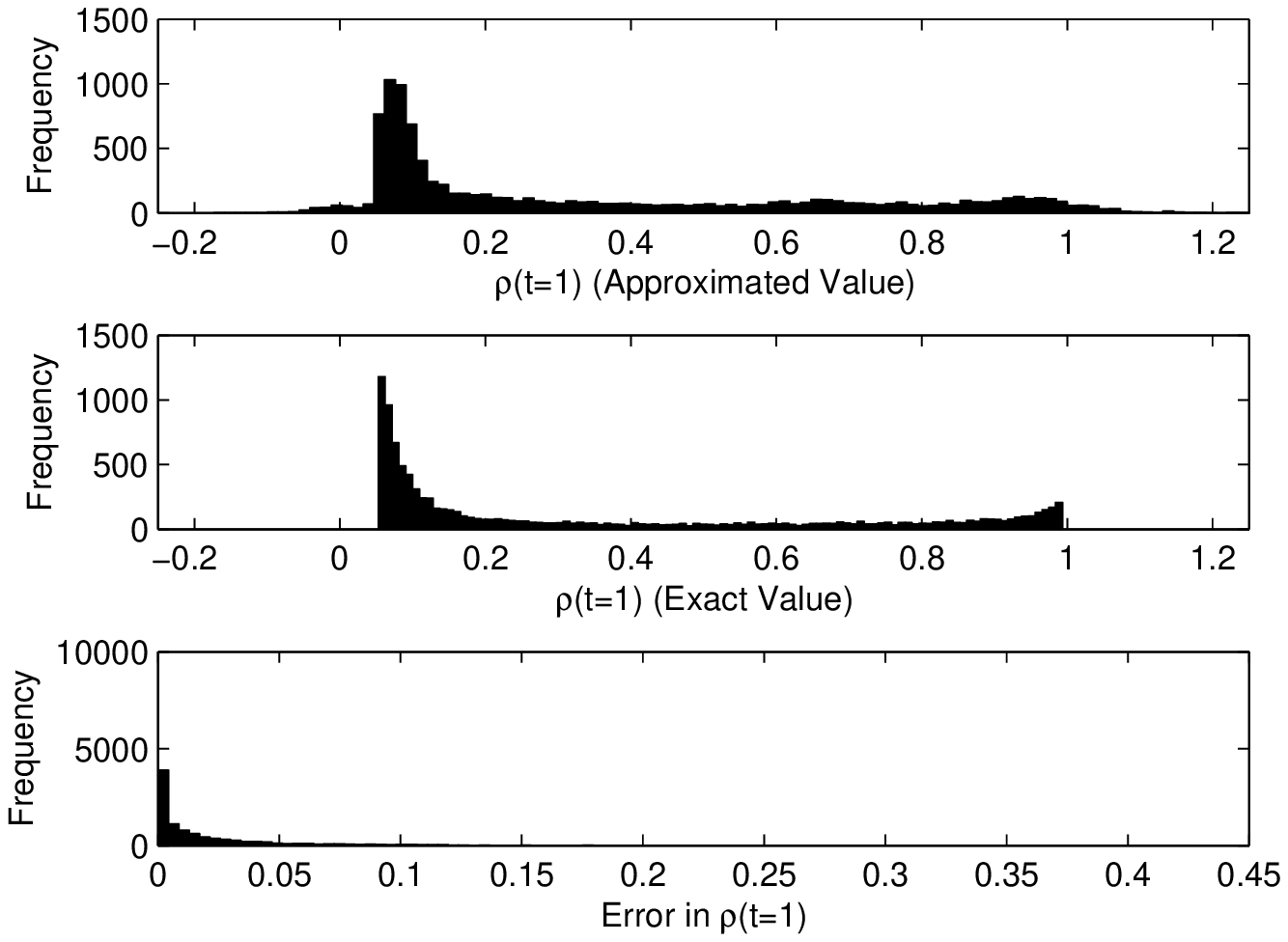}
\label{fig:subfig1}}
\qquad 
\subfloat[][$C=1\times 10^{-2}$]{
\includegraphics[width=0.5\textwidth]{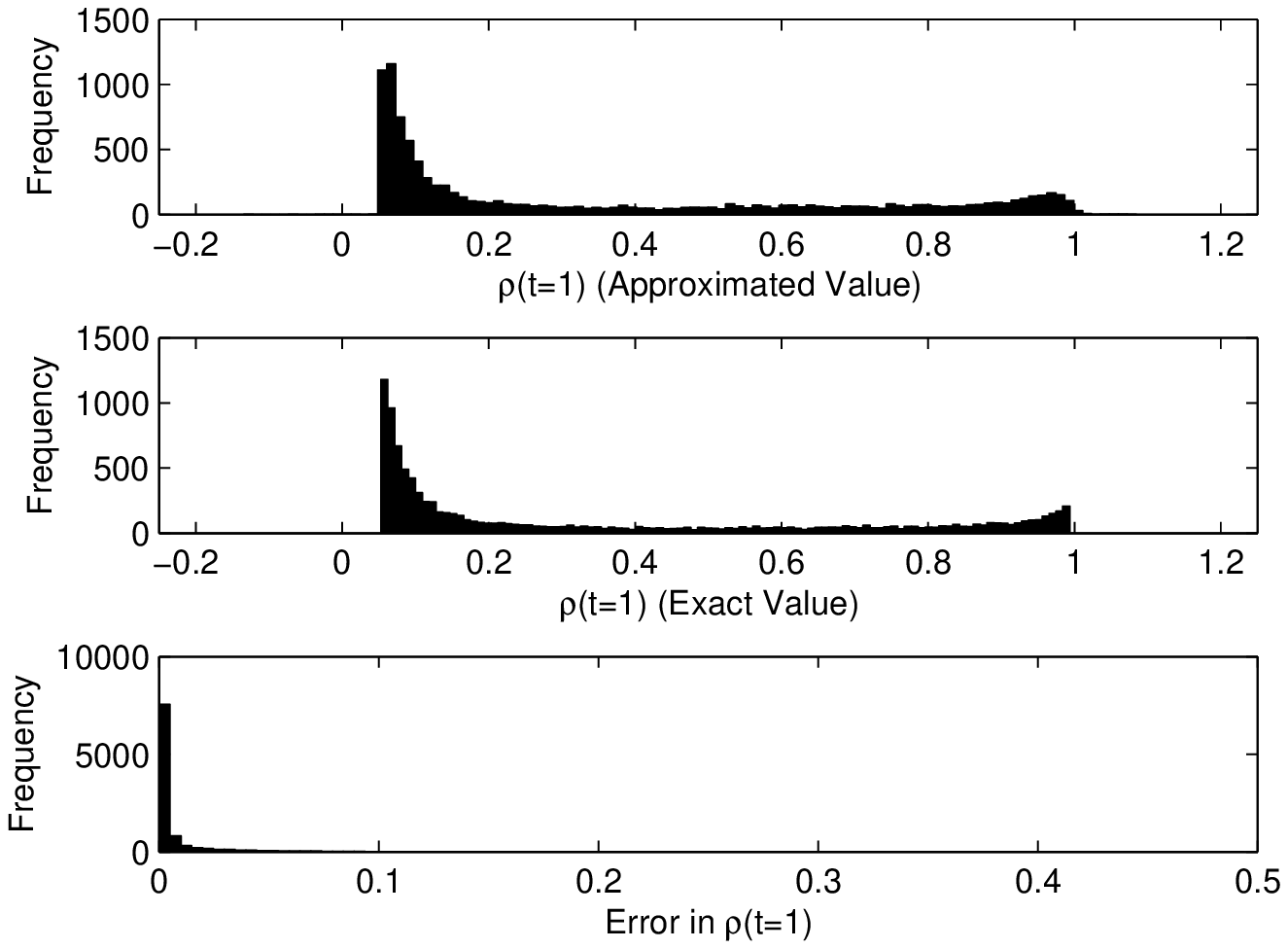}
\label{fig:subfig2}}
\\

\subfloat[][$C=5\times 10^{-3}$]{
\includegraphics[width=0.5\textwidth]{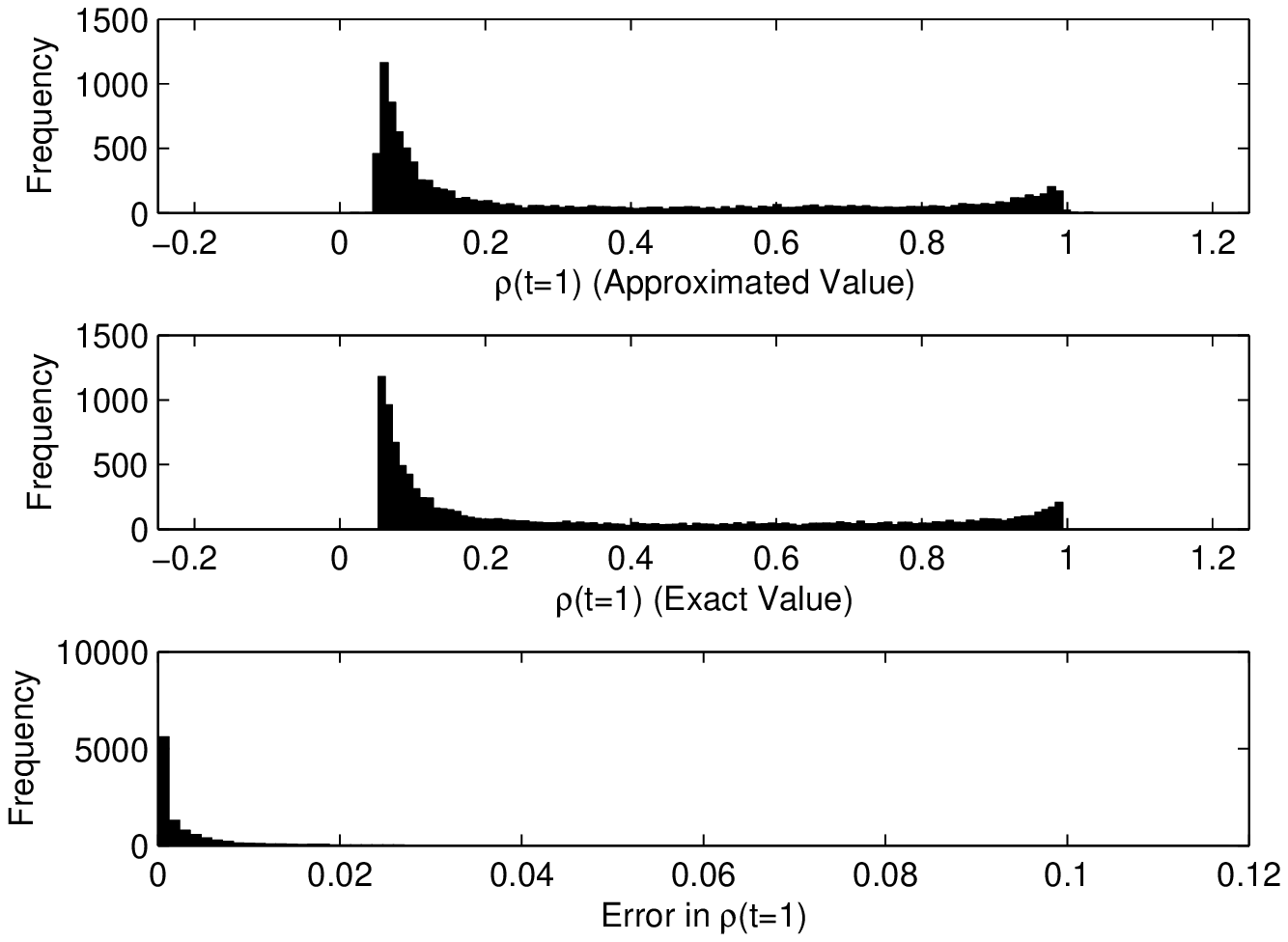}
\label{fig:subfig3}}
\qquad
\subfloat[][$C=1\times 10^{-3}$]{
\includegraphics[width=0.5\textwidth]{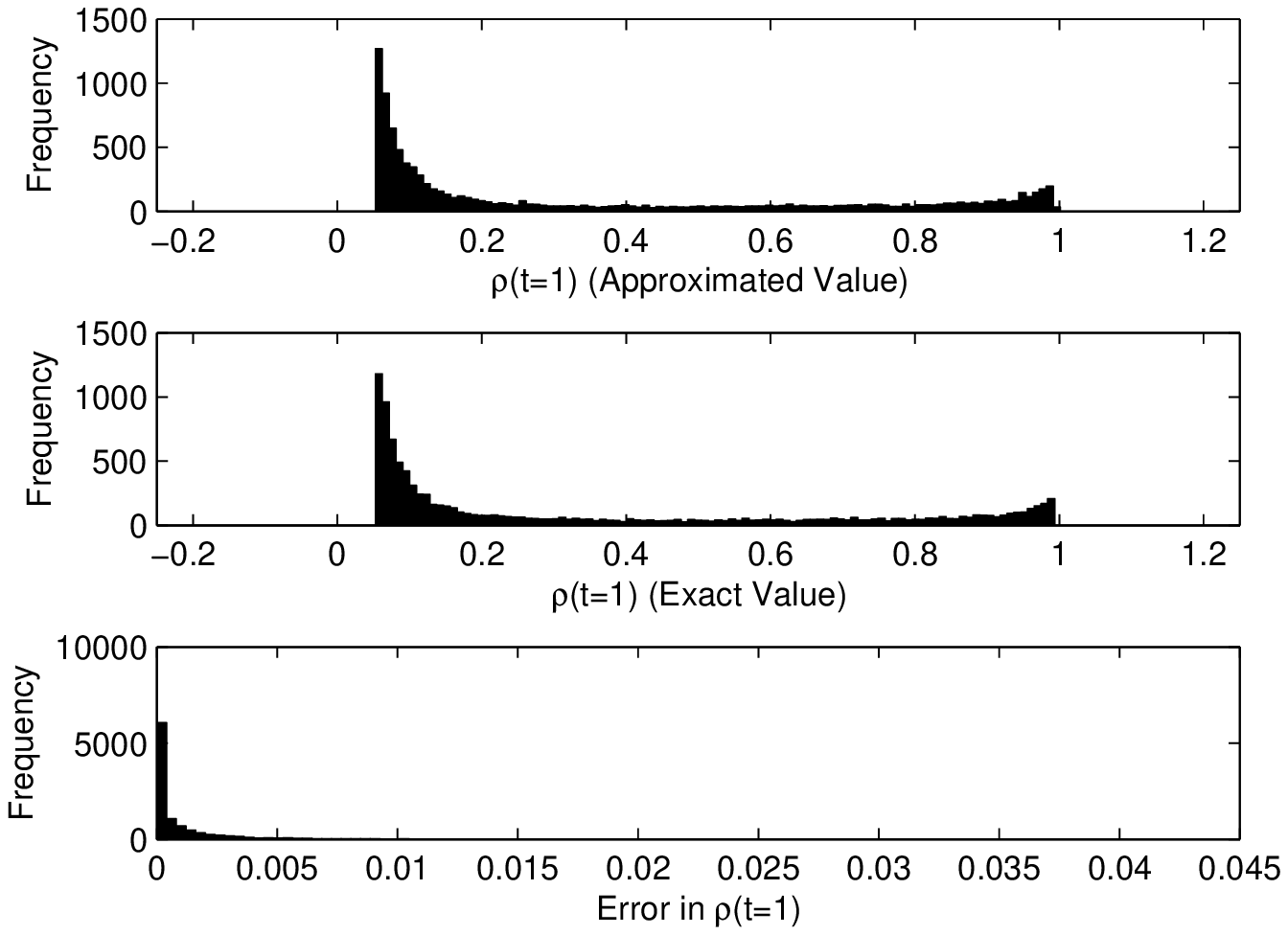}
\label{fig:subfig4}}
\\

\subfloat[][$C=5\times 10^{-4}$]{
\includegraphics[width=0.5\textwidth]{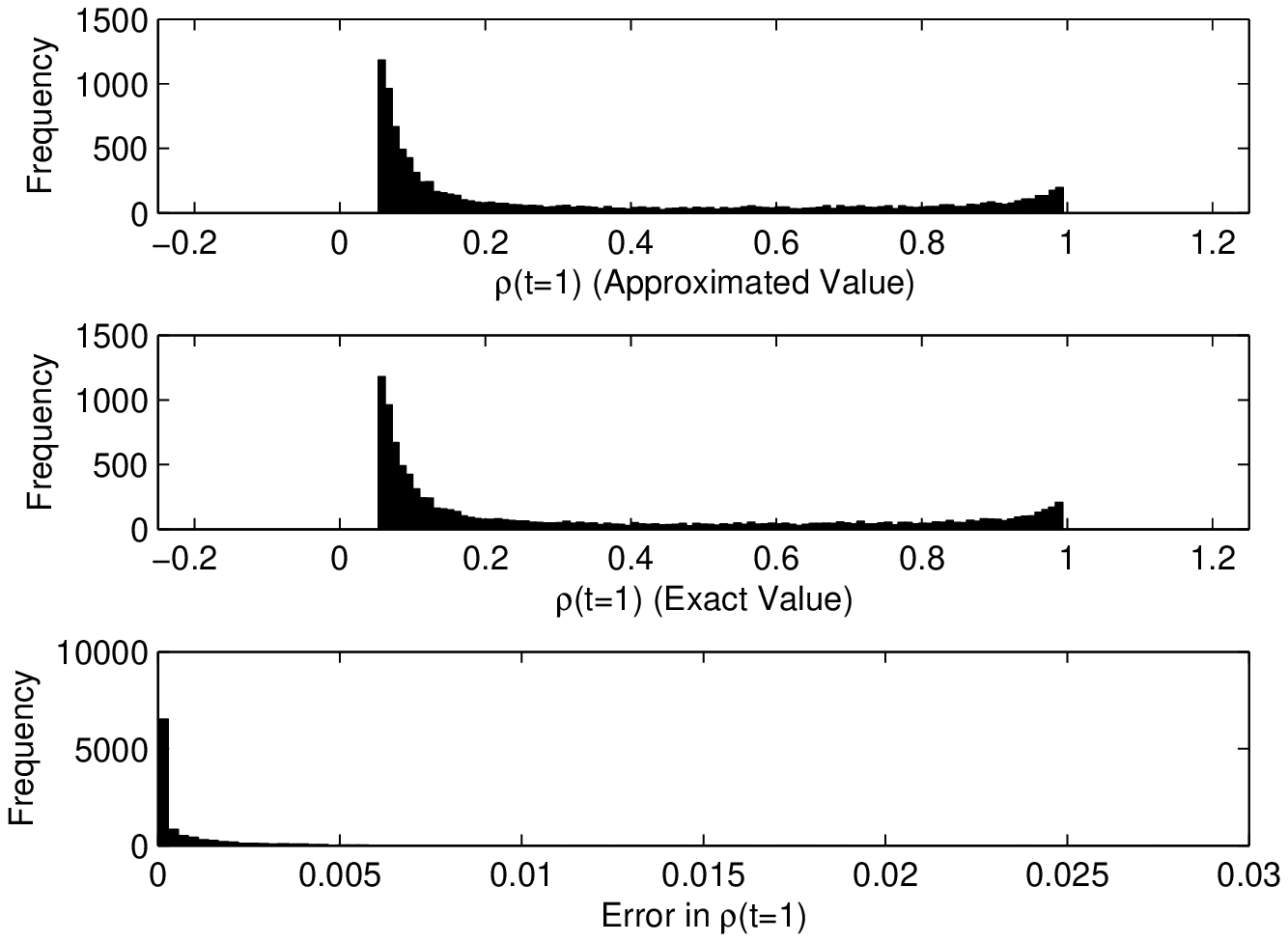}
\label{fig:subfig4}}

\caption{Distribution of function exact values evaluated by Monte Carlo simulation, function approximated evaluated by proposed method, and absolute value of error between them for $n=2, n_0=3, \mbox{ and } n_s=14.$}
\label{fig:globfig}
\end{figure}

\begin{figure}
\subfloat[][$C=5\times 10^{-2}$]{
\includegraphics[width=0.39\textwidth]{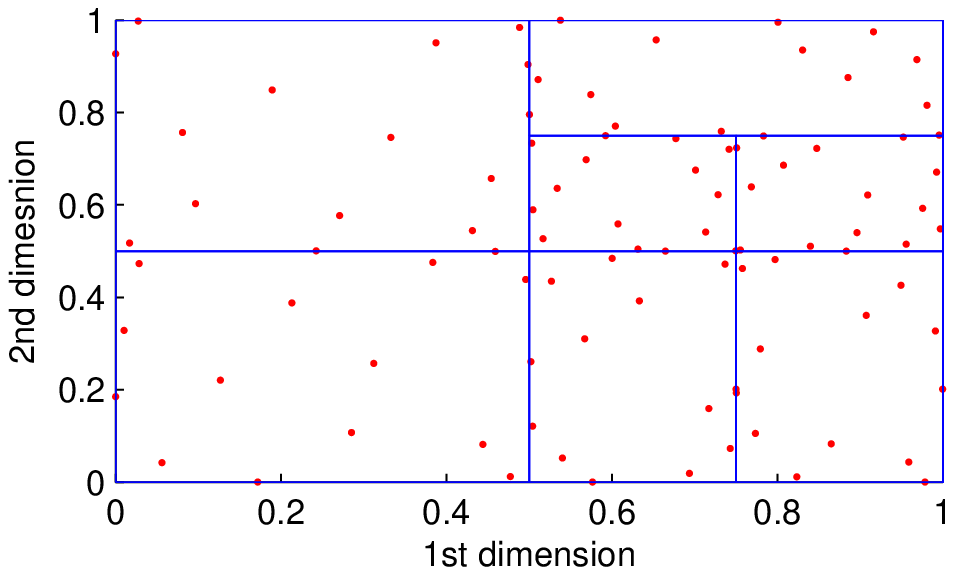}
\label{fig:subfig1}}
\qquad 
\subfloat[][$C=1\times 10^{-2}$]{
\includegraphics[width=0.39\textwidth]{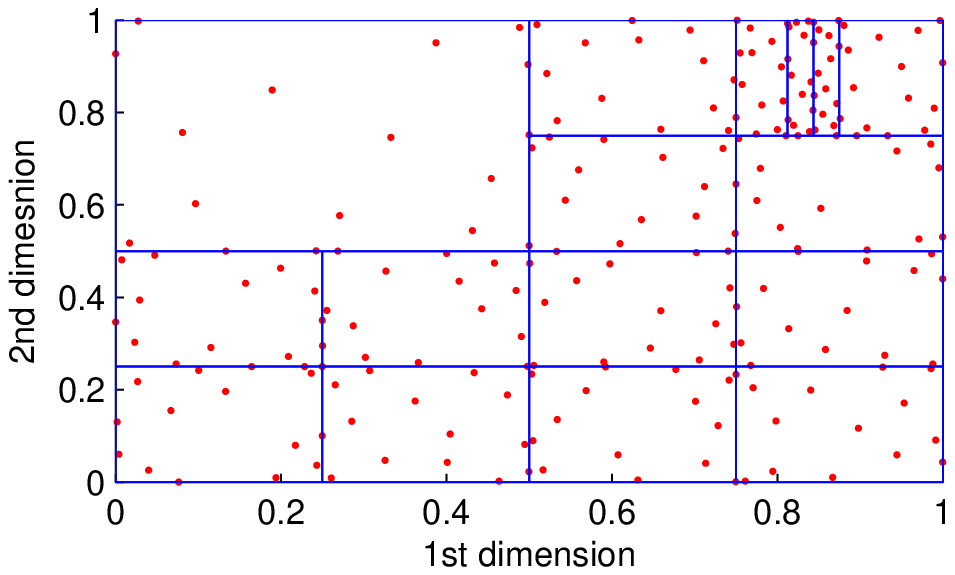}
\label{fig:subfig2}}
\\

\subfloat[][$C=5\times 10^{-3}$]{
\includegraphics[width=0.39\textwidth]{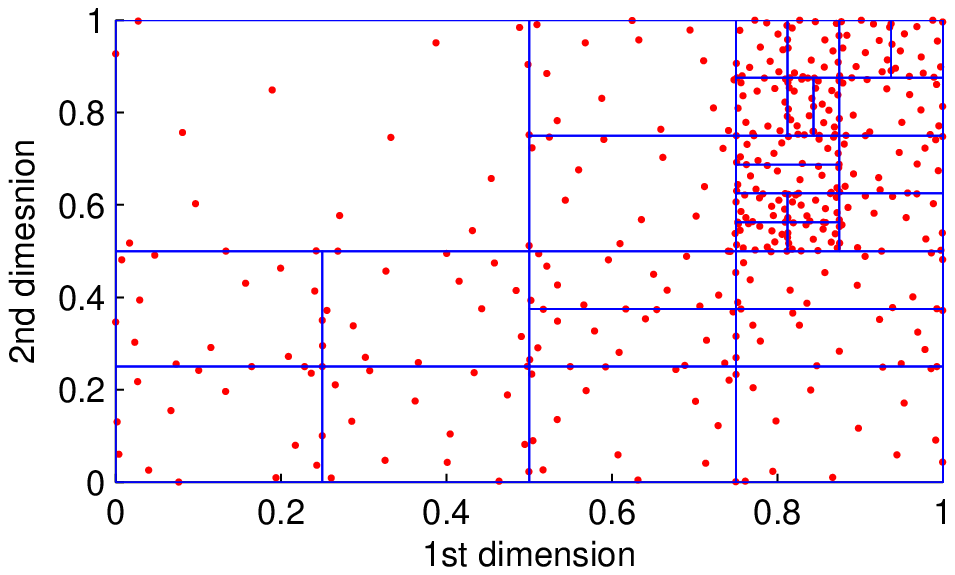}
\label{fig:subfig3}}
\qquad
\subfloat[][$C=1\times 10^{-3}$]{
\includegraphics[width=0.39\textwidth]{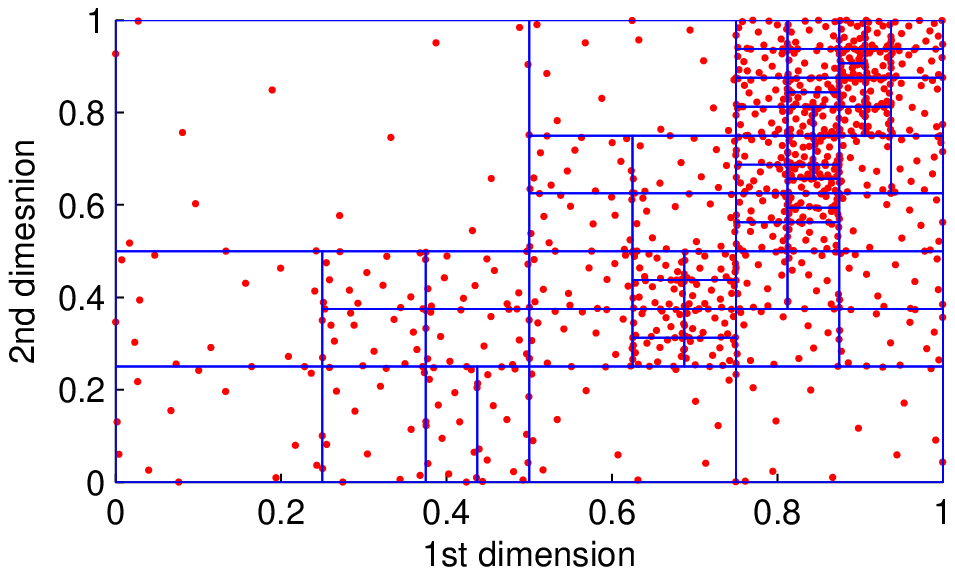}
\label{fig:subfig4}}
\\

\subfloat[][$C=5\times 10^{-4}$]{
\includegraphics[width=0.39\textwidth]{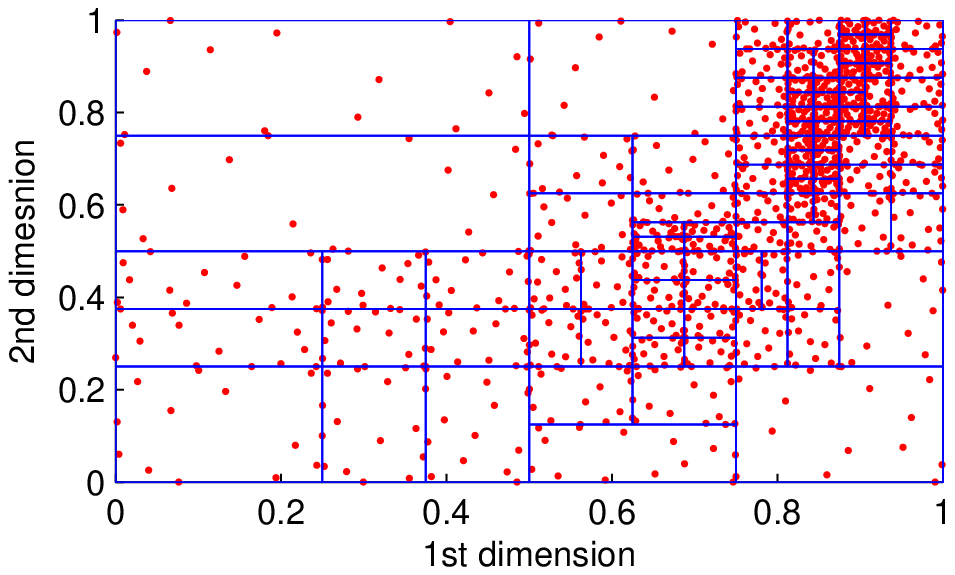}
\label{fig:subfig4}}

\caption{Partition of the random parameter space with sample points utilized for $n=2, n_0=3, \mbox{ and } n_s=14.$}
\label{fig:globfig}
\end{figure}

\begin{figure}
\subfloat[][$C=5\times 10^{-2}$]{
\includegraphics[width=0.5\textwidth]{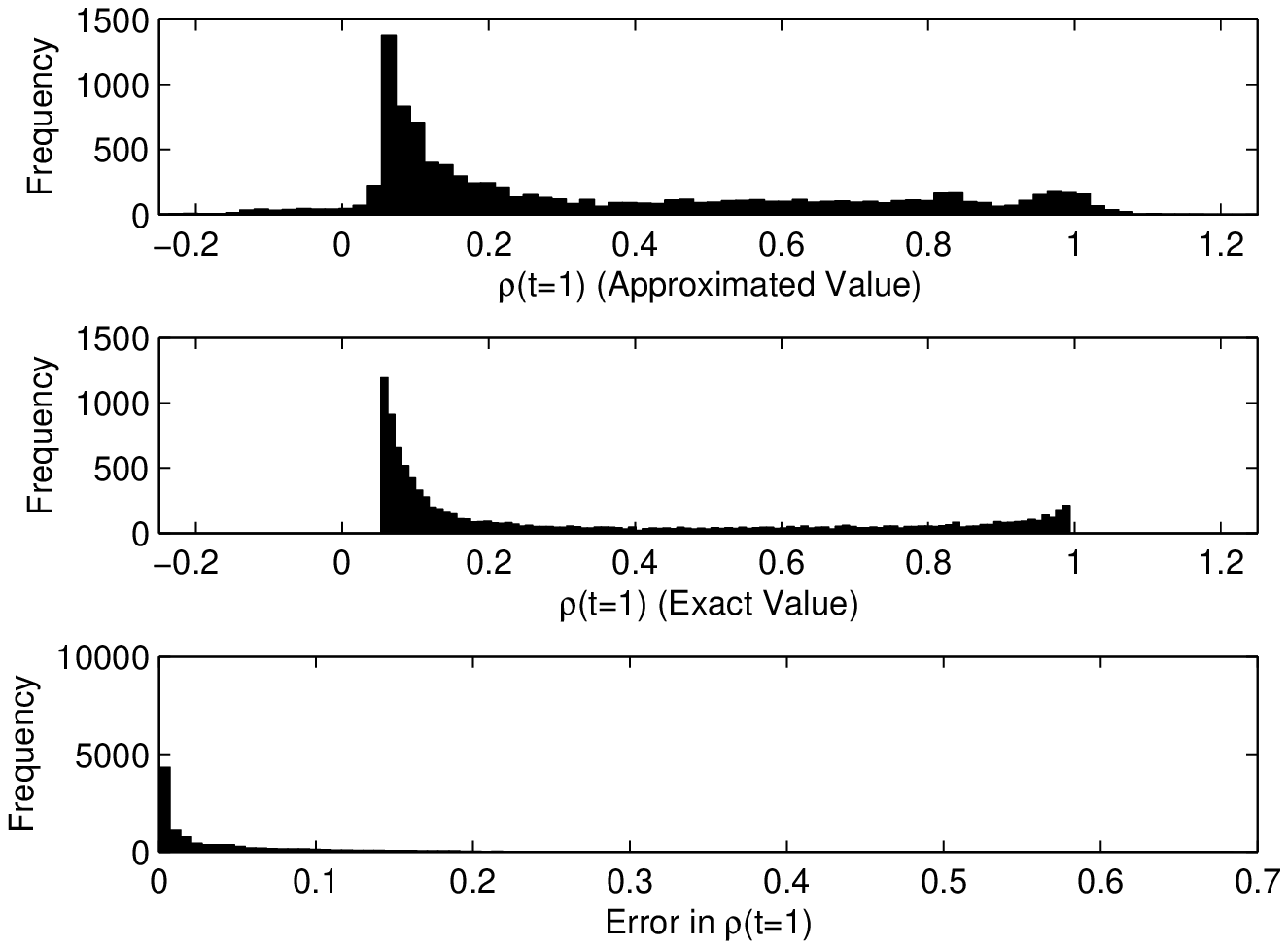}
\label{fig:subfig1}}
\qquad 
\subfloat[][$C=1\times 10^{-2}$]{
\includegraphics[width=0.5\textwidth]{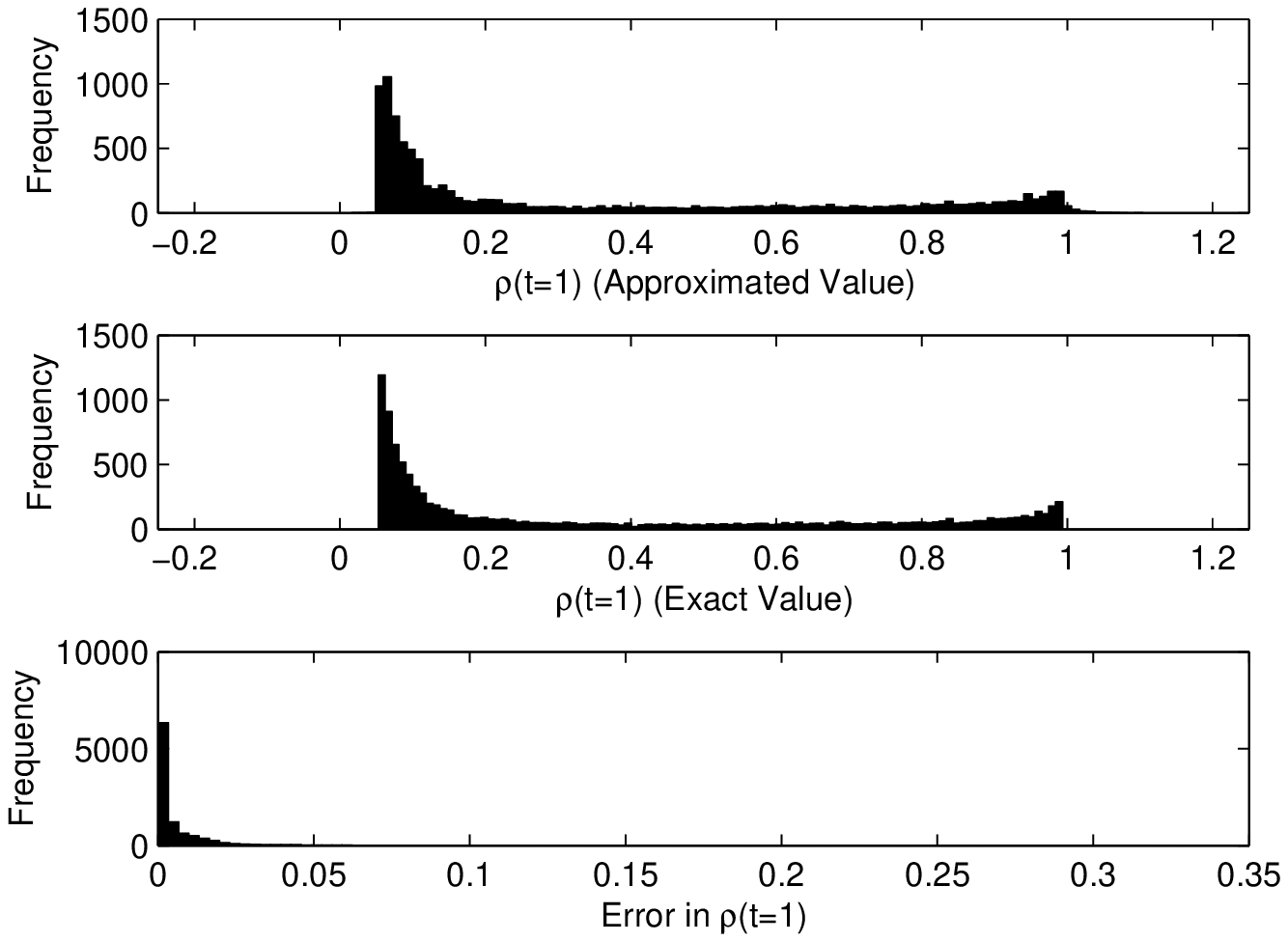}
\label{fig:subfig2}}
\\

\subfloat[][$C=5\times 10^{-3}$]{
\includegraphics[width=0.5\textwidth]{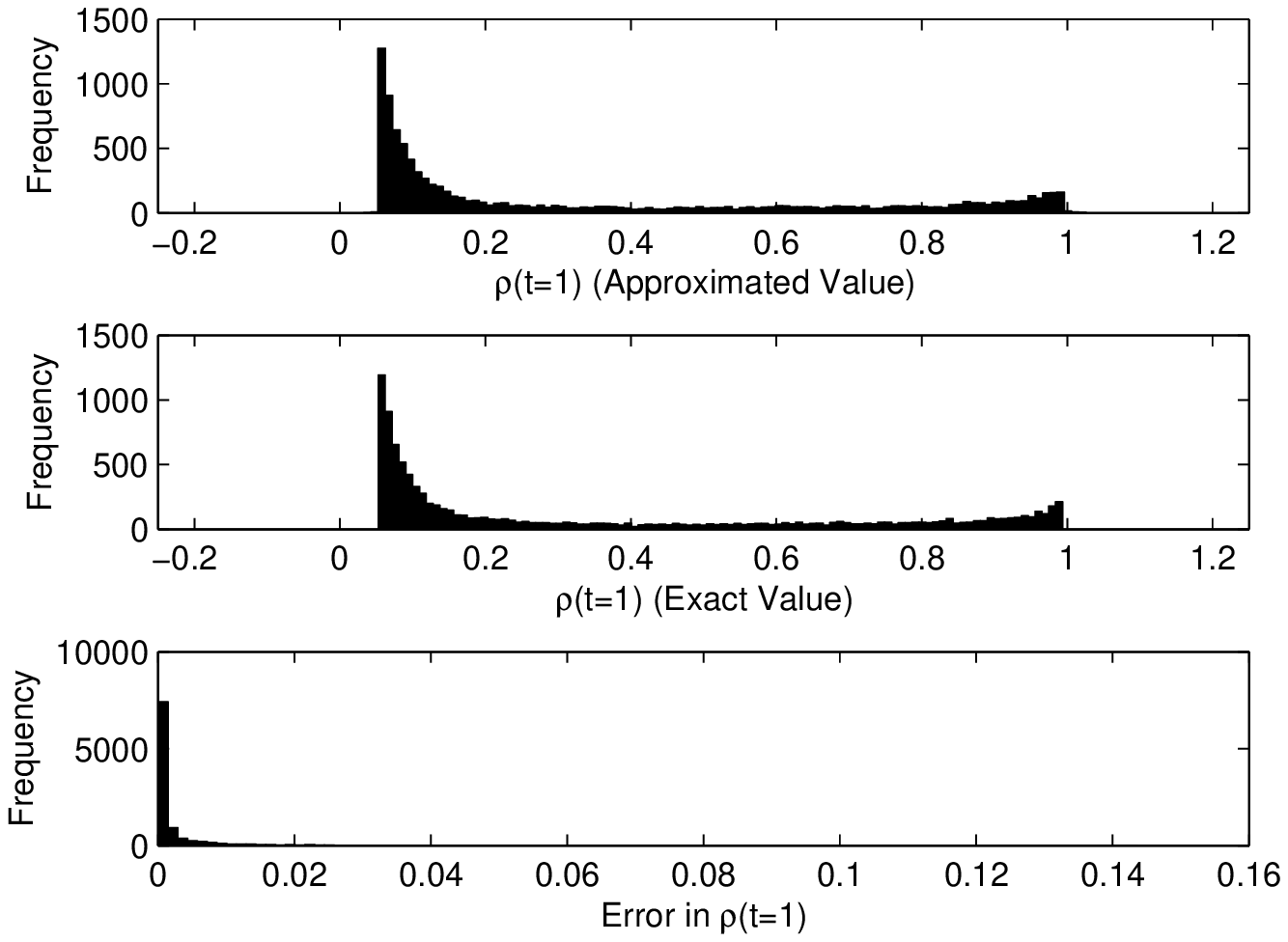}
\label{fig:subfig3}}
\qquad
\subfloat[][$C=1\times 10^{-3}$]{
\includegraphics[width=0.5\textwidth]{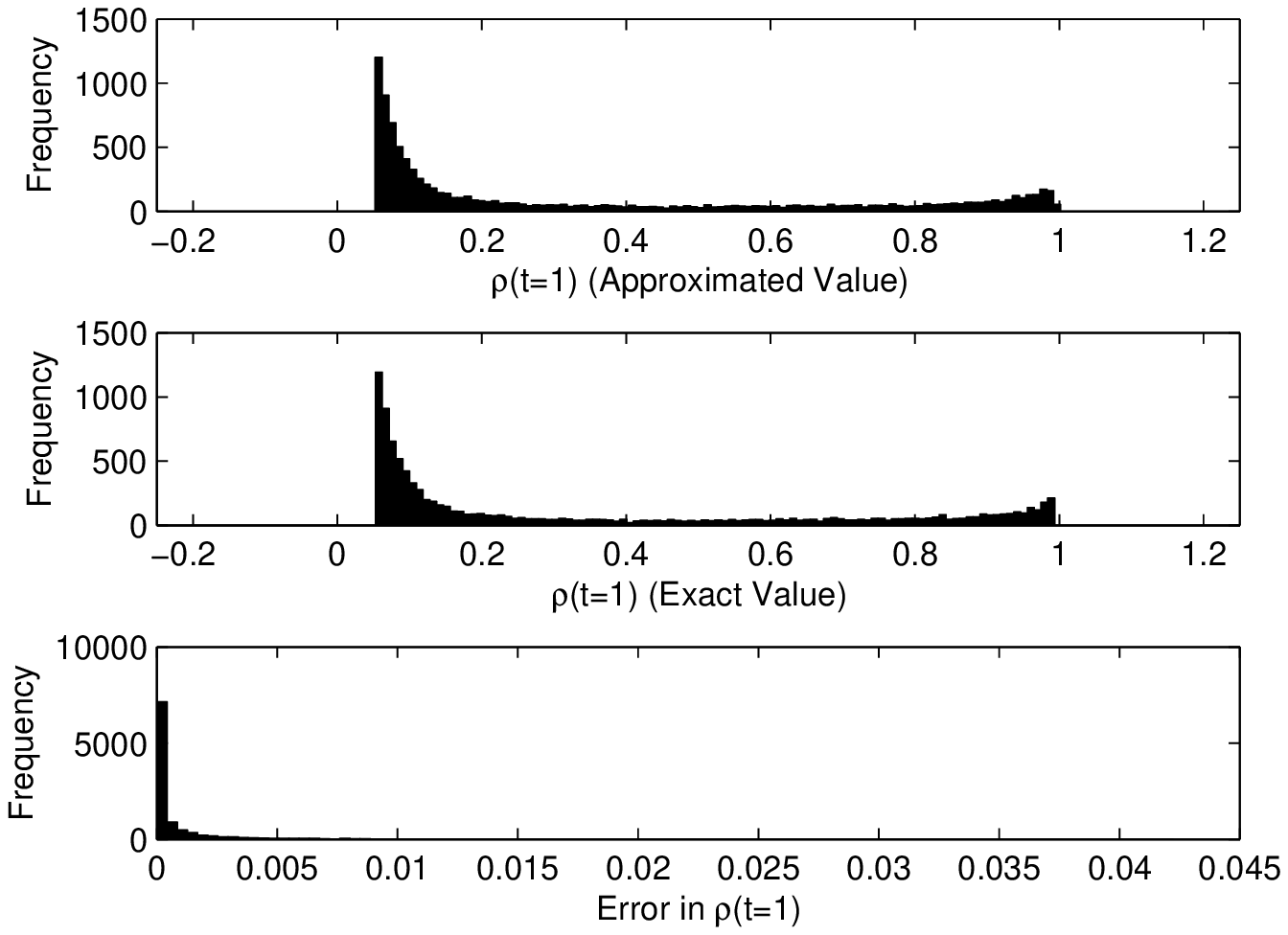}
\label{fig:subfig4}}
\\

\subfloat[][$C=5\times 10^{-4}$]{
\includegraphics[width=0.5\textwidth]{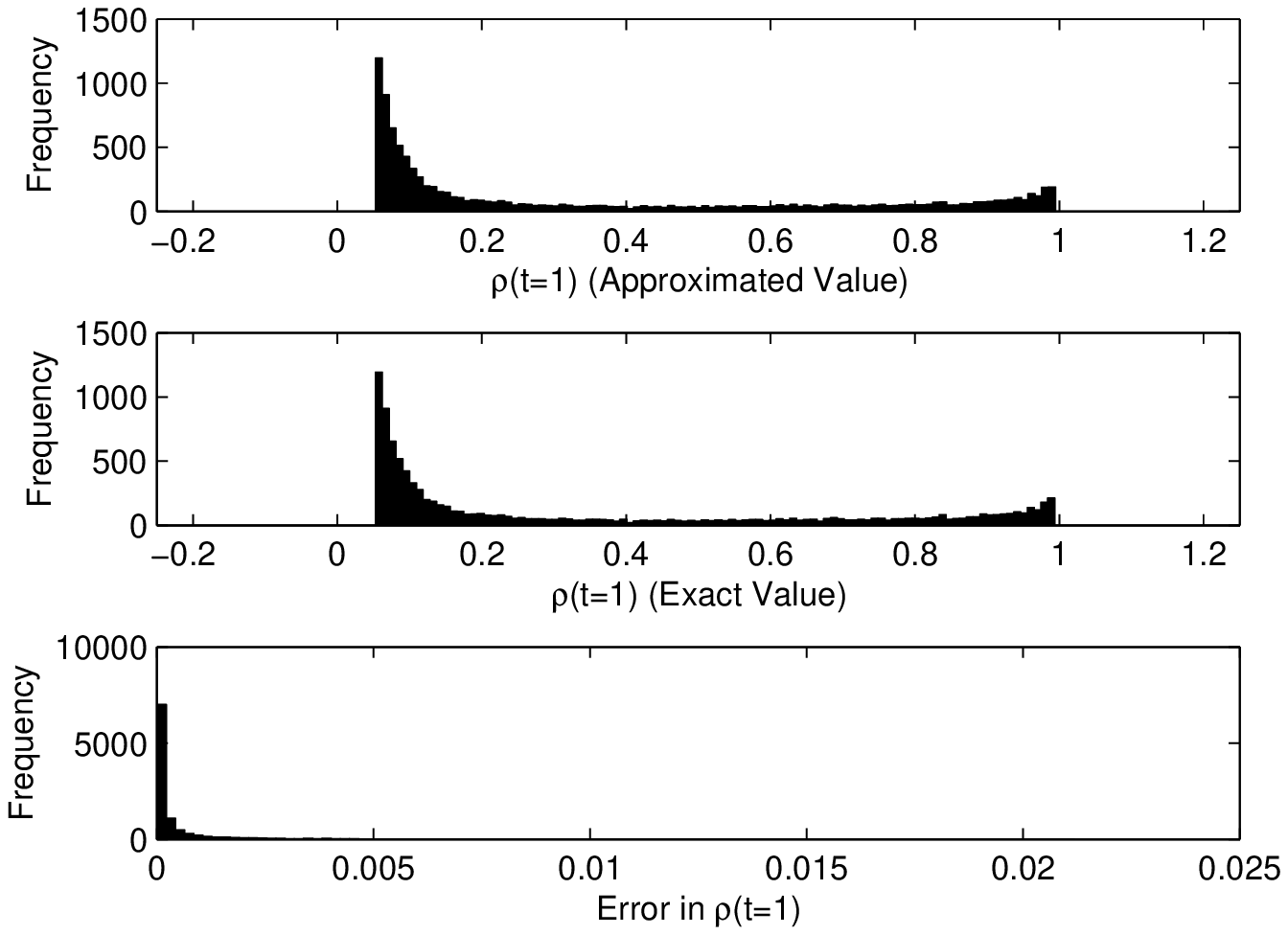}
\label{fig:subfig4}}

\caption{Distribution of function exact values evaluated by Monte Carlo simulation, function approximated evaluated by proposed method, and absolute value of error between them for $n=2, n_0=4, \mbox{ and } n_s=20.$}
\label{fig:globfig}
\end{figure}

\begin{figure}
\subfloat[][$C=5\times 10^{-2}$]{
\includegraphics[width=0.39\textwidth]{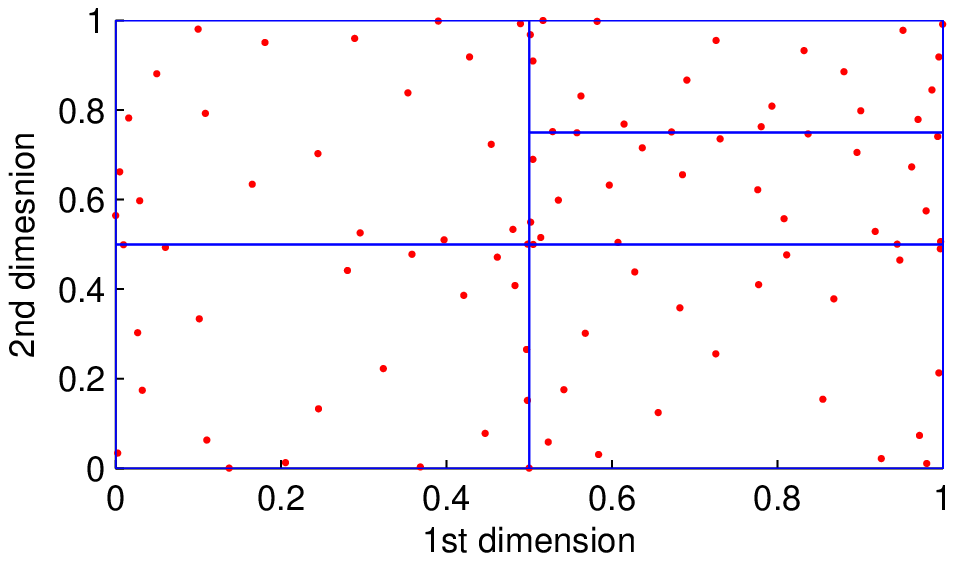}
\label{fig:subfig1}}
\qquad 
\subfloat[][$C=1\times 10^{-2}$]{
\includegraphics[width=0.39\textwidth]{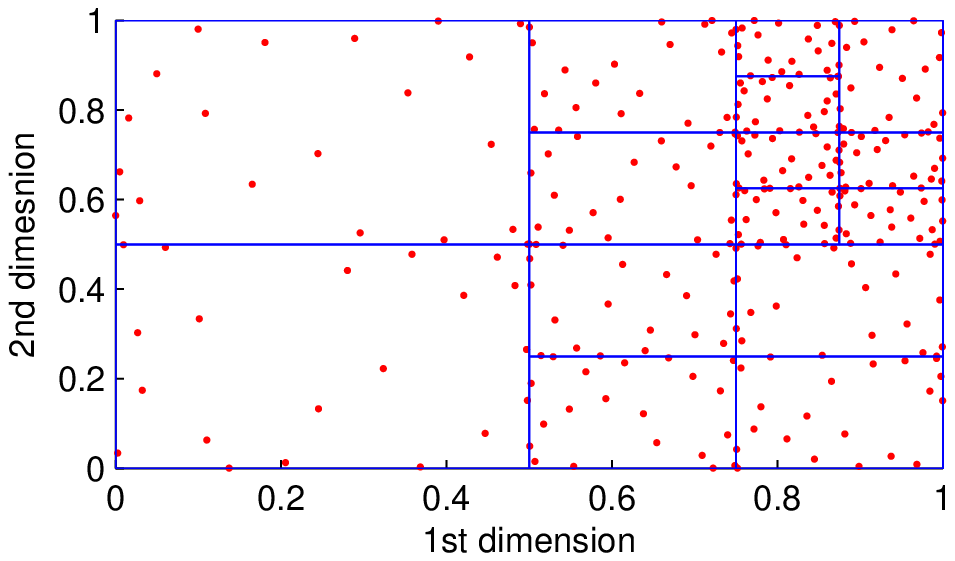}
\label{fig:subfig2}}
\\

\subfloat[][$C=5\times 10^{-3}$]{
\includegraphics[width=0.39\textwidth]{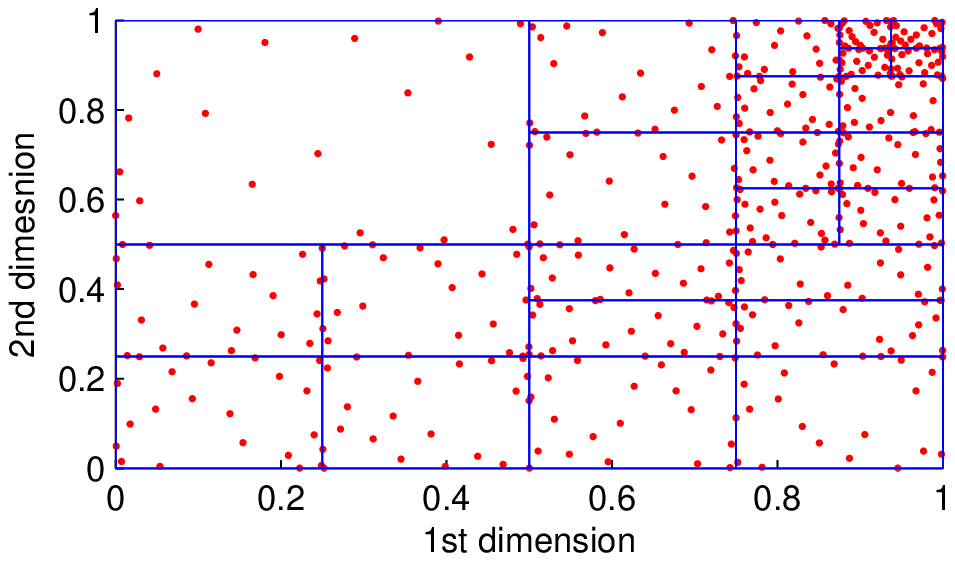}
\label{fig:subfig3}}
\qquad
\subfloat[][$C=1\times 10^{-3}$]{
\includegraphics[width=0.39\textwidth]{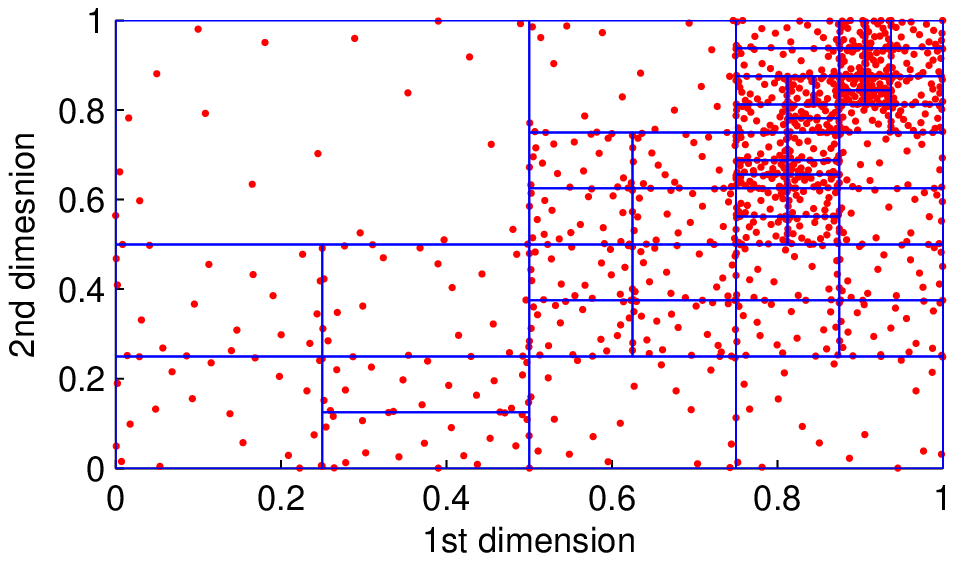}
\label{fig:subfig4}}
\\

\subfloat[][$C=5\times 10^{-4}$]{
\includegraphics[width=0.39\textwidth]{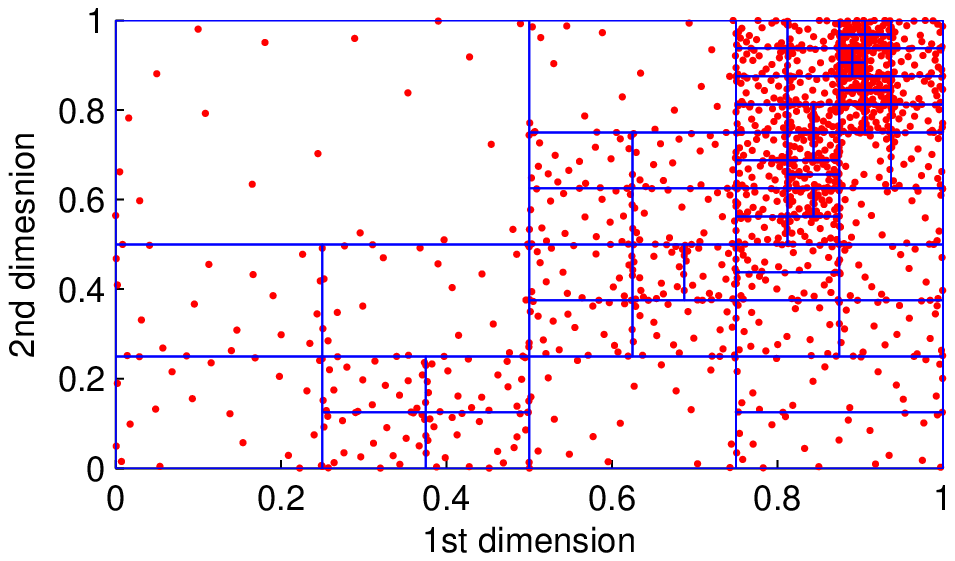}
\label{fig:subfig4}}

\caption{Partition of the random parameter space with sample points utilized for $n=2, n_0=4, \mbox{ and } n_s=20.$}
\label{fig:globfig}
\end{figure}

\begin{figure}
\subfloat[][Relative erros of variance]{
\includegraphics[width=0.45\textwidth]{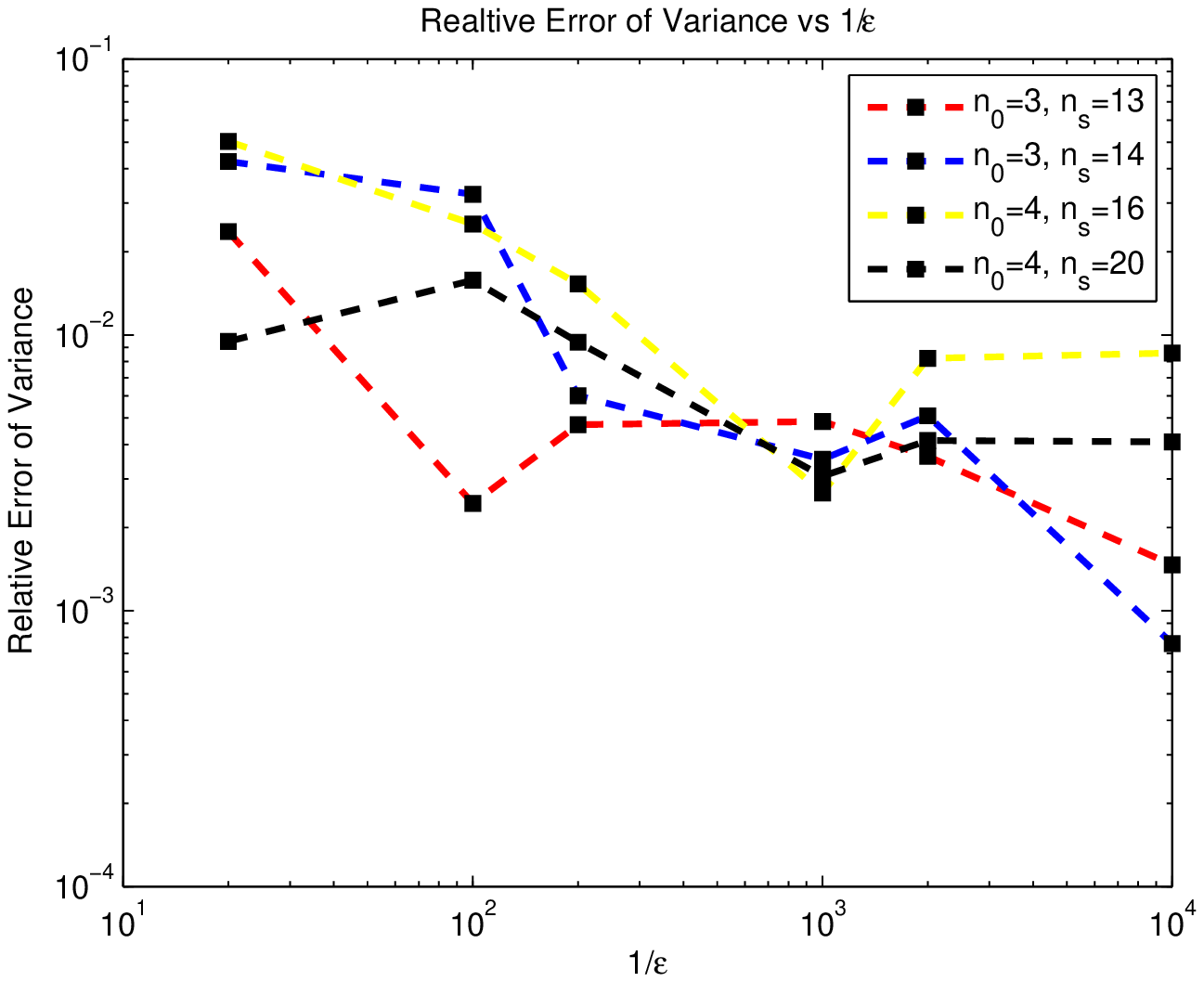}
\label{fig:subfig1}}
\qquad 
\subfloat[][Relative erros of mean]{
\includegraphics[width=0.45\textwidth]{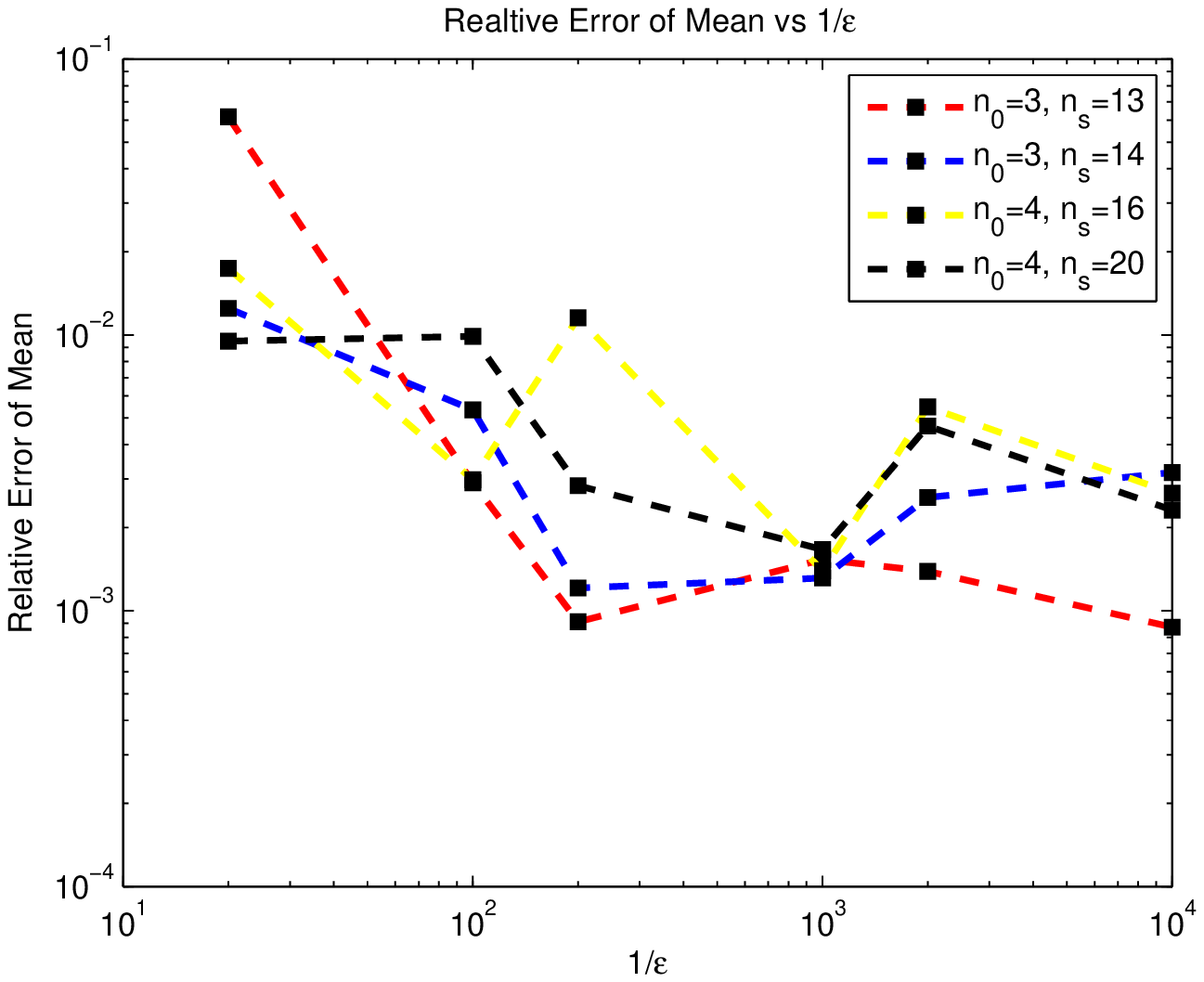}
\label{fig:subfig2}}
\\

\subfloat[][Error values (MSE)]{
\includegraphics[width=0.45\textwidth]{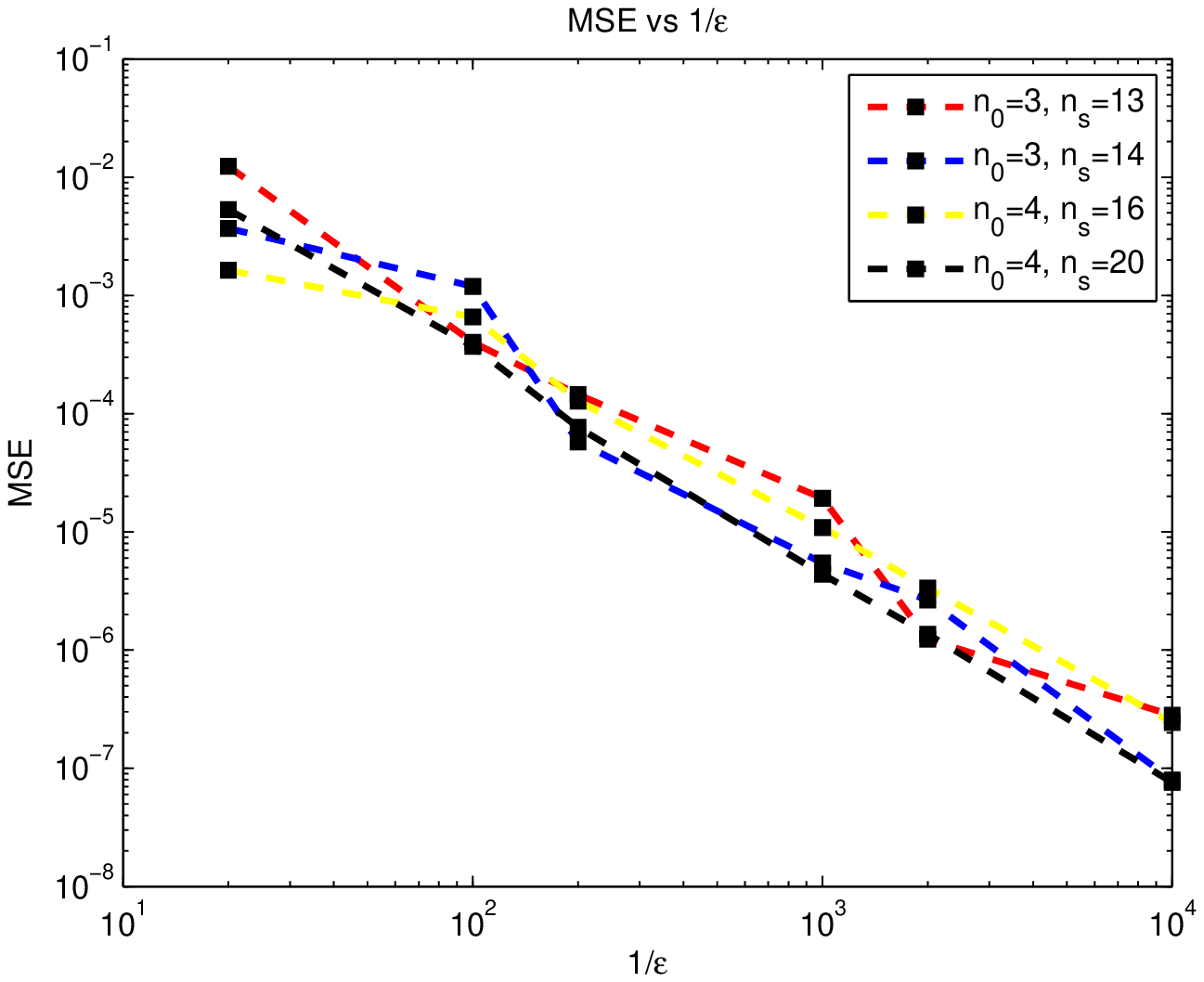}
\label{fig:subfig3}}

\caption{Relative error values of both mean and variance evaluated by Monte Carlo simulation and proposed method, and error (MSE) for all cases.}
\label{fig:globfig}
\end{figure}

The results show that by decreasing $C$, the value of mean and variance converge to its exact value evaluated from Monte Carlo simulation. Moreover, it can be easily seen that in the proposed method more refinements is concentrated in regions where the solution has steepest dependence on random inputs. Indeed, it truly identifies the regions where the solution lacks sufficient regularity and smoothness according to the construction of the adaptive strategy. Consequently, by decreasing the value of $C$, the number of partition should be increased.

There is a fluctuation between the error and the other two parameters, Number of samples $n_s$ and polynomial order $n_0$, which explains one should consider a tradeoff between the two mentioned parameters to optimize CPU time and computational cost with predefined error. In fact, it is not a good idea to continuously increase the polynomial order to reduce the error as it increases the computational costs and time. 

\subsection{Diffusion Problem with Variable Coefficient}
Consider the linear PDE
\begin{equation}
\left\{
	\begin{array}{ll}
\frac{\partial}{\partial x} (a(x,y)\frac{\partial}{\partial x}u(x,y))=1 \hspace{8pt} x\in (0,1)\\
u(0,y)=u(1,y)=0
\end{array}
\right.
\end{equation}

where $y$ is a uniform random variable on  $(-1,1-\varepsilon)$. Here, scalar $\varepsilon$ controls the variability in the diffusion coefficient $a(x,y)$, and thus the solution $u(x,y)$, and is taken to be $\varepsilon=0.02$ in this problem. We consider the following model for $a(x,y)$ as:

\begin{equation}
a(x,y)=1+4y(x^2-x)
\end{equation}

The solution to this problem exhibits rather sharp gradients close to $x=0.5$ and $y=1-\varepsilon$.

\begin{figure}[h!]
\includegraphics[width=4in]{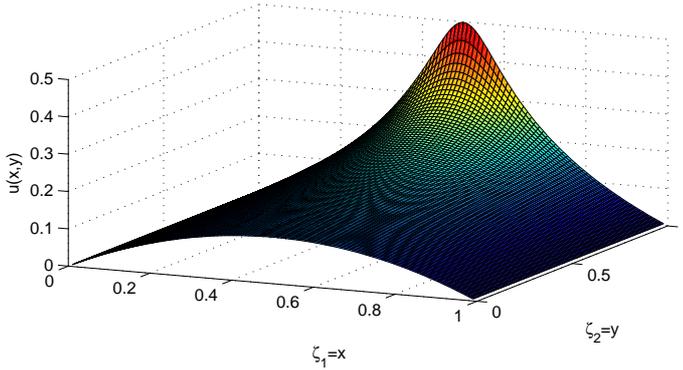}
\label{fig:subfig1}
\caption{Surface response of $u (0.5,y)$)}
\end{figure}

\begin{table}\footnotesize
\caption{Evaluated values for $N_{sb}$, $\mathbb{E}(u(0.5))$, $\sigma^2(u(0.5))$ for different values of  $n_0$, $n_s$, and $C$.}
\medskip
\centering  
\subfloat[][Evaluated values of $N_{sb}$, $\mathbb{E}(u(0.5))$, $\sigma^2(u(0.5))$, and MSE for $n_0=3$, $N_s=5$.]{\begin{tabular}{c c c c c} 
 \hline            
 
 $C$ & $N_{sb}$ & $\mathbb{E}(\rho)$& $\sigma^2(\rho)$ & MSE \\ [0.5ex] 
 \hline                  
 $5\times 10 ^{-2}$ & 2 &  0.19542594 & 0.00538135 & 0.32841073 $\times 10^{-4}$ \\ 
 $1\times 10^{-2}$ & 3 & 0.19651871 & 0.00543813 & 0.01670120 $\times 10^{-4}$\\
 $5\times 10^{-3}$ & 4 & 0.19642843 & 0.00543813 & 0.00039301 $\times 10^{-4}$\\
 $1\times 10^{-3}$ & 7 & 0.19637294 & 0.00543225 & 0.00030281 $\times 10^{-4}$\\
 $5\times 10^{-4}$ & 7 & 0.19637295 & 0.00543225 & 0.00028144 $\times 10^{-4}$\\
 $1\times 10^{-4}$ & 10 &  0.19637290 & 0.00543221 & 0.00026072 $\times 10^{-4}$ \\ [1ex]      
 
 \hline 
 
 \end{tabular}}

\bigskip

\subfloat[][Evaluated values of $N_{sb}$, $\mathbb{E}(u(0.5))$, $\sigma^2(u(0.5))$, and MSE for $n_0=3$, $N_s=6$.]{\begin{tabular}{c c c c c} 
 \hline             
 
 $C$ & $N_{sb}$ & $\mathbb{E}(\rho)$& $\sigma^2(\rho)$ & MSE \\ [0.5ex] 
 \hline                  
 $5\times 10 ^{-2}$ & 2 &  0.19664675 & 0.00561008 & 0.12122560 $\times 10^{-4}$ \\ 
 $1\times 10^{-2}$ & 4 & 0.19642821 & 0.00543813 & 0.00044188 $\times 10^{-4}$\\
 $5\times 10^{-3}$ & 4 & 0.19642843 & 0.00543813 & 0.00039301 $\times 10^{-4}$\\
 $1\times 10^{-3}$ & 5 & 0.19642686 & 0.00543363 & 0.00002203 $\times 10^{-4}$\\
 $5\times 10^{-4}$ & 5 & 0.19642708 & 0.00543364 & 0.00002281 $\times 10^{-4}$\\
 $1\times 10^{-4}$ & 9 &  0.19642435 & 0.00543546 & 0.00001376 $\times 10^{-4}$ \\ [1ex]      
 
 \hline 
 \end{tabular}}

\bigskip

\subfloat[][Evaluated values of $N_{sb}$, $\mathbb{E}(u(0.5))$, $\sigma^2(u(0.5))$, and MSE for $n_0=4$, $N_s=6$.]{\begin{tabular}{c c c c c} 
 \hline             
 
 $C$ & $N_{sb}$ & $\mathbb{E}(\rho)$& $\sigma^2(\rho)$ & MSE \\ [0.5ex] 
 \hline                  
 $5\times 10 ^{-2}$ & 1 &  0.19813695 & 0.00608860 & 0.36032433 $\times 10^{-4}$ \\ 
 $1\times 10^{-2}$ & 2 & 0.19670716 & 0.00557323 & 0.03555089 $\times 10^{-4}$\\
 $5\times 10^{-3}$ & 3 & 0.19643187 & 0.00543439 & 0.00495663 $\times 10^{-4}$\\
 $1\times 10^{-3}$ & 6 & 0.19642224 & 0.00543427 & 0.00021754 $\times 10^{-4}$\\
 $5\times 10^{-4}$ & 8 & 0.19642323 & 0.00543529 & 0.00002544 $\times 10^{-4}$\\
 $1\times 10^{-4}$ & 9 &  0.19642150 & 0.00543427 & 0.00003535 $\times 10^{-4}$ \\ [1ex]      
 
 \hline 
 
 \end{tabular}}
 
 \bigskip
 
 \subfloat[][Evaluated values of $N_{sb}$, $\mathbb{E}(u(0.5))$, $\sigma^2(u(0.5))$, and MSE for $n_0=3$, $N_s=7$.]{\begin{tabular}{c c c c c} 
  \hline\hline             
  
  $C$ & $N_{sb}$ & $\mathbb{E}(\rho)$& $\sigma^2(\rho)$ & MSE \\ [0.5ex] 
  \hline                  
  $5\times 10 ^{-2}$ & 2 &  0.19614737 & 0.00545924 & 0.70143879 $\times 10^{-5}$ \\ 
  $1\times 10^{-2}$ & 3 & 0.19646409 & 0.00545924 & 0.03969327 $\times 10^{-5}$\\
  $5\times 10^{-3}$ & 3 & 0.19646409 & 0.00545924 & 0.03981619 $\times 10^{-5}$\\
  $1\times 10^{-3}$ & 4 & 0.19643156 & 0.00543410 & 0.00020607 $\times 10^{-5}$\\
  $5\times 10^{-4}$ & 5 & 0.19642966 & 0.00543231 & 0.00037291 $\times 10^{-5}$\\
  $1\times 10^{-4}$ & 6 &  0.19642270 & 0.00543378 & 0.00005475 $\times 10^{-5}$ \\ [1ex]      
  
  \hline 
  
  \end{tabular}}
\end{table}

\begin{figure}
\subfloat[][$C=5\times 10^{-2}$]{
\includegraphics[width=0.5\textwidth]{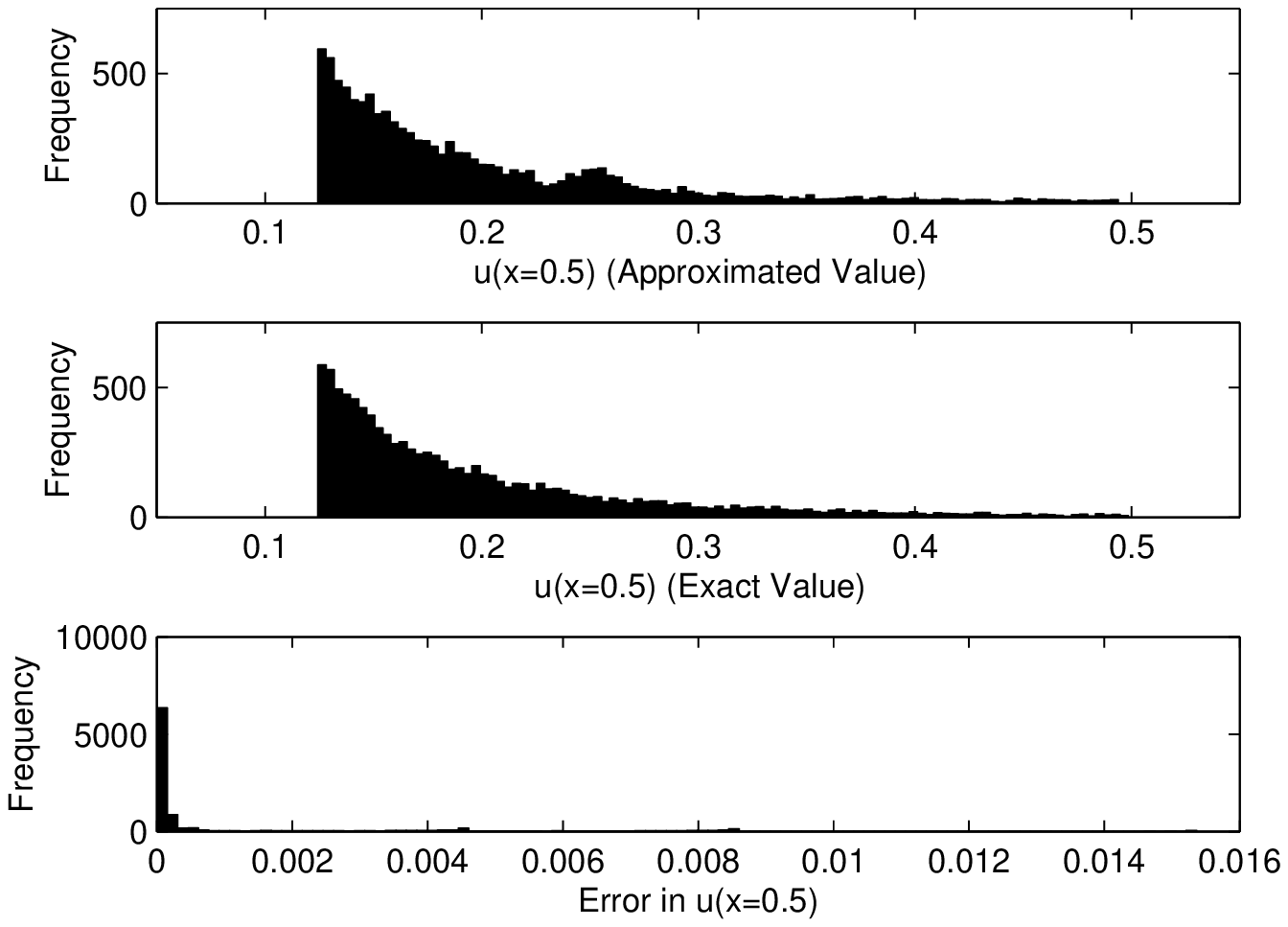}
\label{fig:subfig1}}
\qquad 
\subfloat[][$C=1\times 10^{-2}$]{
\includegraphics[width=0.5\textwidth]{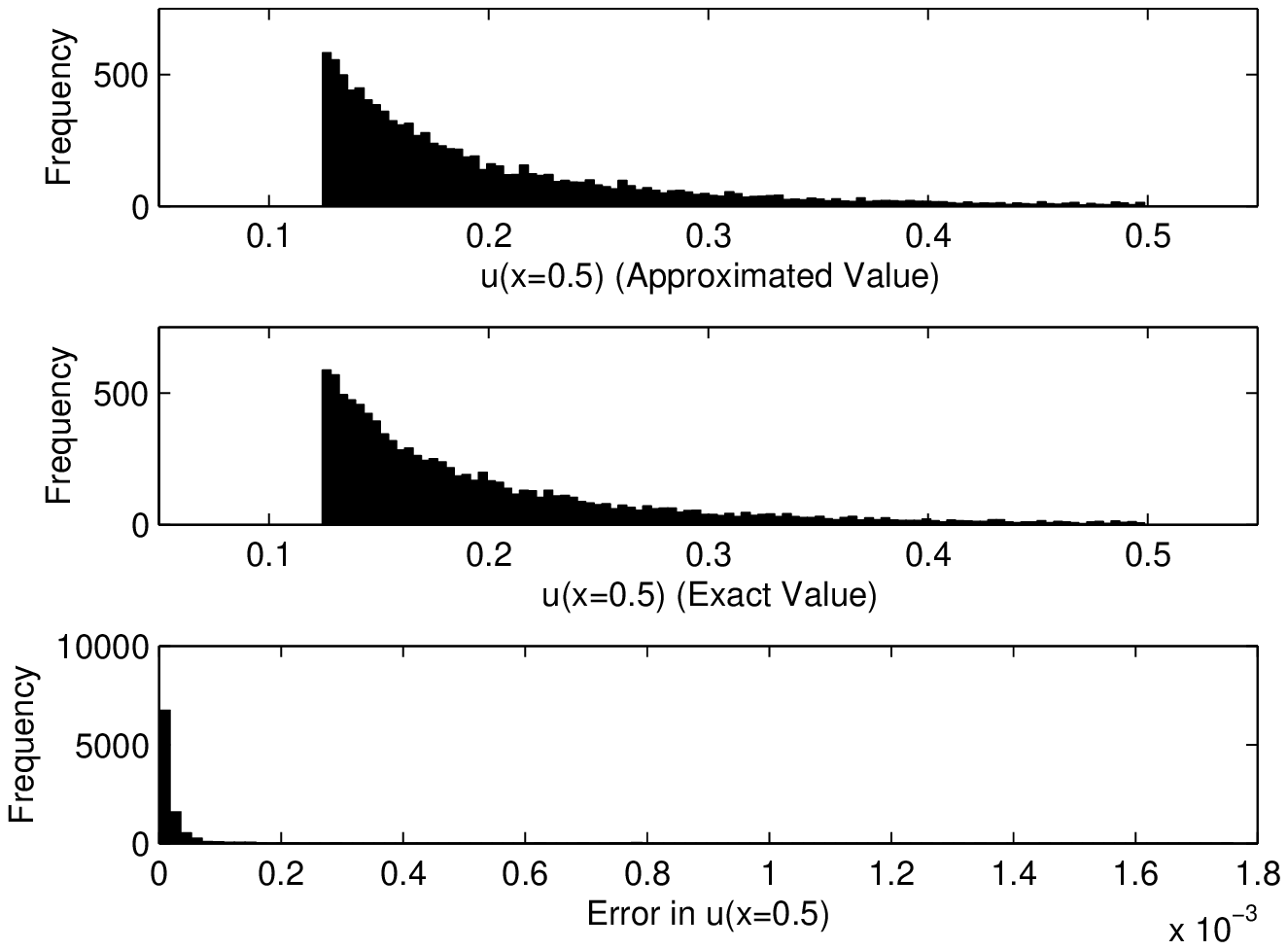}
\label{fig:subfig2}}
\\

\subfloat[][$C=5\times 10^{-3}$]{
\includegraphics[width=0.5\textwidth]{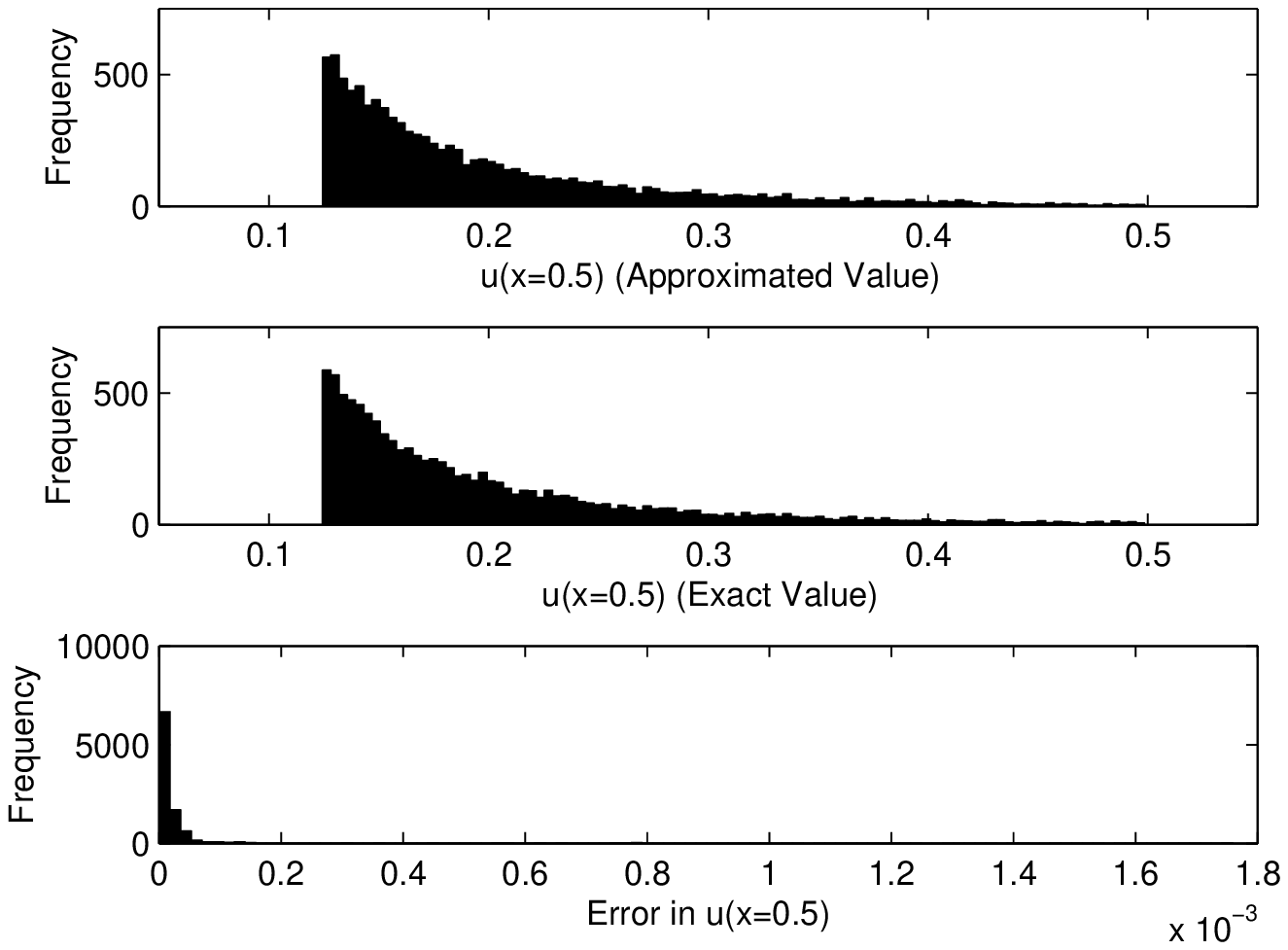}
\label{fig:subfig3}}
\qquad
\subfloat[][$C=1\times 10^{-3}$]{
\includegraphics[width=0.5\textwidth]{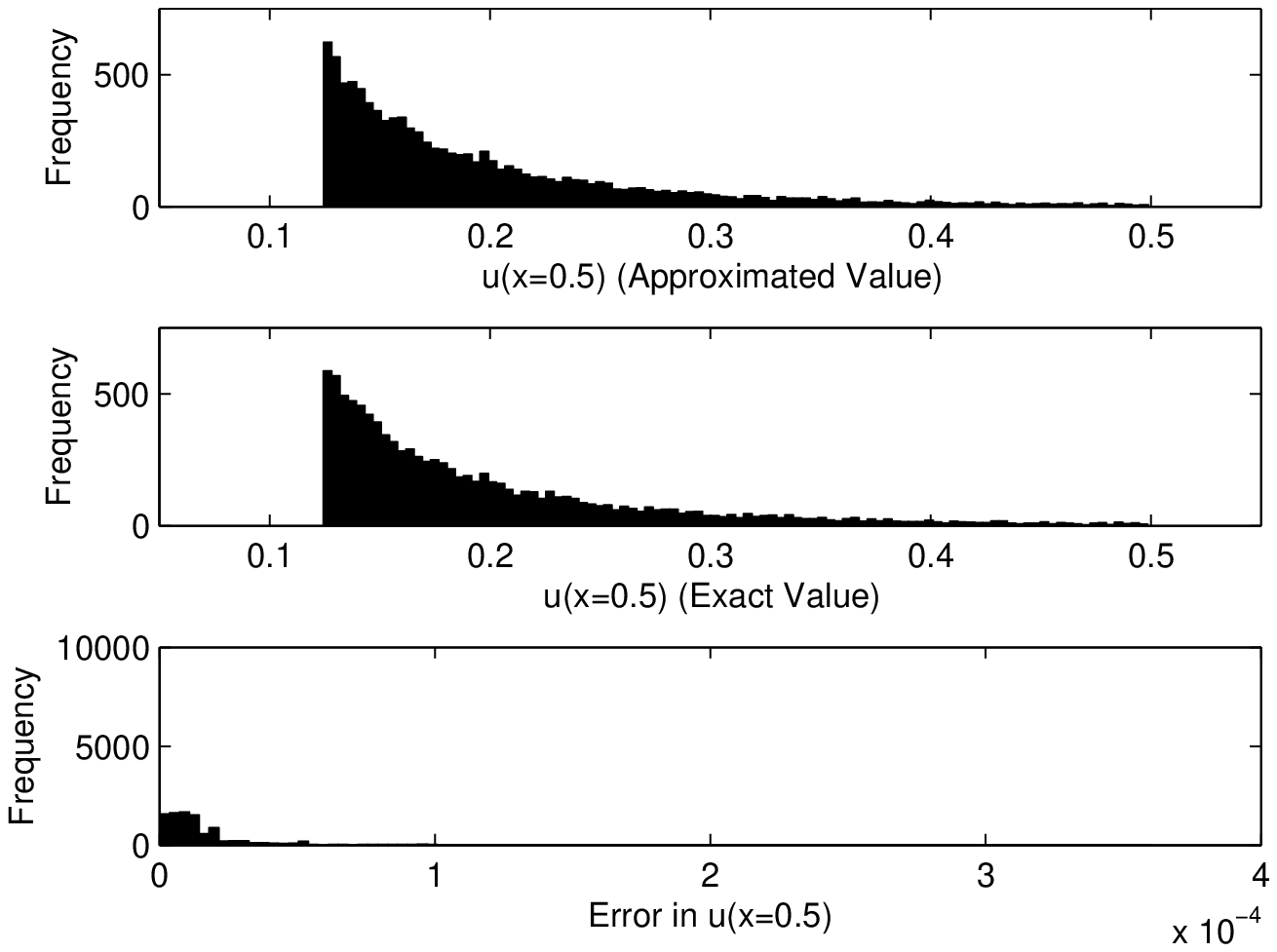}
\label{fig:subfig4}}
\\

\subfloat[][$C=5\times 10^{-4}$]{
\includegraphics[width=0.5\textwidth]{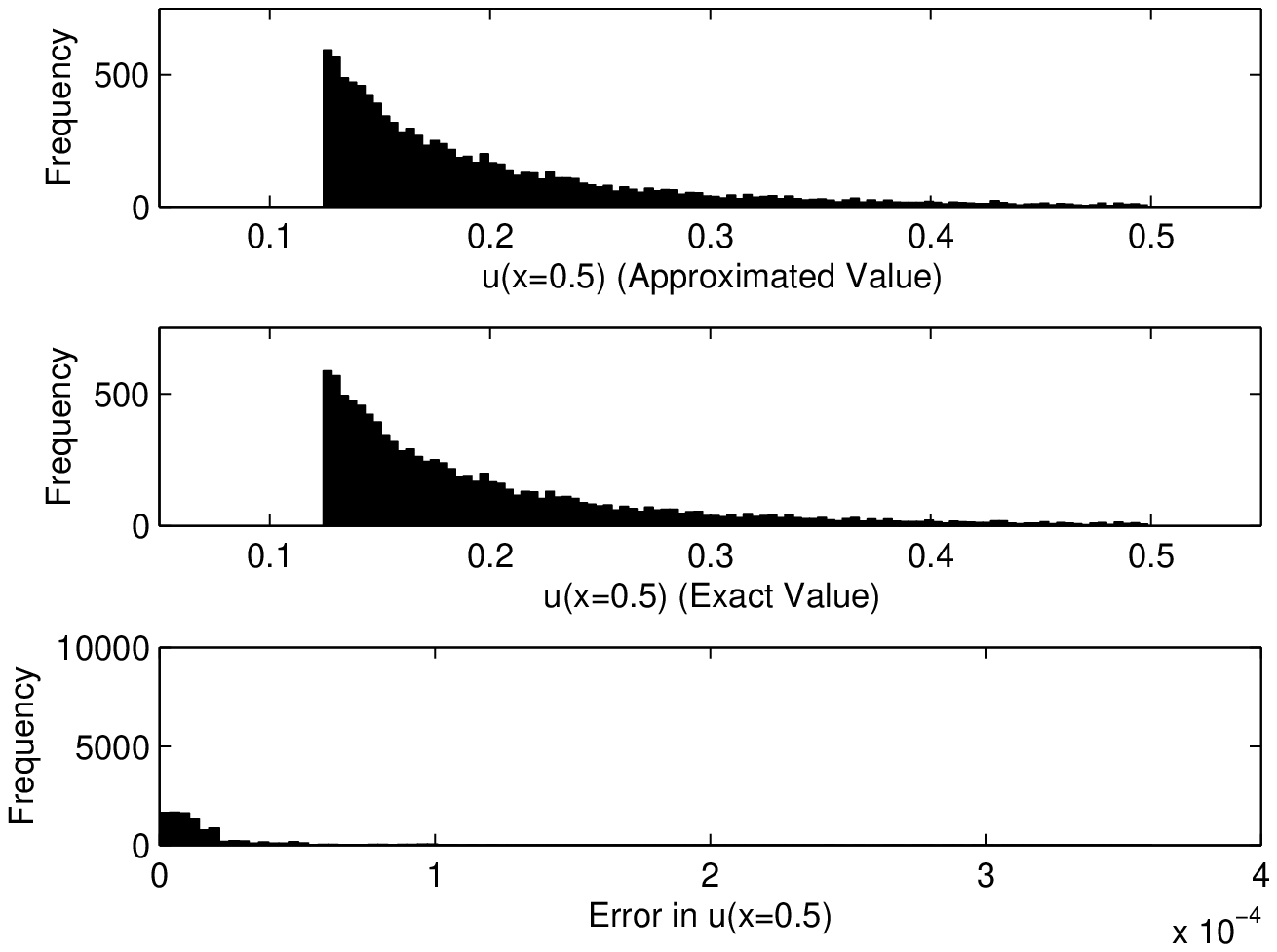}
\label{fig:subfig4}}

\caption{Distribution of function exact values evaluated by Monte Carlo simulation, function approximated evaluated by proposed method, and absolute value of error between them for $n=1, n_0=3, \mbox{ and } n_s=6.$}
\label{fig:globfig}
\end{figure}

\begin{figure}
\subfloat[][$C=5\times 10^{-2}$]{
\includegraphics[width=0.39\textwidth]{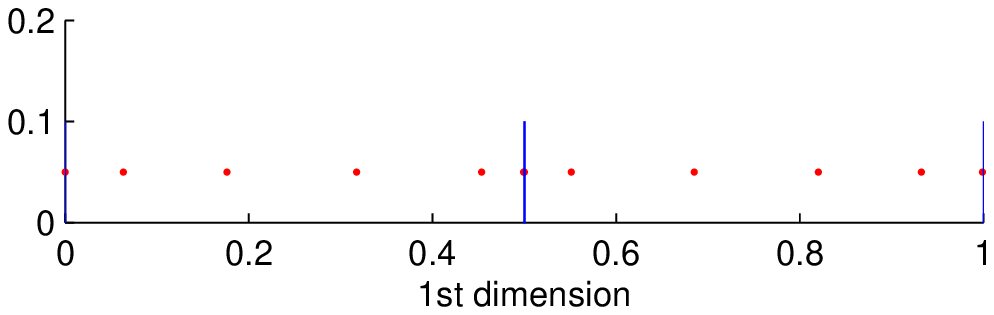}
\label{fig:subfig1}}
\qquad 
\subfloat[][$C=1\times 10^{-2}$]{
\includegraphics[width=0.39\textwidth]{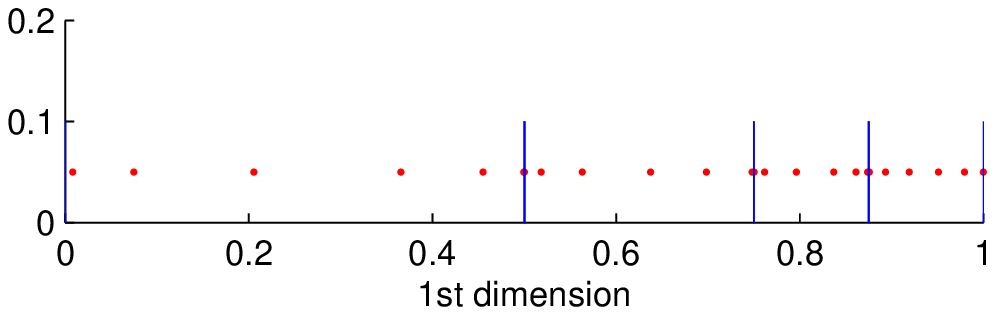}
\label{fig:subfig2}}
\\

\subfloat[][$C=5\times 10^{-3}$]{
\includegraphics[width=0.39\textwidth]{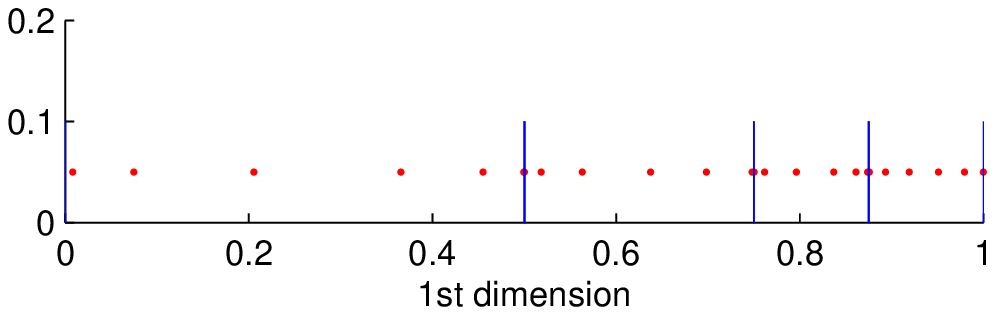}
\label{fig:subfig3}}
\qquad
\subfloat[][$C=1\times 10^{-3}$]{
\includegraphics[width=0.39\textwidth]{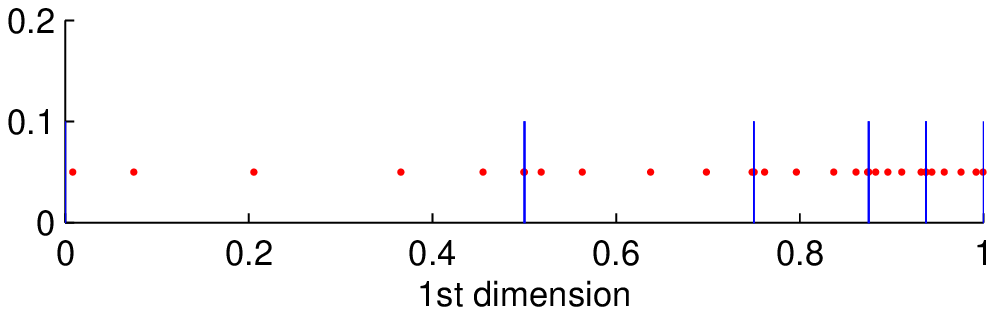}
\label{fig:subfig4}}
\\

\subfloat[][$C=5\times 10^{-4}$]{
\includegraphics[width=0.39\textwidth]{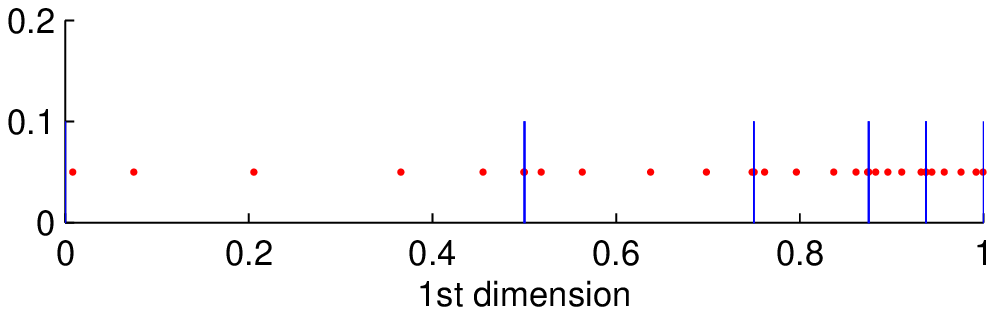}
\label{fig:subfig4}}
\caption{Partition of the random parameter space with sample points utilized for $n=1, n_0=3, \mbox{ and } n_s=6.$}
\label{fig:globfig}
\end{figure}

\begin{figure}
\subfloat[][$C=5\times 10^{-2}$]{
\includegraphics[width=0.5\textwidth]{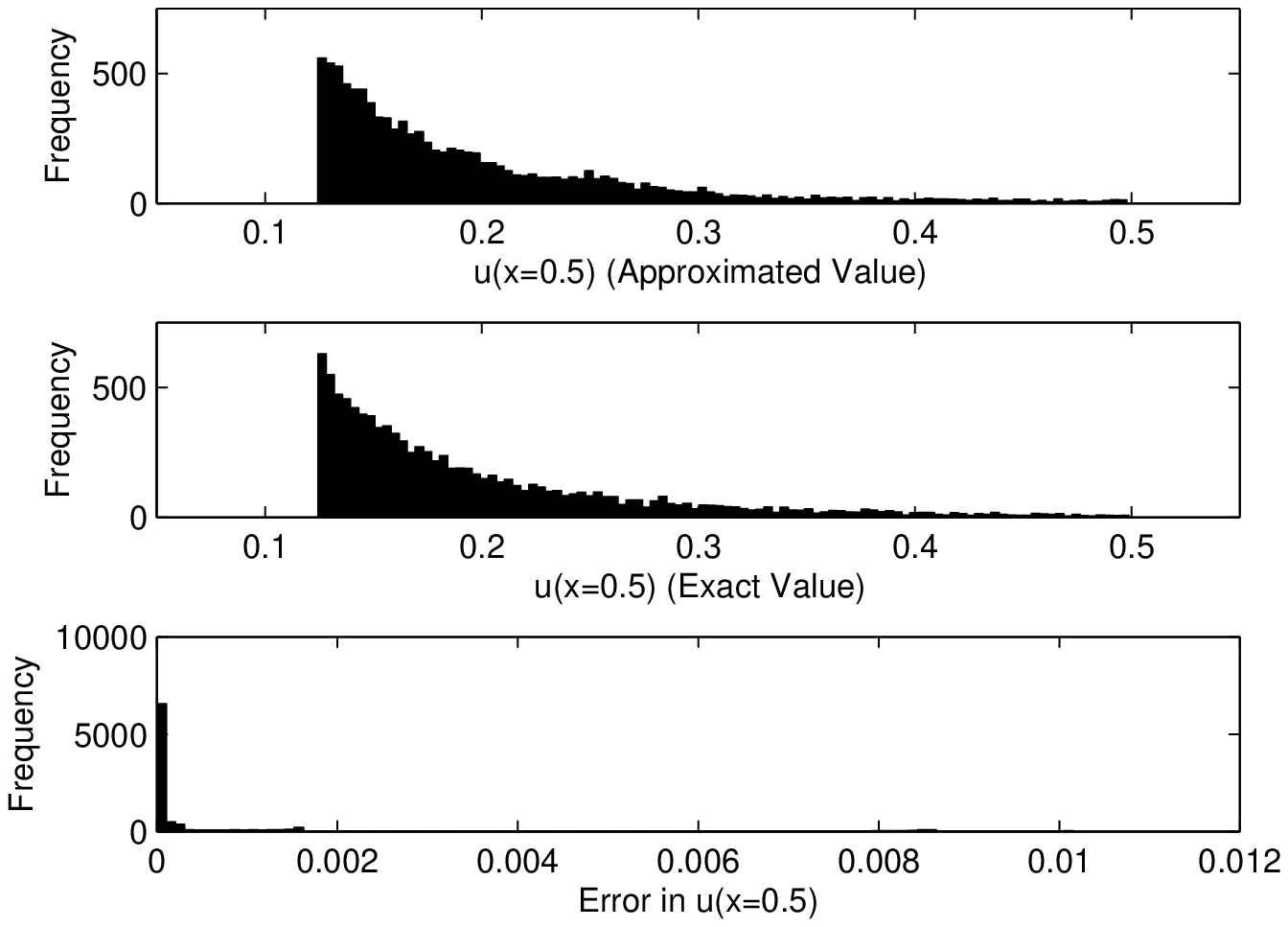}
\label{fig:subfig1}}
\qquad 
\subfloat[][$C=1\times 10^{-2}$]{
\includegraphics[width=0.5\textwidth]{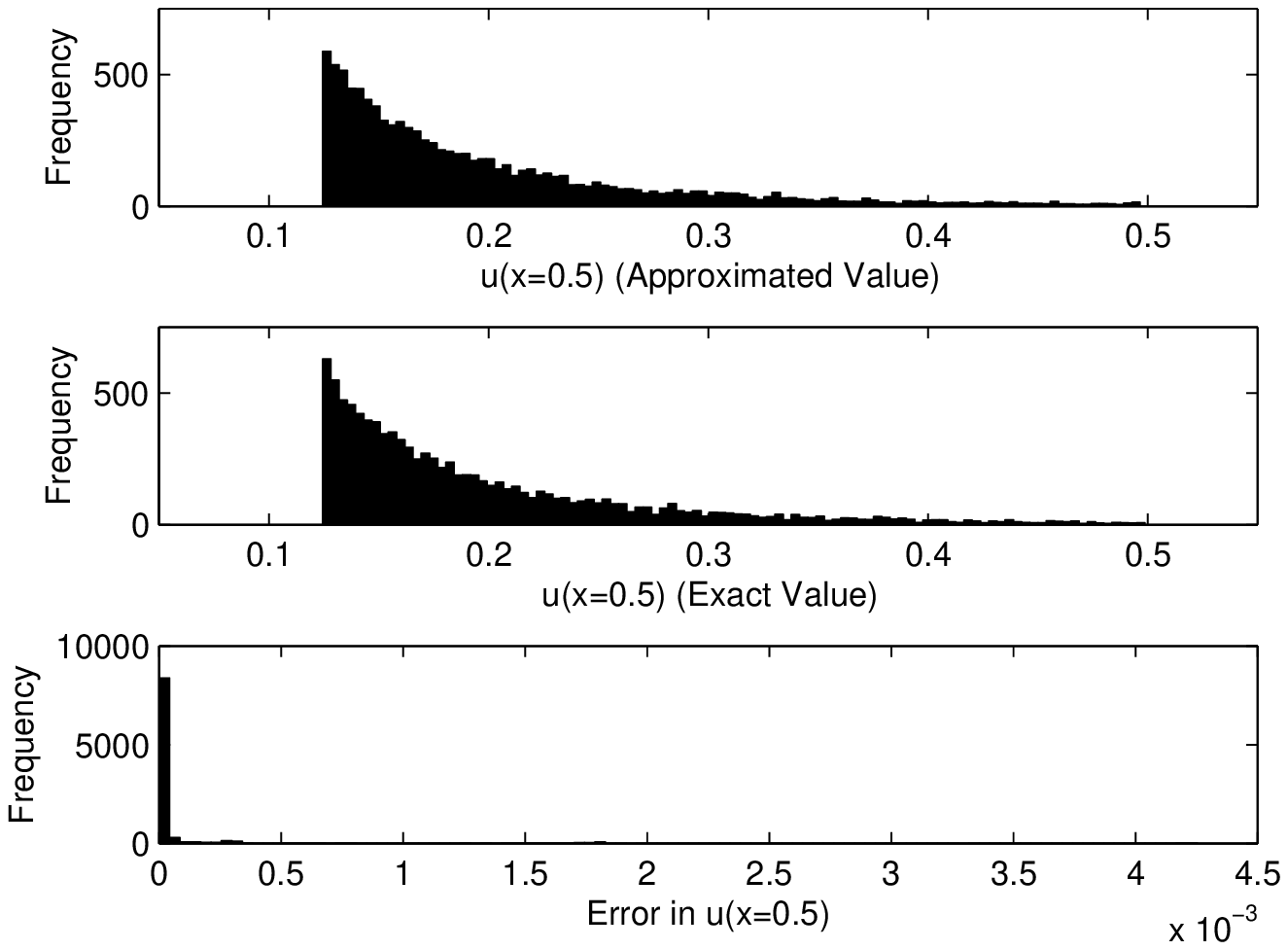}
\label{fig:subfig2}}
\\

\subfloat[][$C=5\times 10^{-3}$]{
\includegraphics[width=0.5\textwidth]{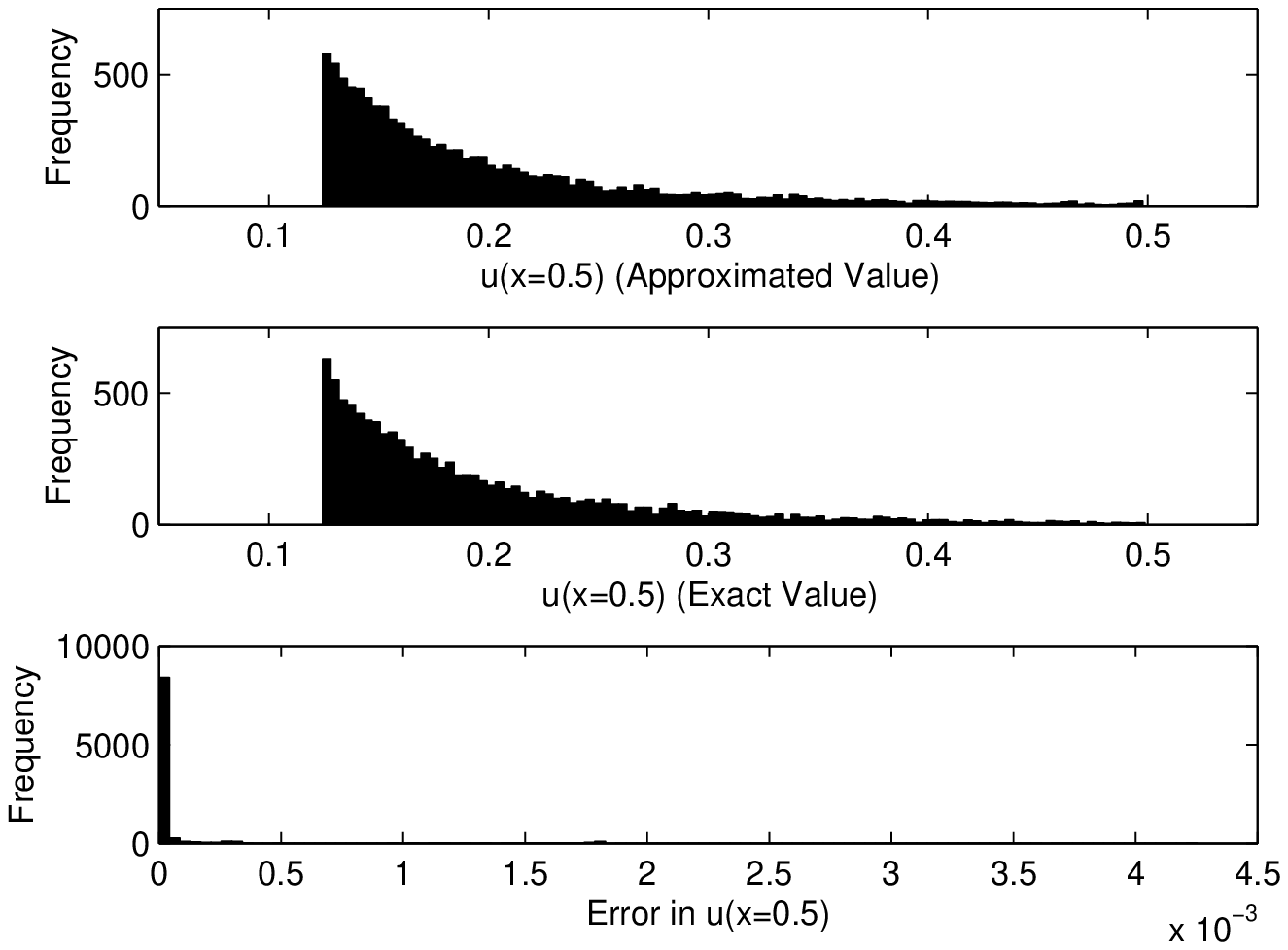}
\label{fig:subfig3}}
\qquad
\subfloat[][$C=1\times 10^{-3}$]{
\includegraphics[width=0.5\textwidth]{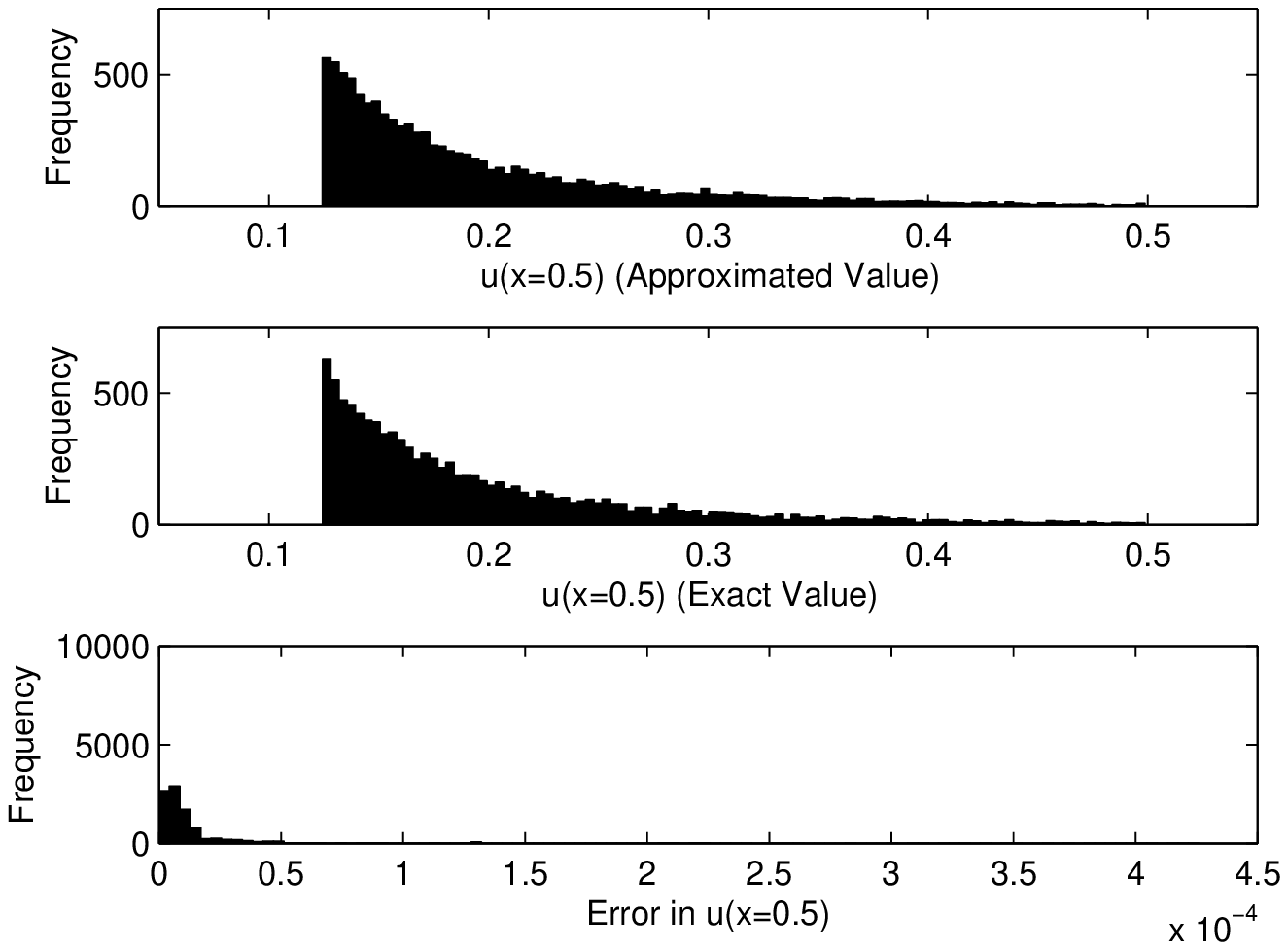}
\label{fig:subfig4}}
\\

\subfloat[][$C=5\times 10^{-4}$]{
\includegraphics[width=0.5\textwidth]{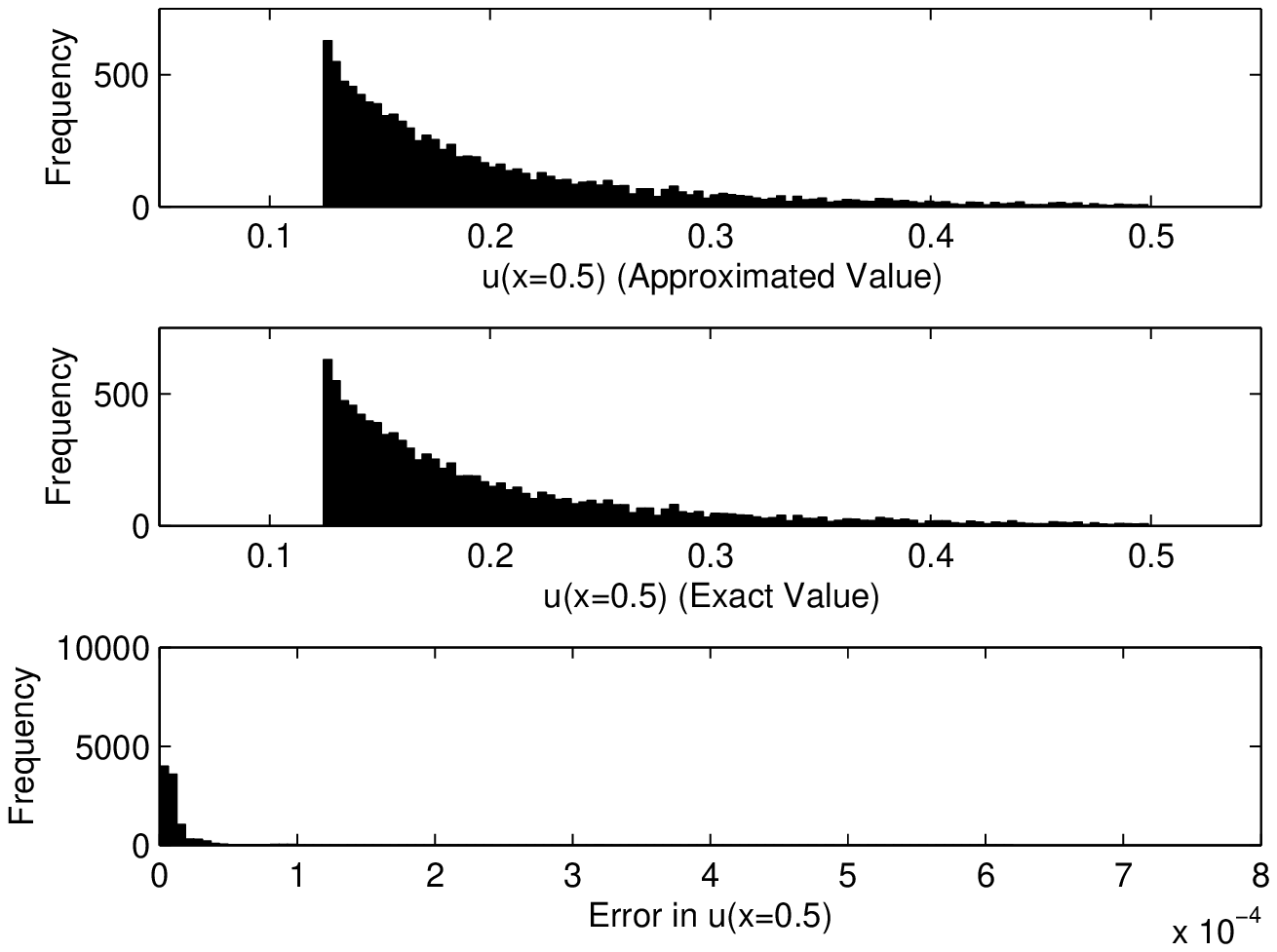}
\label{fig:subfig4}}

\caption{Distribution of function exact values evaluated by Monte Carlo simulation, function approximated evaluated by proposed method, and absolute value of error between them for $n=1, n_0=4, \mbox{ and } n_s=7.$}
\label{fig:globfig}
\end{figure}

\begin{figure}
\subfloat[][$C=5\times 10^{-2}$]{
\includegraphics[width=0.39\textwidth]{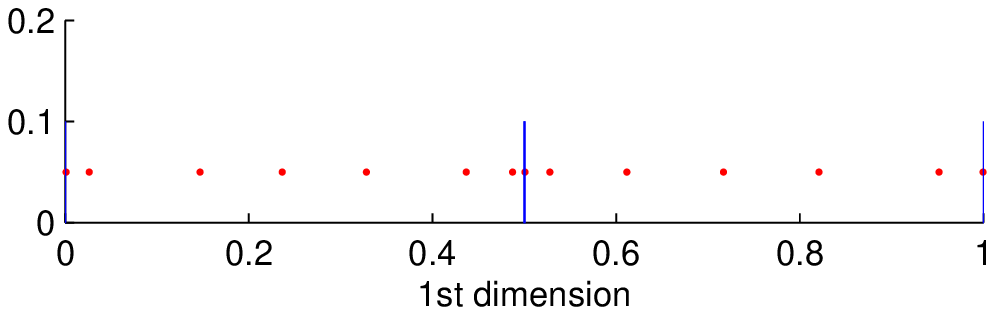}
\label{fig:subfig1}}
\qquad 
\subfloat[][$C=1\times 10^{-2}$]{
\includegraphics[width=0.39\textwidth]{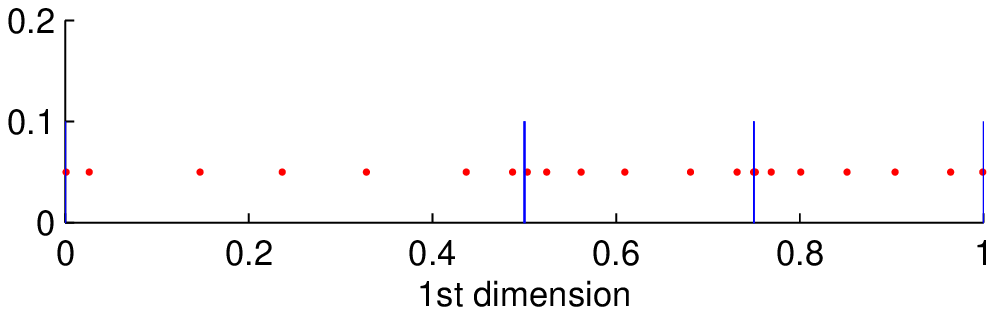}
\label{fig:subfig2}}
\\

\subfloat[][$C=5\times 10^{-3}$]{
\includegraphics[width=0.39\textwidth]{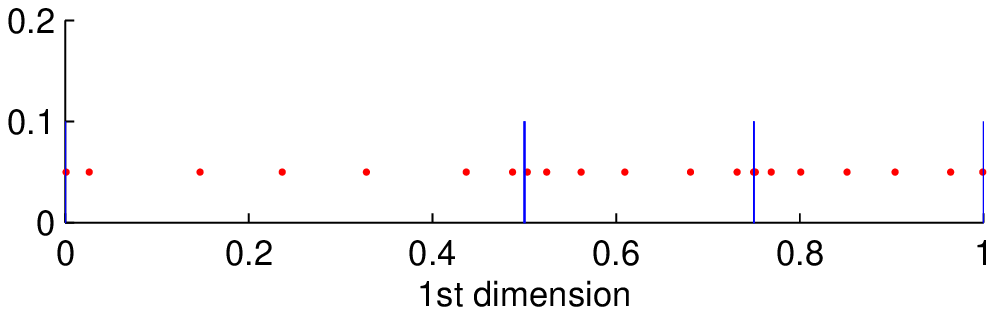}
\label{fig:subfig3}}
\qquad
\subfloat[][$C=1\times 10^{-3}$]{
\includegraphics[width=0.39\textwidth]{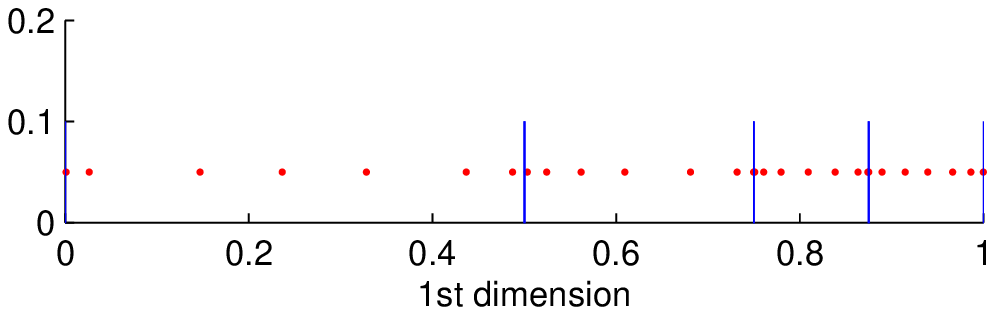}
\label{fig:subfig4}}
\\

\subfloat[][$C=5\times 10^{-4}$]{
\includegraphics[width=0.39\textwidth]{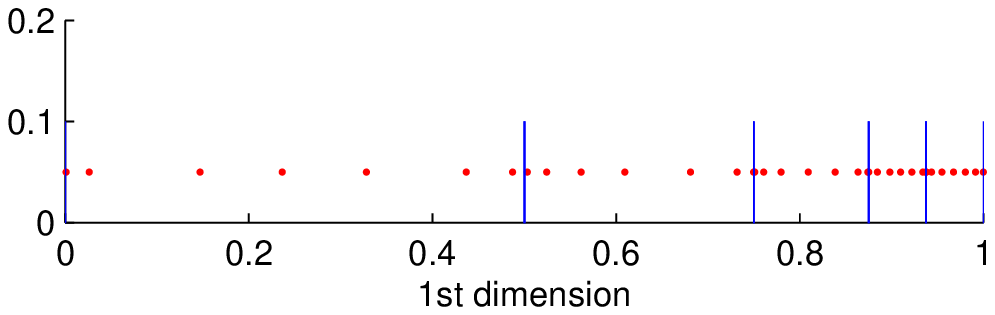}
\label{fig:subfig4}}

\caption{Partition of the random parameter space with sample points utilized for $n=1, n_0=4, \mbox{ and } n_s=7.$}
\label{fig:globfig}
\end{figure}

\begin{figure}
\subfloat[][Relative erros of variance]{
\includegraphics[width=0.45\textwidth]{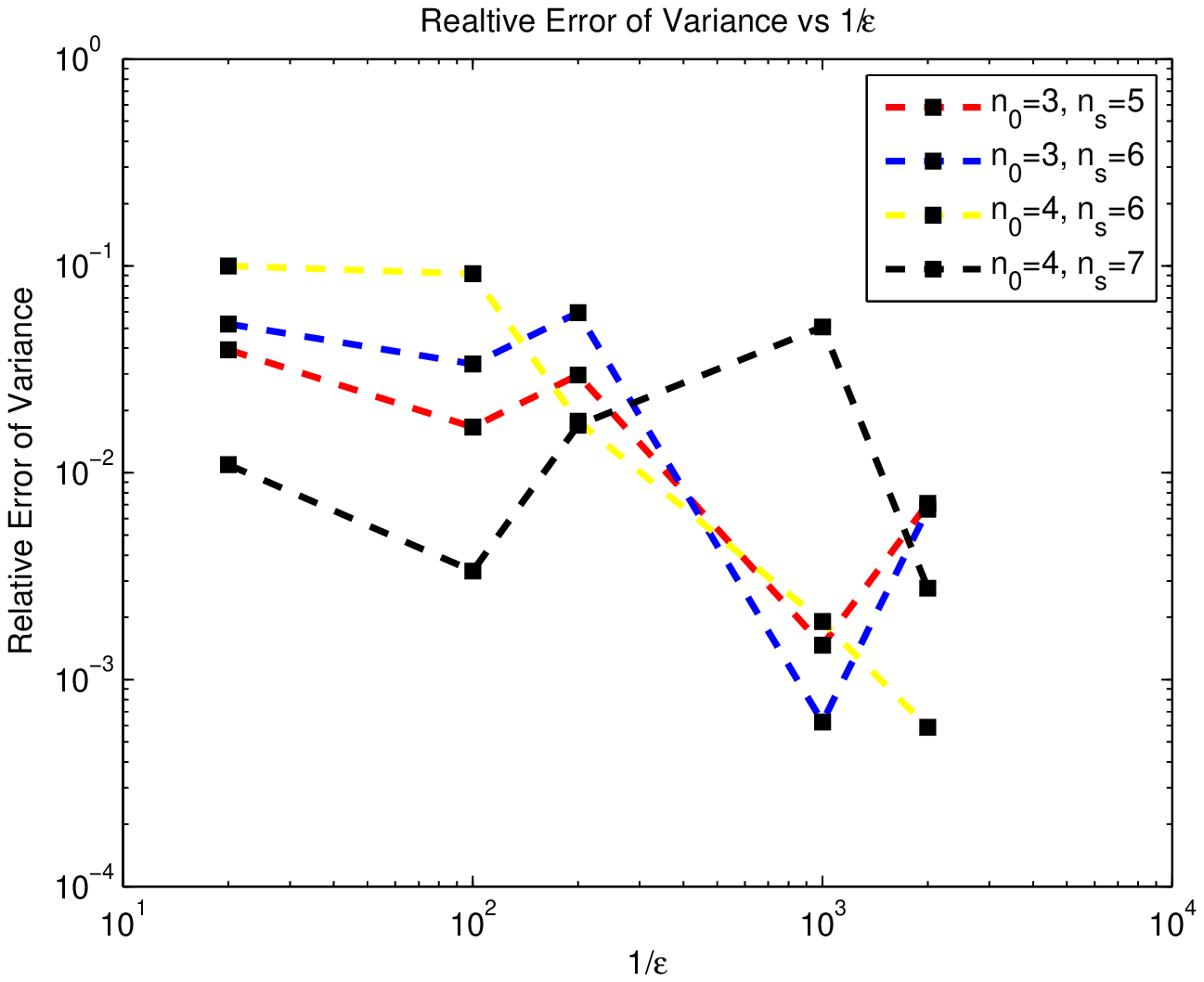}
\label{fig:subfig1}}
\qquad 
\subfloat[][Relative erros of mean]{
\includegraphics[width=0.45\textwidth]{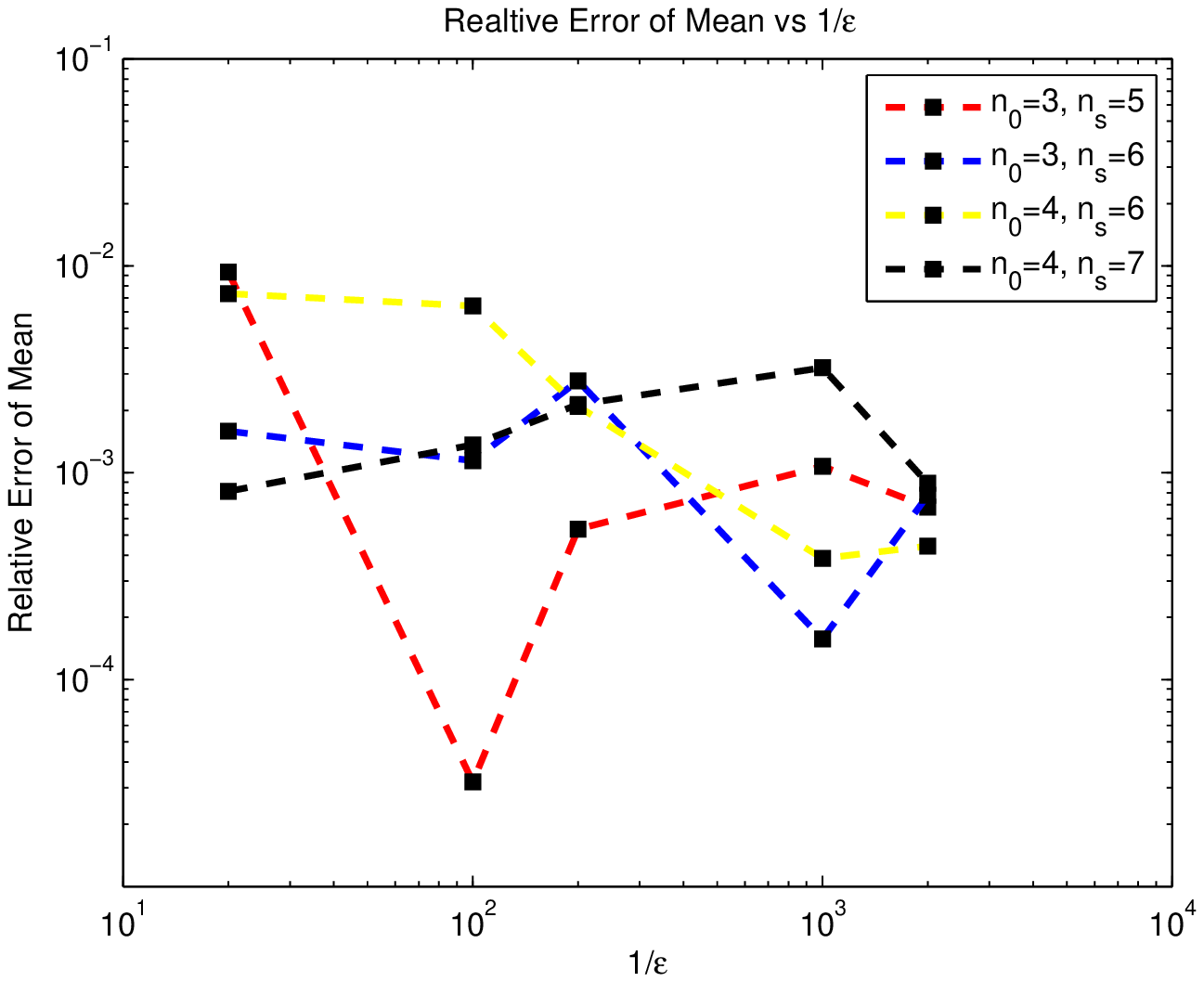}
\label{fig:subfig2}}
\\

\subfloat[][Error values (MSE)]{
\includegraphics[width=0.45\textwidth]{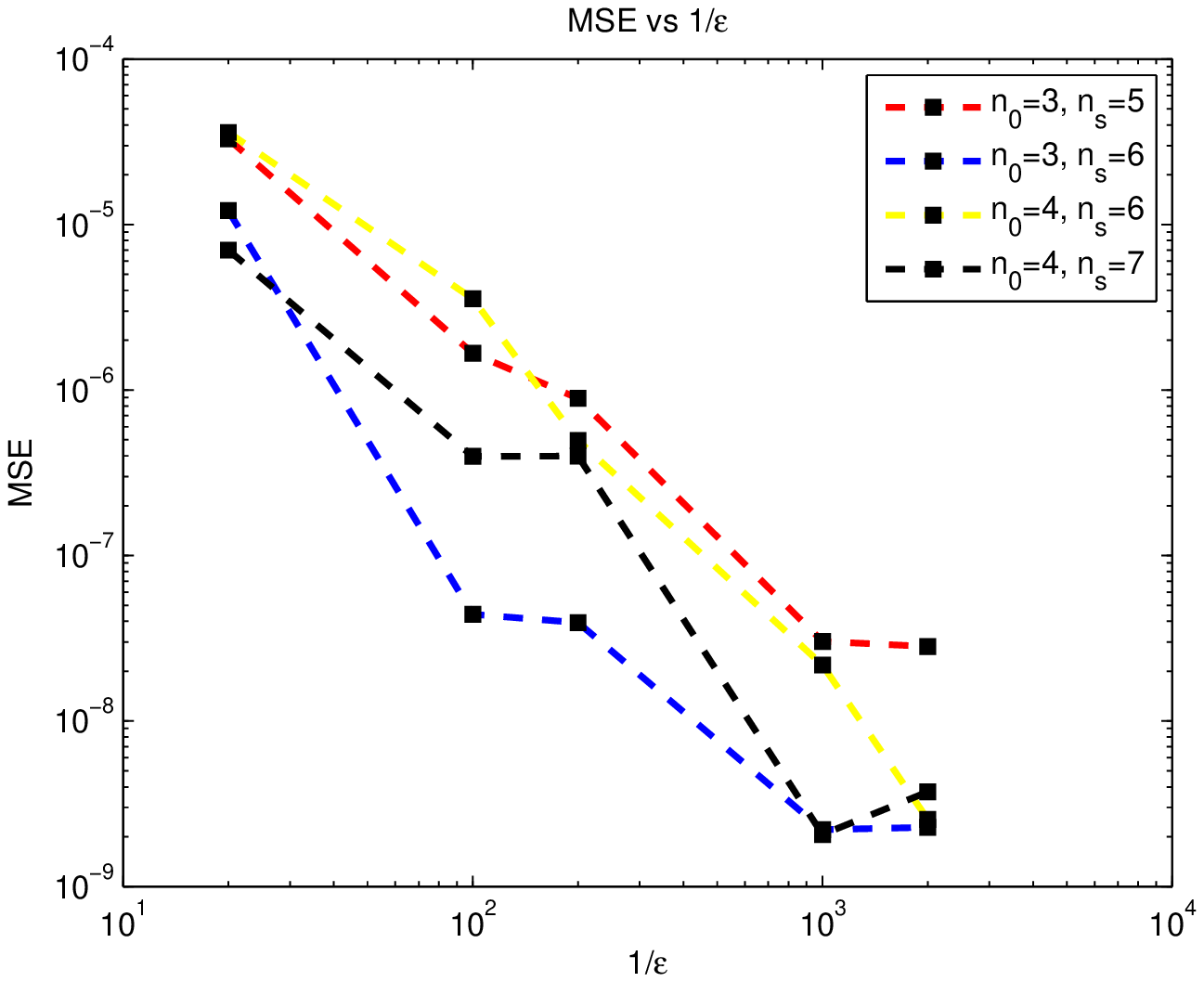}
\label{fig:subfig3}}
\caption{Relative error values of both mean and variance evaluated by Monte Carlo simulation and proposed method, and error (MSE) for all cases.}
\end{figure}

Figures (14-17) show that refinement is concentrated in regions where the solution is highly dependent on random input data, as one should expect according to the construction of the adaptive strategy. Also, note that there are no non-physical realizations for this test case. As $C$ is further decreased the partition becomes increasingly more refined, particularly along the directions of sharp variation with respect to the random data. Here, the sensitivity of the pdf to the selected value of $C$ can be readily observed. A simple Matlab code to solve this PDE is utilized.

It can be seen that the minimum error is belong to the polynomial order $n_0=3$. The number of sample $n_s=5$ confirmed that we should consider a tradeoff between two parameters $n_0$ and $n_s$; however the errors of all cases are approximately in the same order. Moreover, the fluctuation among is a result of the inefficient number of Monte Carlo realizations set to 10,000.

\newpage

\section{Conclusion and Future Works}
In this paper, MRA scheme for the quantification of parametric uncertainties has been developed based on generalizing GPC representations. The proposed method depends on an orthonormal representation of the solution dependency on stochastic input parameters. In fact, the impact of uncertainties is utilizing an orthogonal projection of the model solution on spectral basis constructed from piecewise-continuous polynomial functions (MWB) of a set of random variables representing the variability in the system parameters. Then, unknown coefficients in the expansion are computed employing a compressive sensing methodology based on $\ell_1$-minimization. This method is developed to perform local refinement of the representation of the uncertain parameter space where the system exhibits complex, discontinuous, and steep dependences on the random inputs. Hence, applying this method, one avoids using high order GPC; instead, a series of local low order expansions at controlled resolution levels are employed.  Here, the behavior of the proposed method is analysed in light of computations of a simple function with line singularity, and of simulations of a diffusion equation as well as absorption problem.

The main advantages of the proposed method are both overcoming the numerical instability challenges existing in the original GPC, and decreasing computational cost; specifically at the higher orders of GPC. In fact, utilizing a low order MWB expansion is practical in order to still have robust computations, whereas the convergence process can still be controlled by selecting an appropriate level of resolution. However, the fast increase of the terms in MWB expansion with the resolution level prevents the straightforward application of the MRA scheme to situations involving multiple independent random parameters, and calls for adaptive techniques. Former numerical studies have shown the validity of the proposed adaptive scheme; however, they also highlight the need for an improvement of the methodology to increase the speed of computations, the convergence rate, and also decrease the computational cost. Here, compressive sensing was proposed to fulfill all desire needs. In fact, compressive sensing ro compressive sampling is utilized here as a non-intrusive approach to decrease the computational cost and increase the convergence rate. The statistical results, such as mean and variance, evaluated from both the proposed method and Monte Carlo simulation, are compared in the benchmarks; Their high consistency show the robustness of the proposed method.

Based on the three problems discussed in this paper, and due to the local nature of the representation and utilizing compressive sensing as a powerful tool for recovery of unknown coefficient in MWB expansion, the proposed method seems to be superior than other methods when confronting with constrained problems, such as  discontinuity, presence of sharp gradient, steep dependence of solution on random input parameters, or any other physical constraints.

In the future work, we will focus on extending the proposed method for random input parameters with arbitrary probability density functions. The uniform distribution has specific characteristics which makes us easy to employ compressive sensing as a non-intrusive method to perform adaptive partitioning strategy; however, an arbitrary pdf does not benefit them in general which prevent utilizing previous samples preserved in memory. To overcome this issue, we should modify the proposed method in next studies.

\begin{figure}
\subfloat[][$C=5\times 10^{-2}$]{
\includegraphics[width=0.5\textwidth]{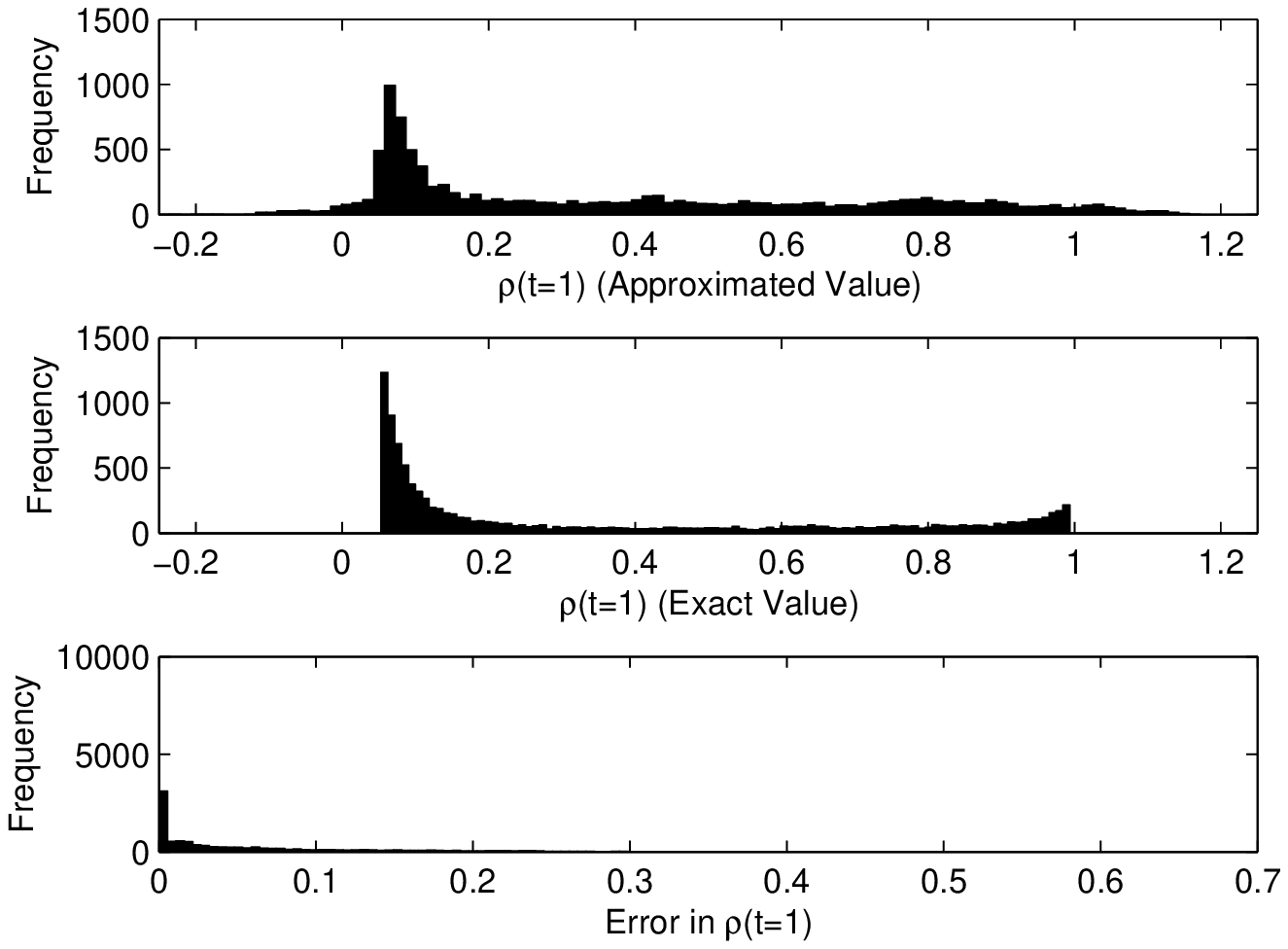}
\label{fig:subfig1}}
\qquad 
\subfloat[][$C=1\times 10^{-2}$]{
\includegraphics[width=0.5\textwidth]{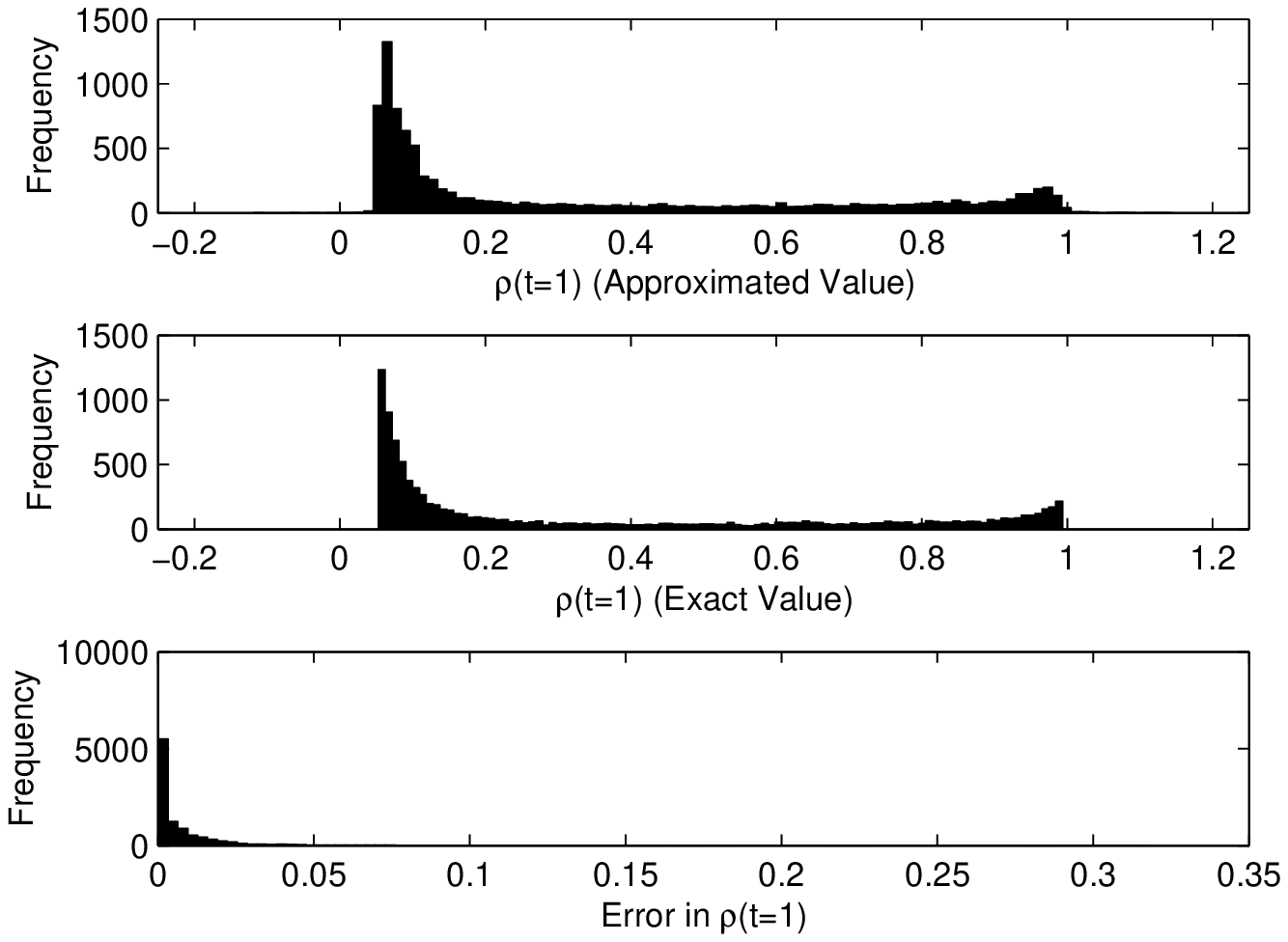}
\label{fig:subfig2}}
\\

\subfloat[][$C=5\times 10^{-3}$]{
\includegraphics[width=0.5\textwidth]{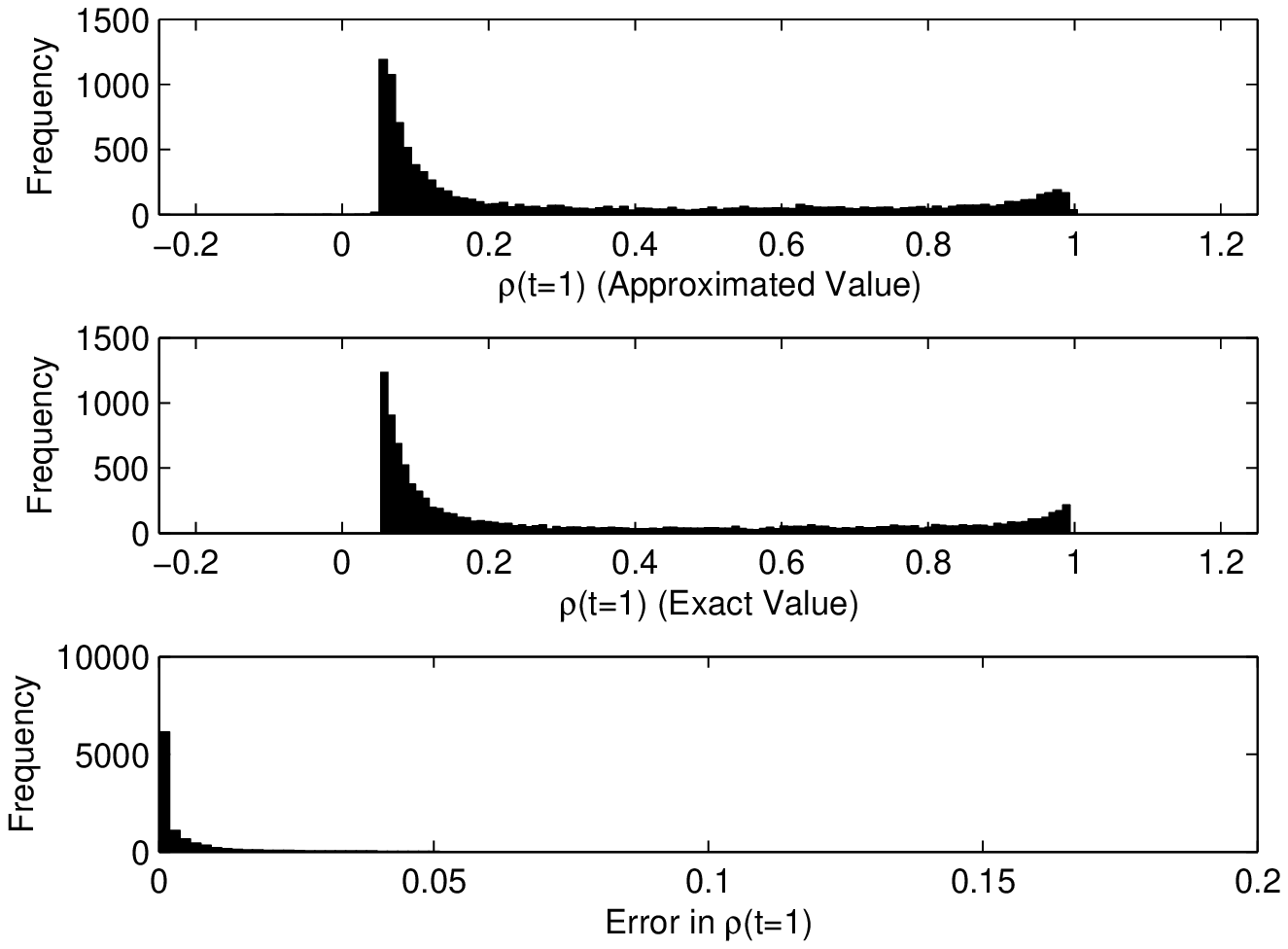}
\label{fig:subfig3}}
\qquad
\subfloat[][$C=1\times 10^{-3}$]{
\includegraphics[width=0.5\textwidth]{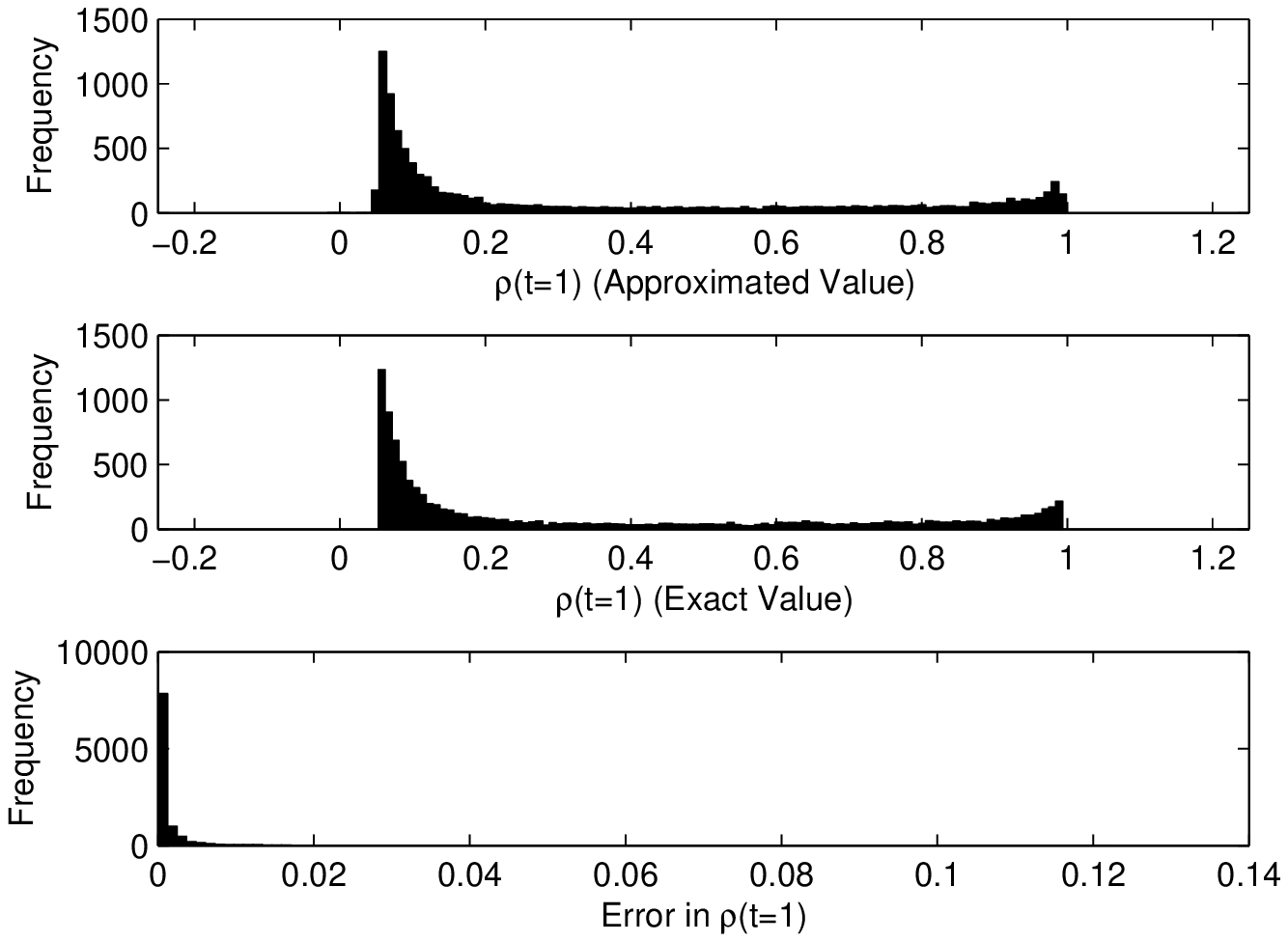}
\label{fig:subfig4}}
\\

\subfloat[][$C=5\times 10^{-4}$]{
\includegraphics[width=0.5\textwidth]{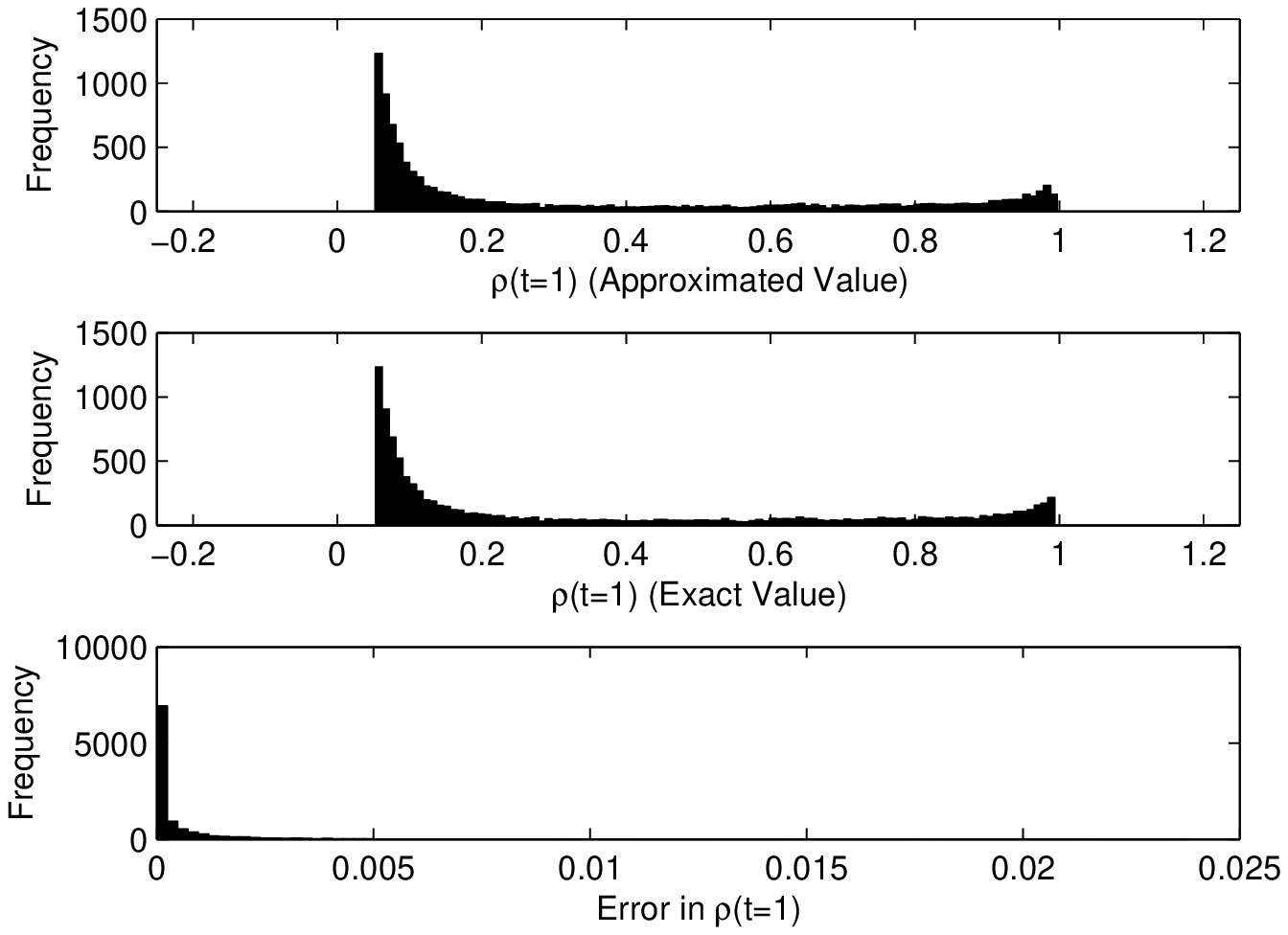}
\label{fig:subfig4}}

\caption{Distribution of function exact values evaluated by Monte Carlo simulation, function approximated evaluated by proposed method, and absolute value of error between them for $n=2, n_0=3, \mbox{ and } n_s=13.$}
\label{fig:globfig}
\end{figure}

\begin{figure}
\subfloat[][$C=5\times 10^{-2}$]{
\includegraphics[width=0.39\textwidth]{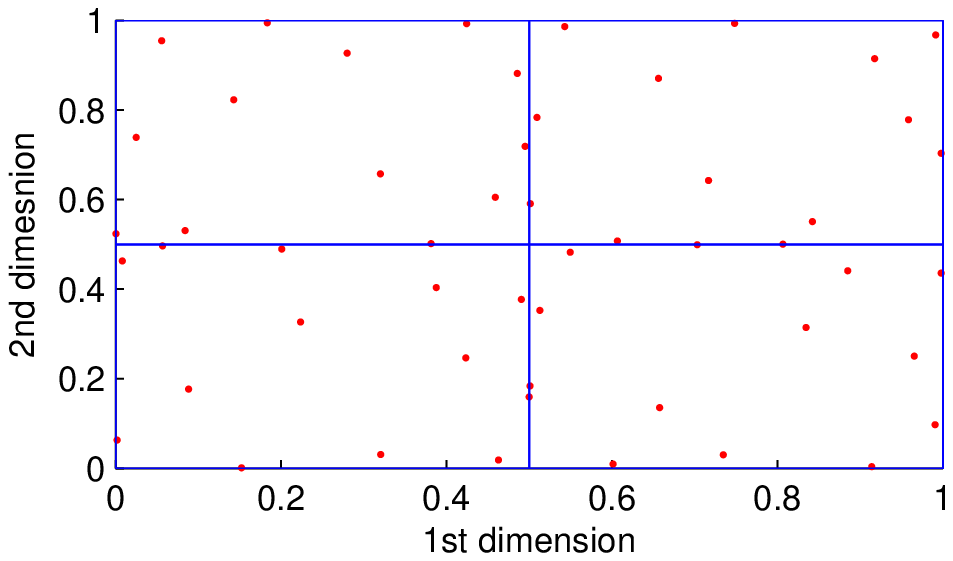}
\label{fig:subfig1}}
\qquad 
\subfloat[][$C=1\times 10^{-2}$]{
\includegraphics[width=0.39\textwidth]{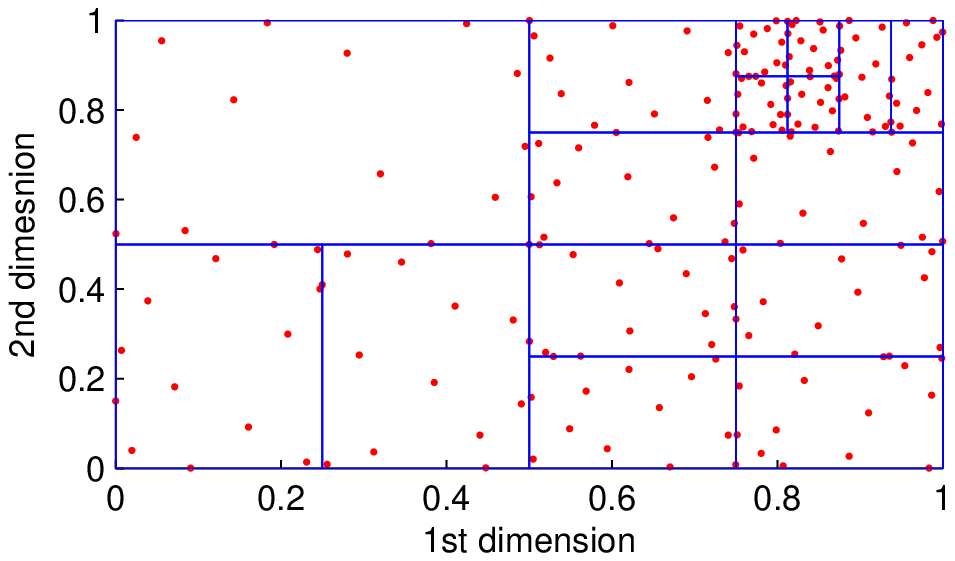}
\label{fig:subfig2}}
\\

\subfloat[][$C=5\times 10^{-3}$]{
\includegraphics[width=0.39\textwidth]{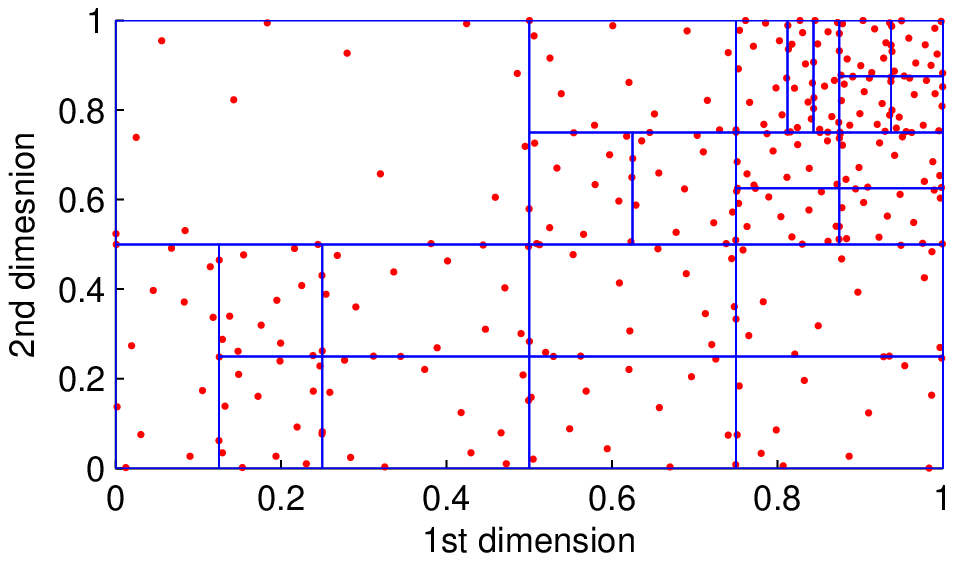}
\label{fig:subfig3}}
\qquad
\subfloat[][$C=1\times 10^{-3}$]{
\includegraphics[width=0.39\textwidth]{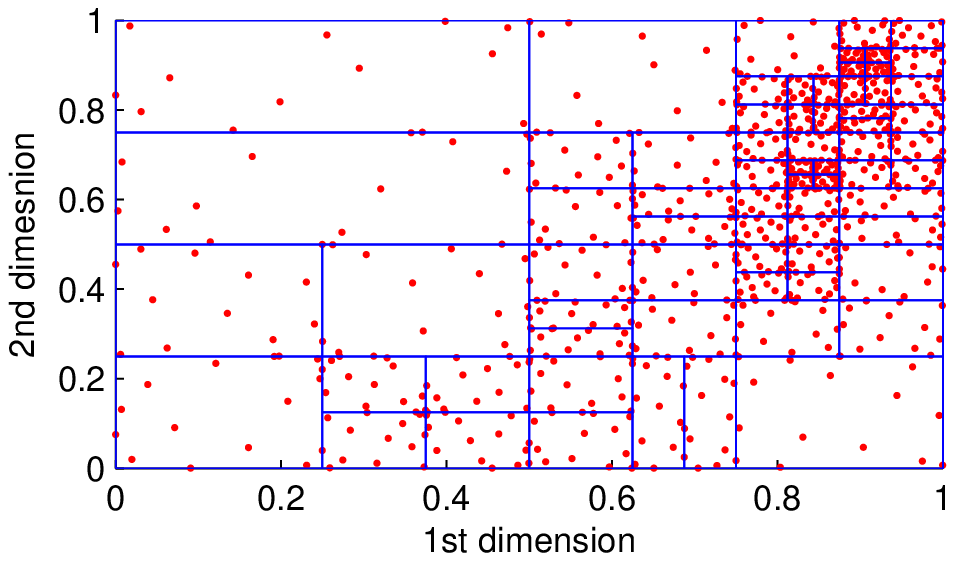}
\label{fig:subfig4}}
\\

\subfloat[][$C=5\times 10^{-4}$]{
\includegraphics[width=0.39\textwidth]{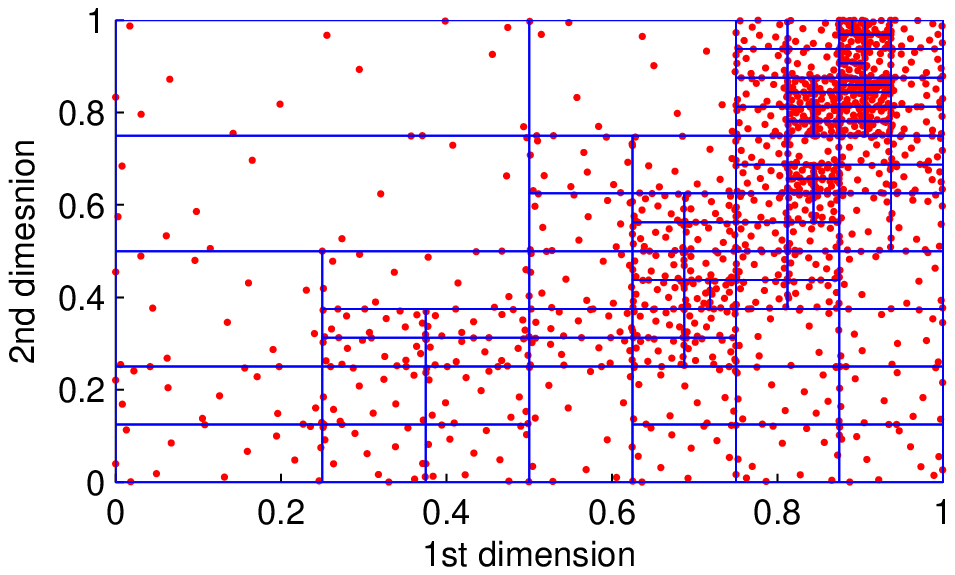}
\label{fig:subfig4}}

\caption{Partition of the random parameter space with sample points utilized for $n=2, n_0=3, \mbox{ and } n_s=13.$}
\label{fig:globfig}
\end{figure}

\begin{figure}
\subfloat[][$C=5\times 10^{-2}$]{
\includegraphics[width=0.5\textwidth]{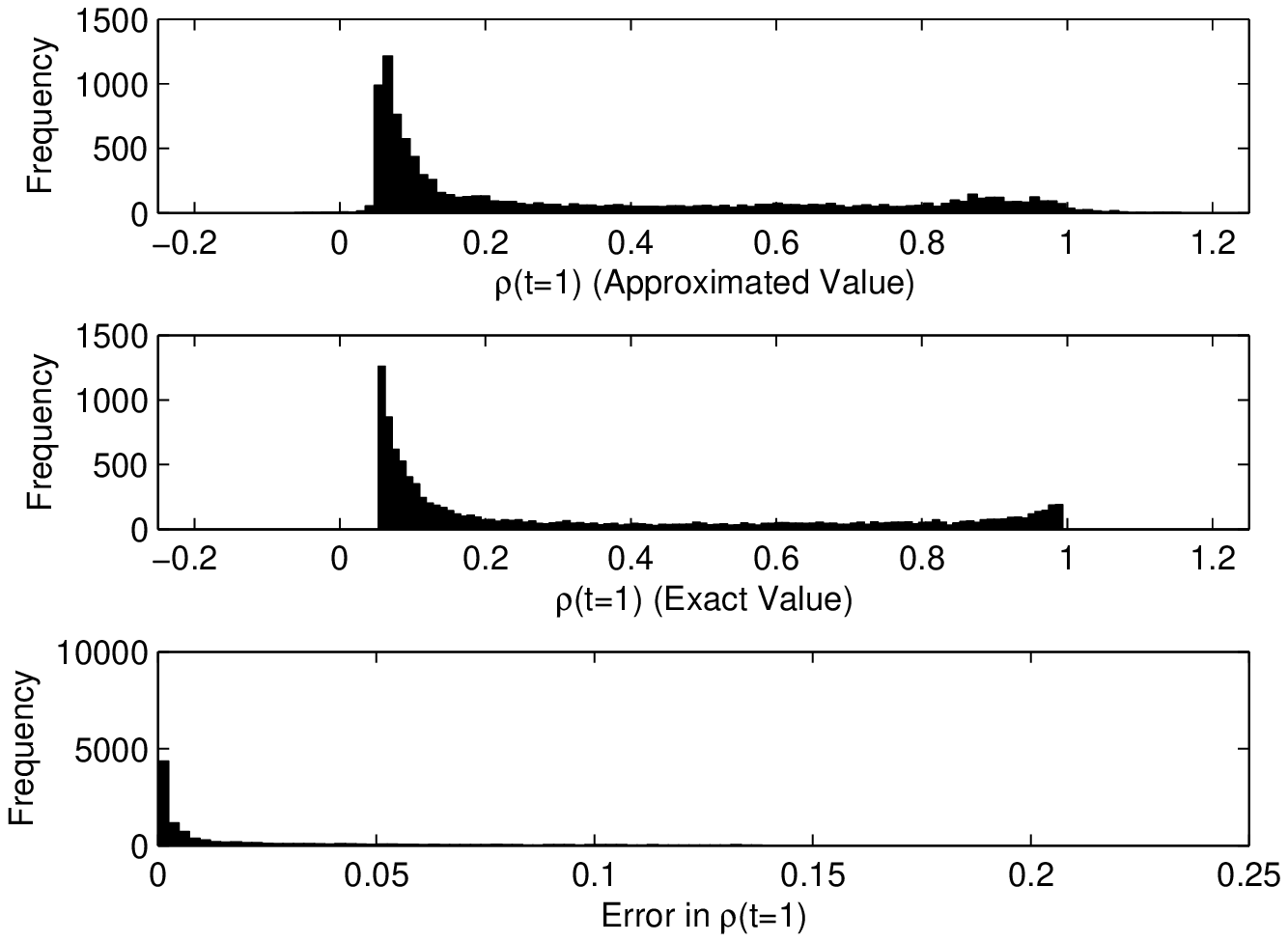}
\label{fig:subfig1}}
\qquad 
\subfloat[][$C=1\times 10^{-2}$]{
\includegraphics[width=0.5\textwidth]{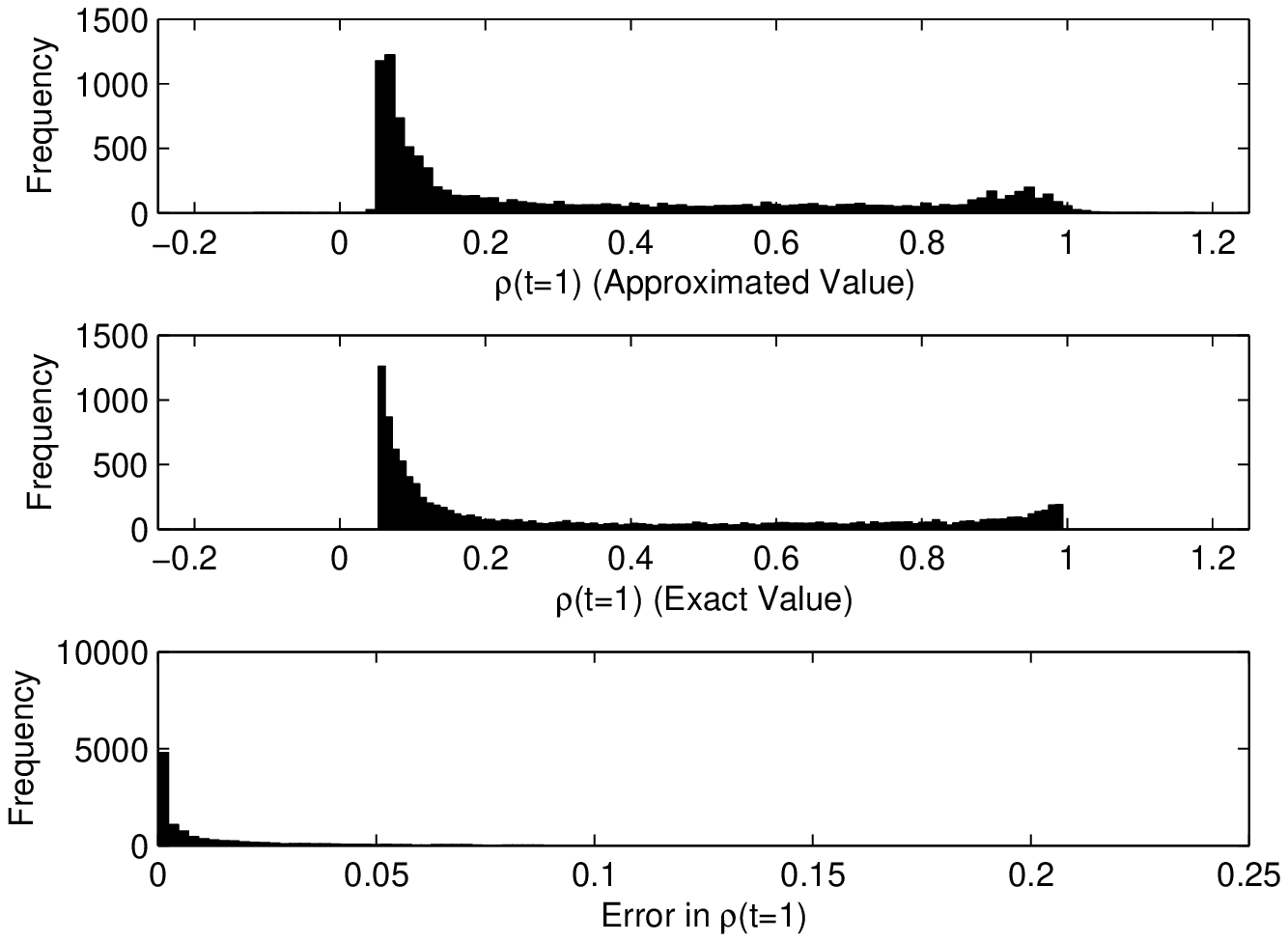}
\label{fig:subfig2}}
\\

\subfloat[][$C=5\times 10^{-3}$]{
\includegraphics[width=0.5\textwidth]{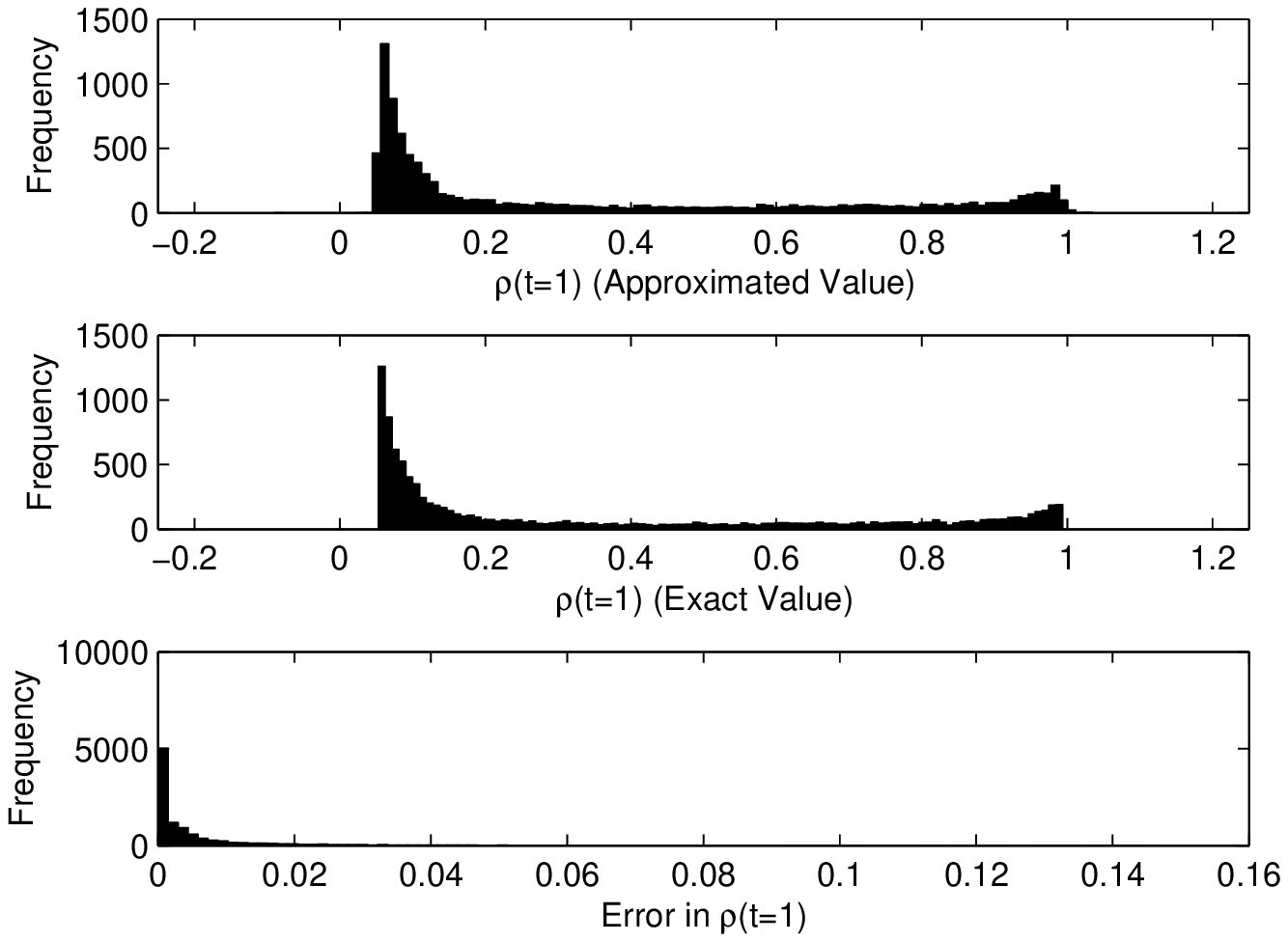}
\label{fig:subfig3}}
\qquad
\subfloat[][$C=1\times 10^{-3}$]{
\includegraphics[width=0.5\textwidth]{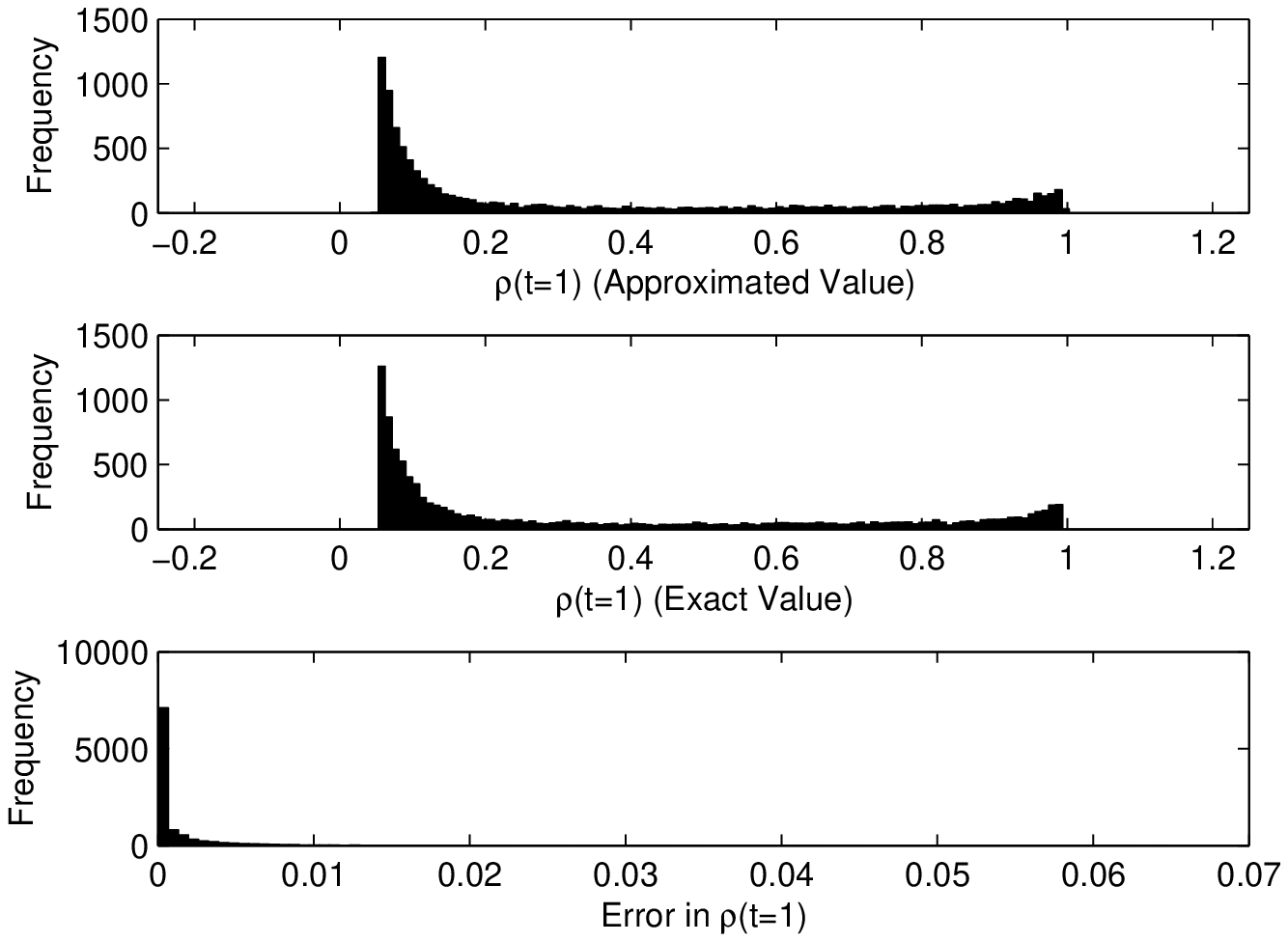}
\label{fig:subfig4}}
\\

\subfloat[][$C=5\times 10^{-4}$]{
\includegraphics[width=0.5\textwidth]{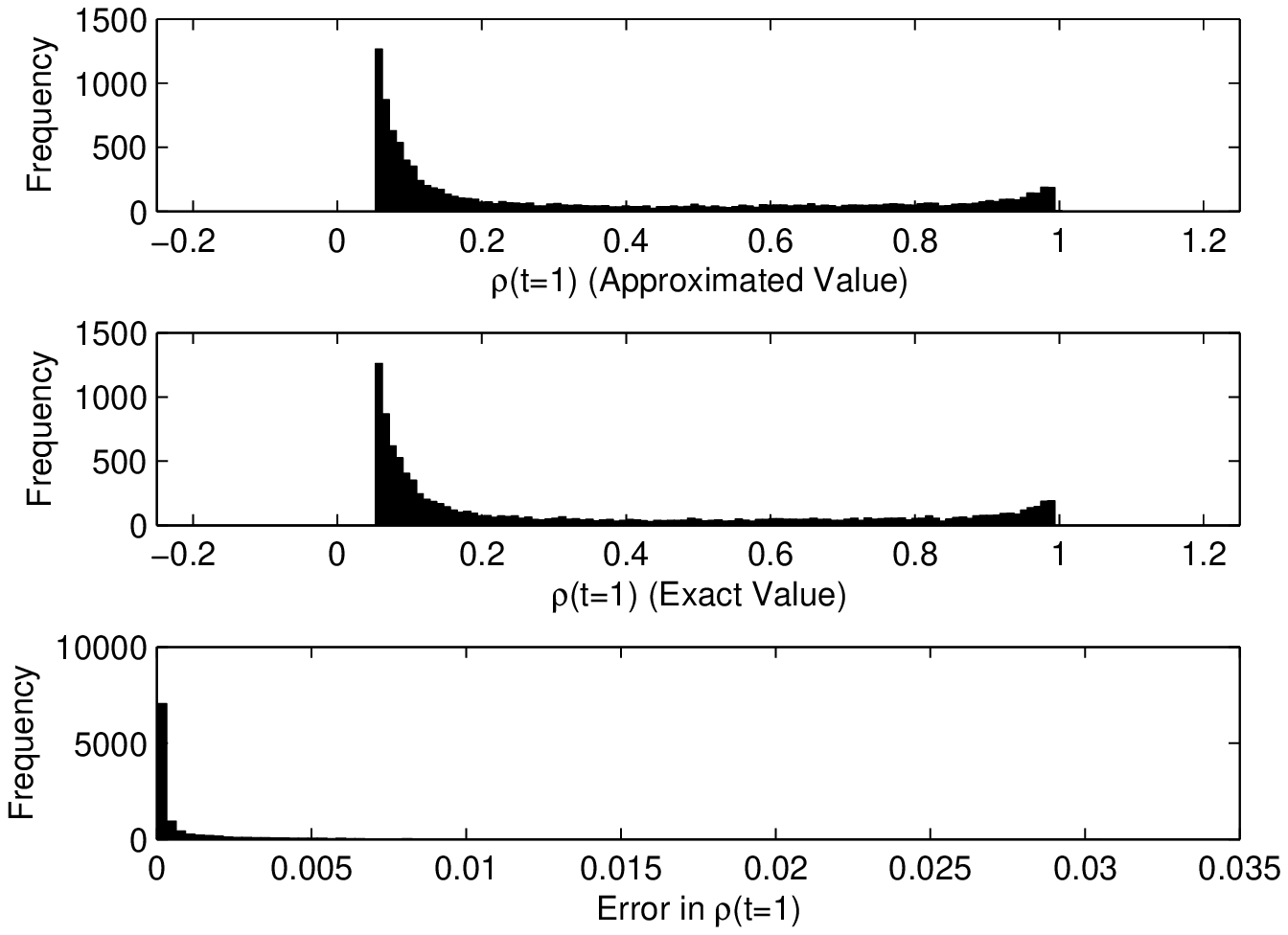}
\label{fig:subfig4}}

\caption{Distribution of function exact values evaluated by Monte Carlo simulation, function approximated evaluated by proposed method, and absolute value of error between them for $n=2, n_0=4, \mbox{ and } n_s=16.$}
\label{fig:globfig}
\end{figure}

\begin{figure}
\subfloat[][$C=5\times 10^{-2}$]{
\includegraphics[width=0.39\textwidth]{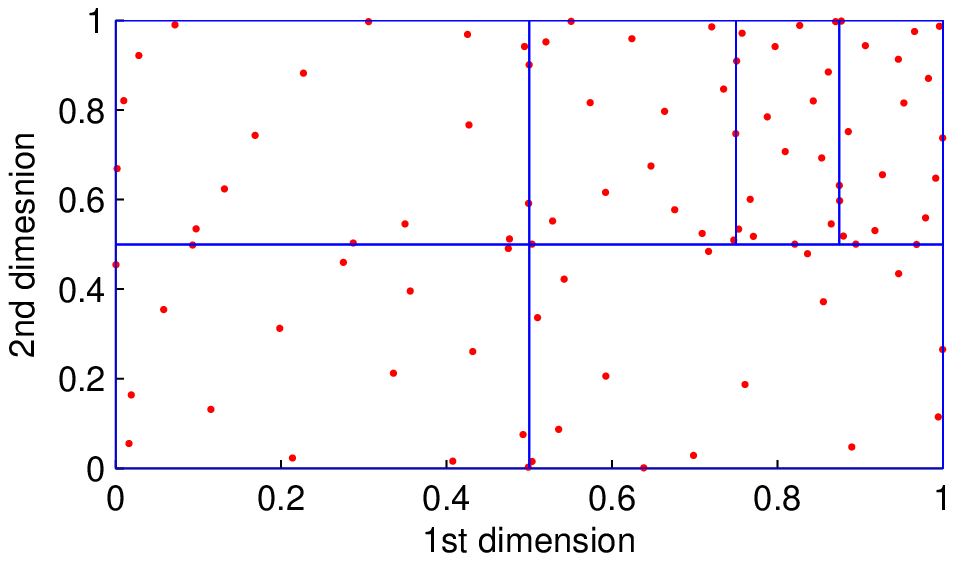}
\label{fig:subfig1}}
\qquad 
\subfloat[][$C=1\times 10^{-2}$]{
\includegraphics[width=0.39\textwidth]{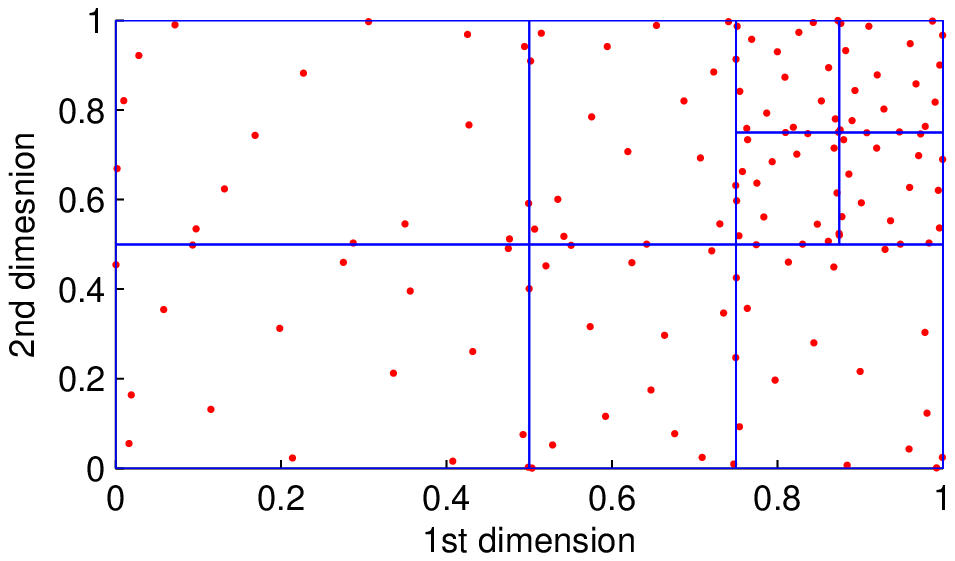}
\label{fig:subfig2}}
\\

\subfloat[][$C=5\times 10^{-3}$]{
\includegraphics[width=0.39\textwidth]{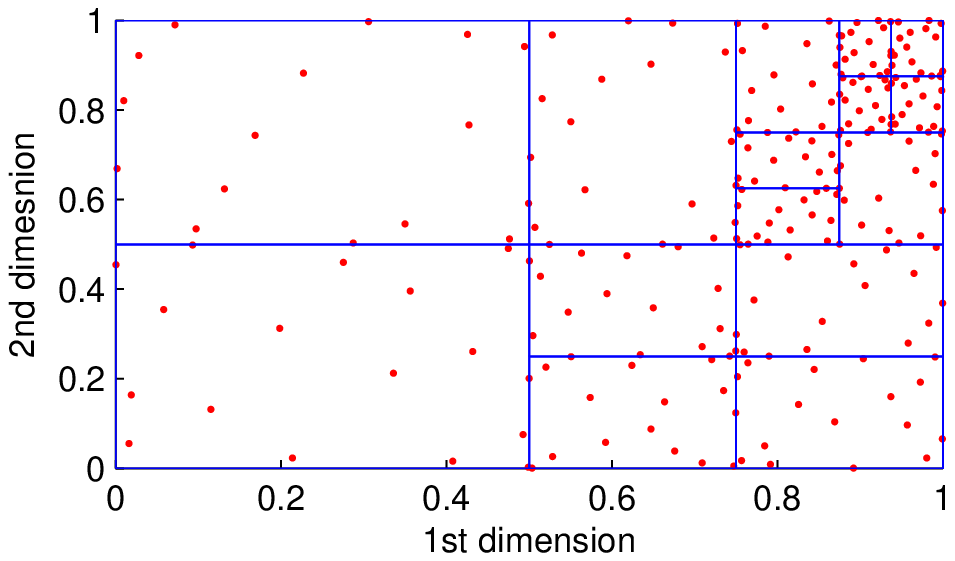}
\label{fig:subfig3}}
\qquad
\subfloat[][$C=1\times 10^{-3}$]{
\includegraphics[width=0.39\textwidth]{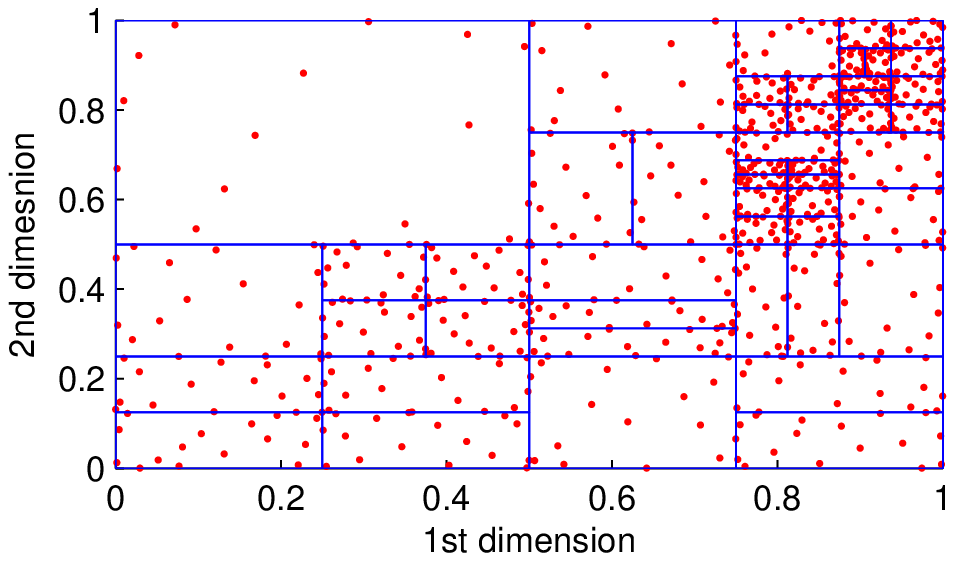}
\label{fig:subfig4}}
\\

\subfloat[][$C=5\times 10^{-4}$]{
\includegraphics[width=0.39\textwidth]{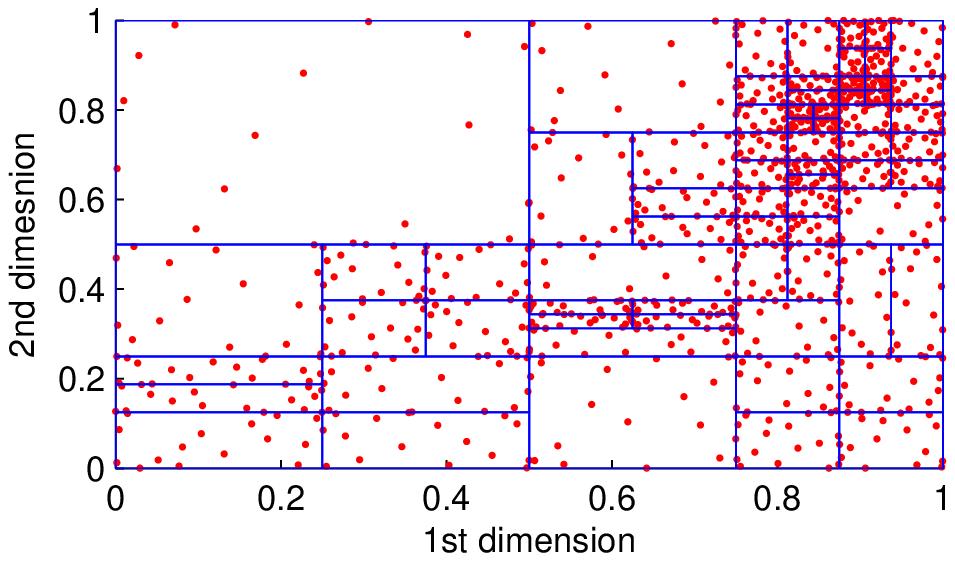}
\label{fig:subfig4}}

\caption{Partition of the random parameter space with sample points utilized for $n=2, n_0=4, \mbox{ and } n_s=16.$}
\label{fig:globfig}
\end{figure}

\begin{figure}
\subfloat[][$C=5\times 10^{-2}$]{
\includegraphics[width=0.5\textwidth]{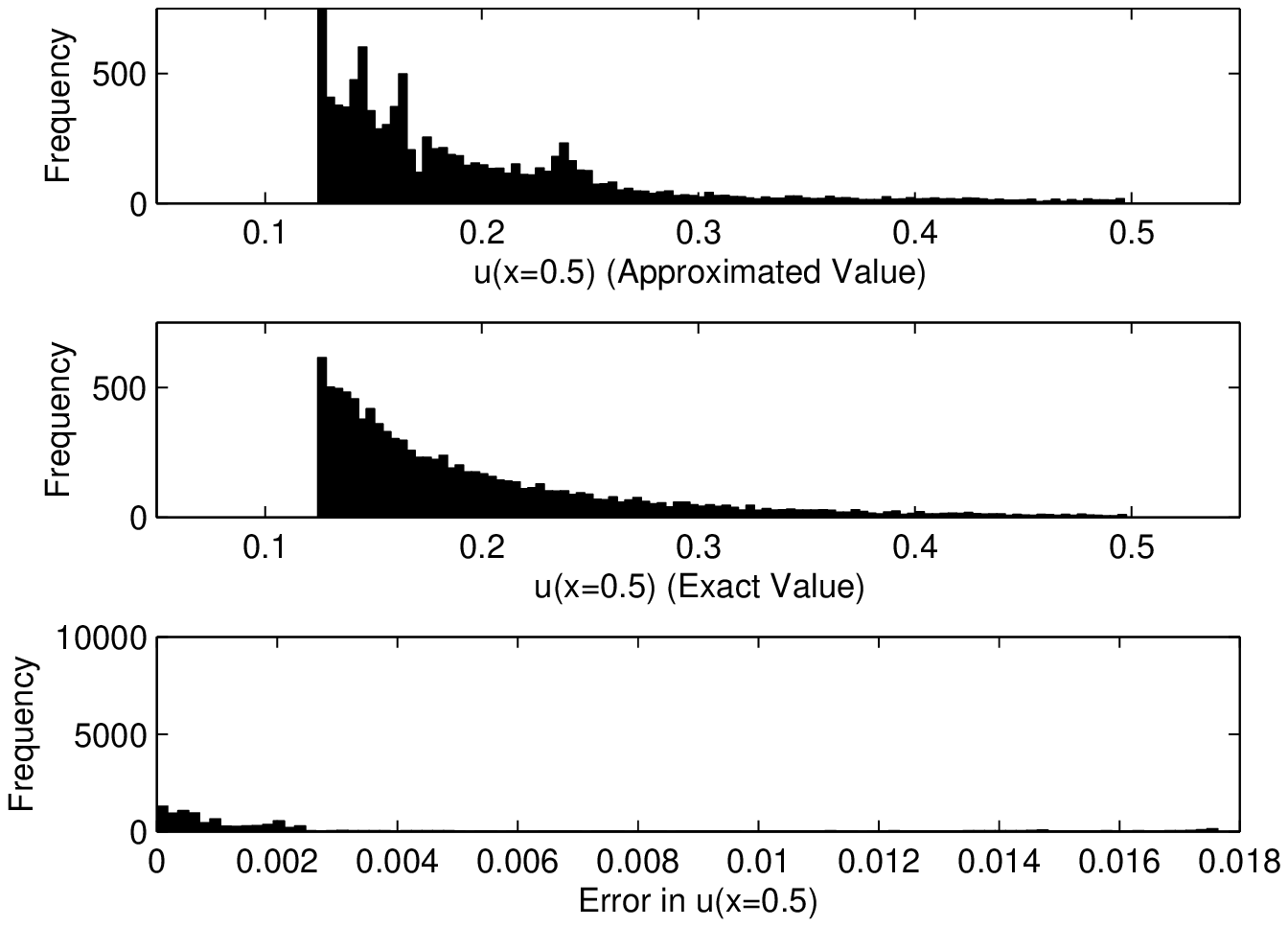}
\label{fig:subfig1}}
\qquad 
\subfloat[][$C=1\times 10^{-2}$]{
\includegraphics[width=0.5\textwidth]{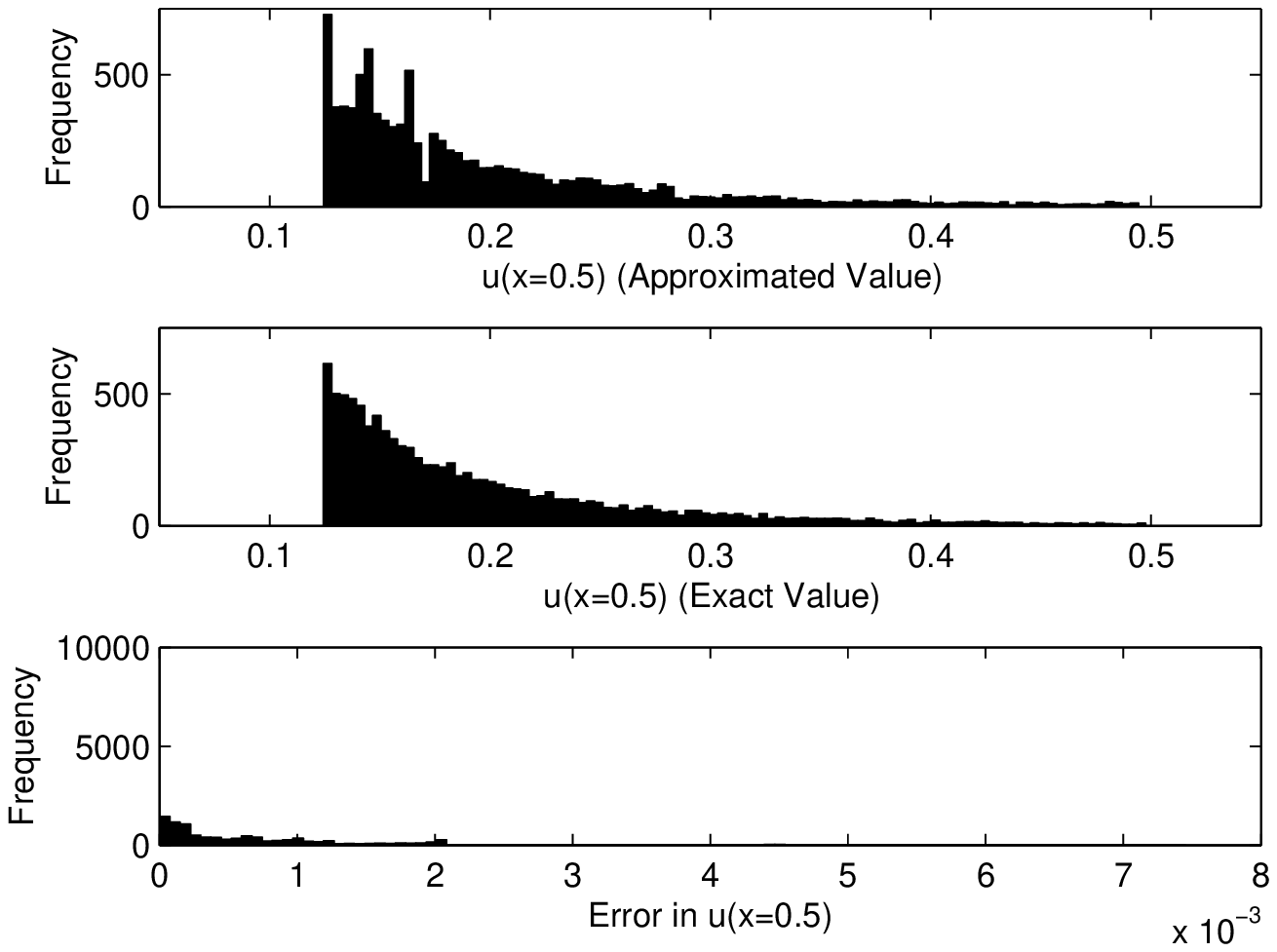}
\label{fig:subfig2}}
\\

\subfloat[][$C=5\times 10^{-3}$]{
\includegraphics[width=0.5\textwidth]{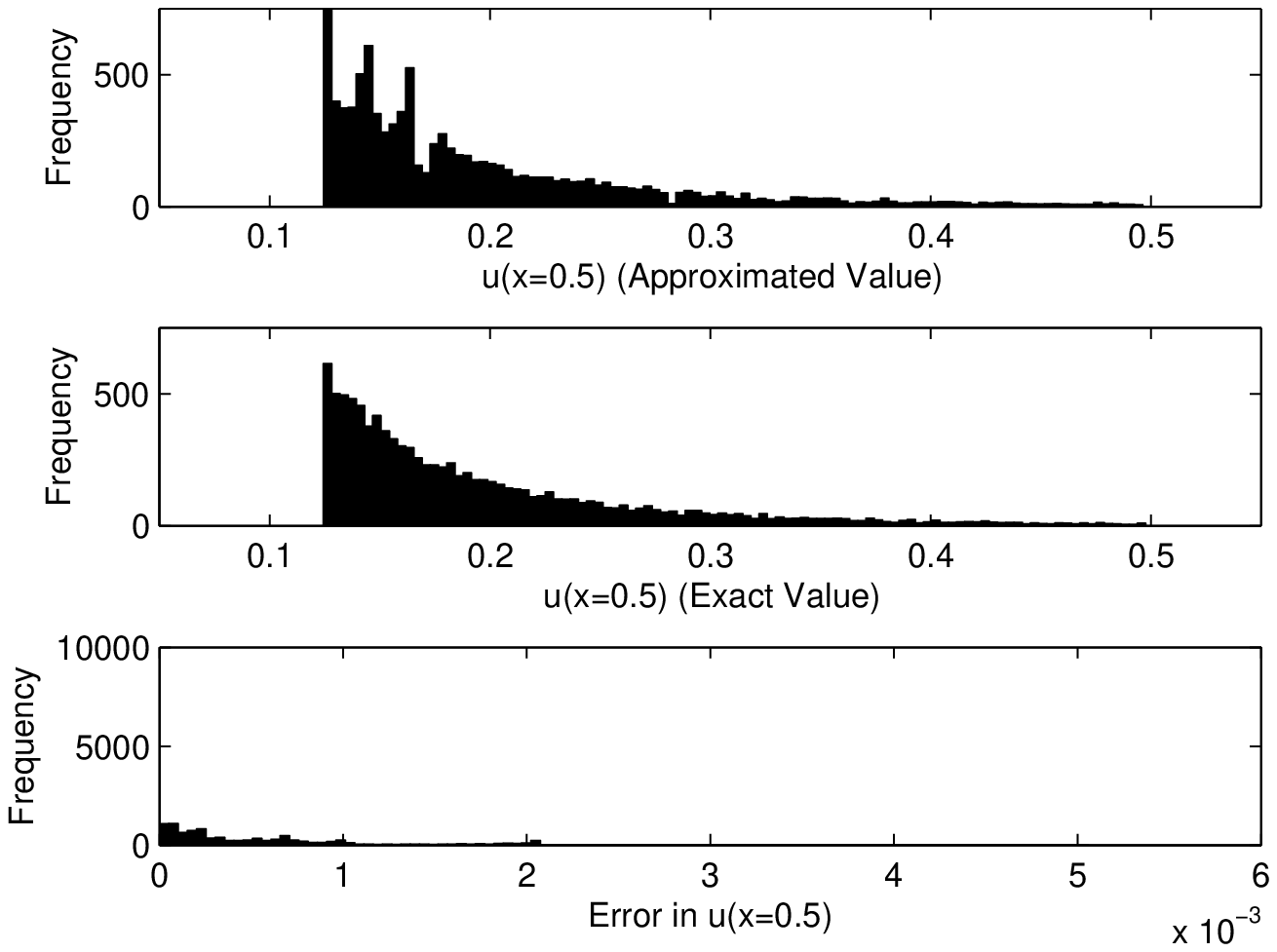}
\label{fig:subfig3}}
\qquad
\subfloat[][$C=1\times 10^{-3}$]{
\includegraphics[width=0.5\textwidth]{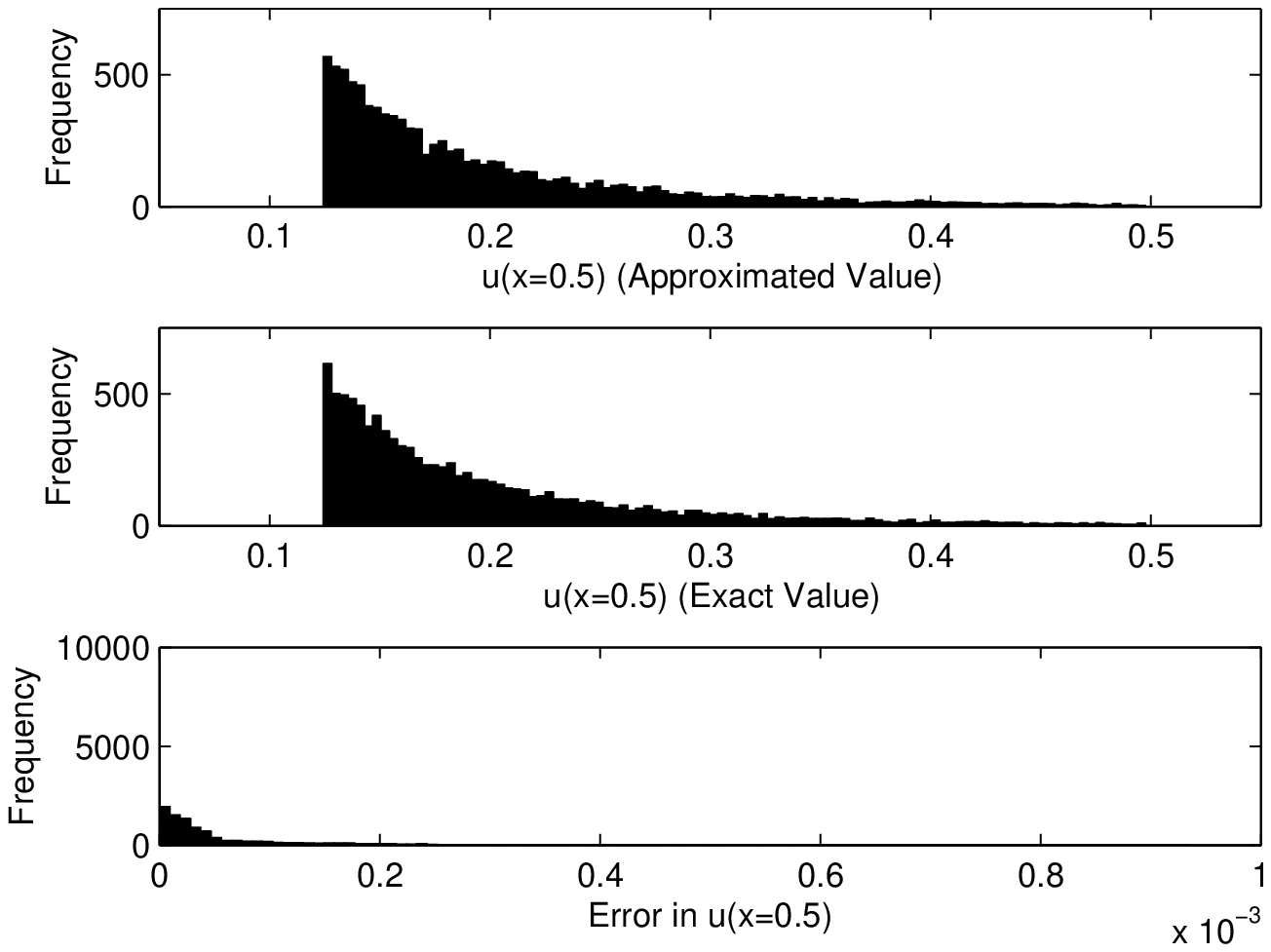}
\label{fig:subfig4}}
\\

\subfloat[][$C=5\times 10^{-4}$]{
\includegraphics[width=0.5\textwidth]{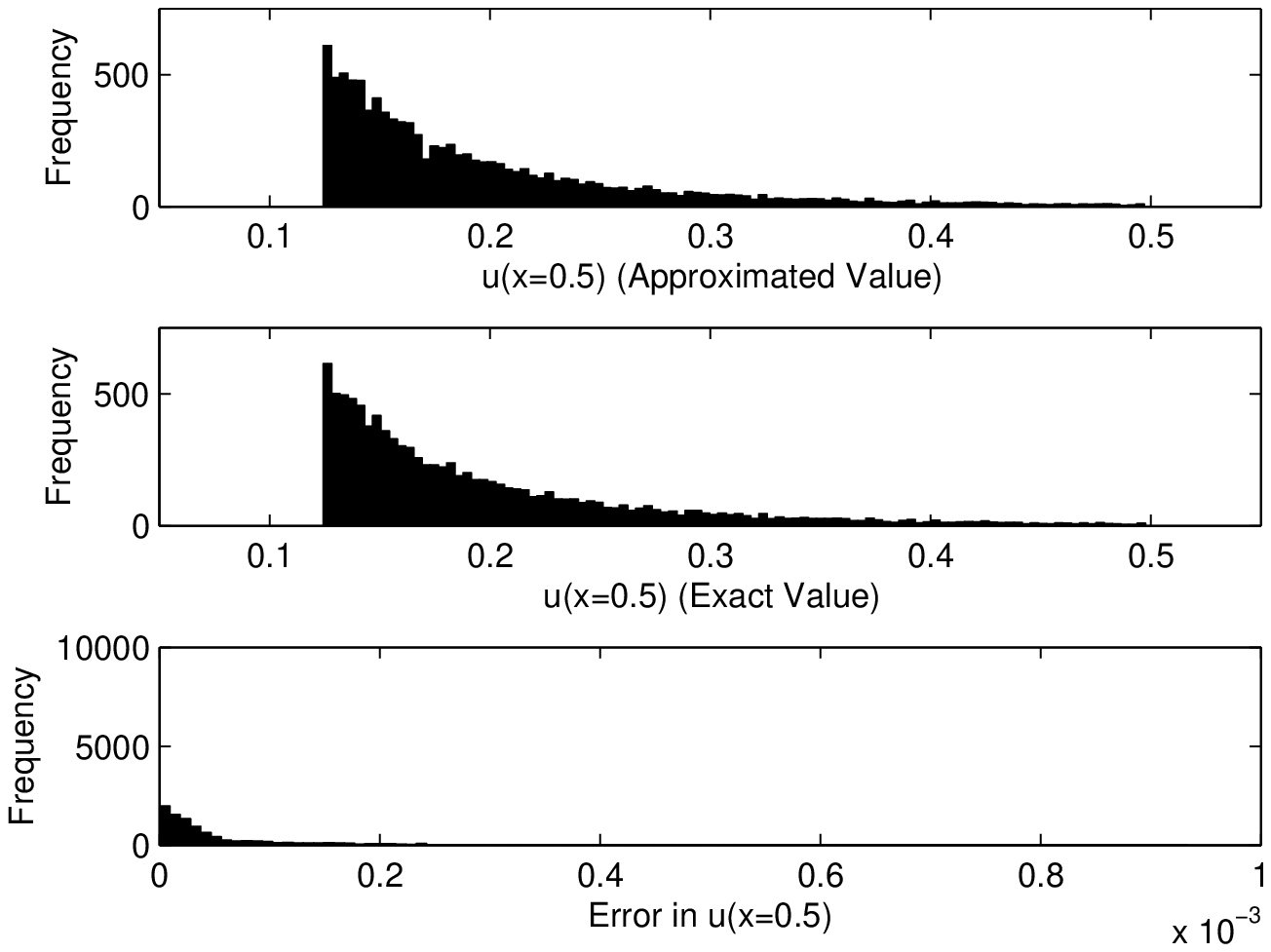}
\label{fig:subfig4}}

\caption{Distribution of function exact values evaluated by Monte Carlo simulation, function approximated evaluated by proposed method, and absolute value of error between them for $n=1, n_0=3, \mbox{ and } n_s=5.$}
\label{fig:globfig}
\end{figure}

\begin{figure}
\subfloat[][$C=5\times 10^{-2}$]{
\includegraphics[width=0.39\textwidth]{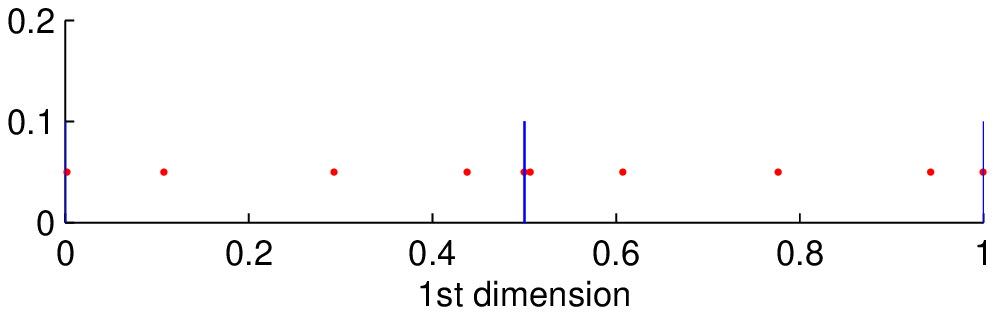}
\label{fig:subfig1}}
\qquad 
\subfloat[][$C=1\times 10^{-2}$]{
\includegraphics[width=0.39\textwidth]{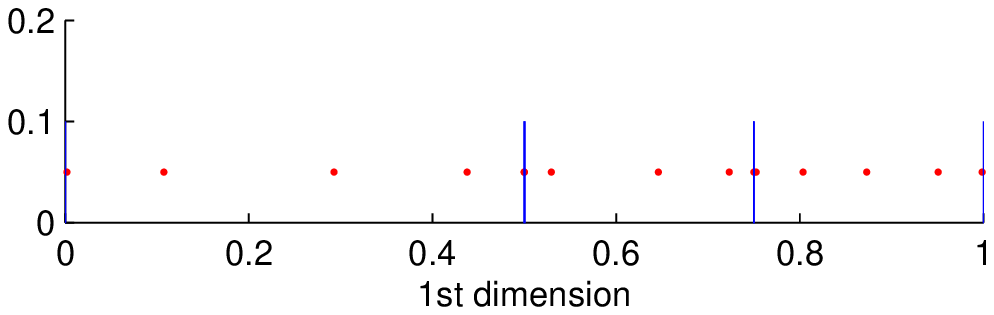}
\label{fig:subfig2}}
\\

\subfloat[][$C=5\times 10^{-3}$]{
\includegraphics[width=0.39\textwidth]{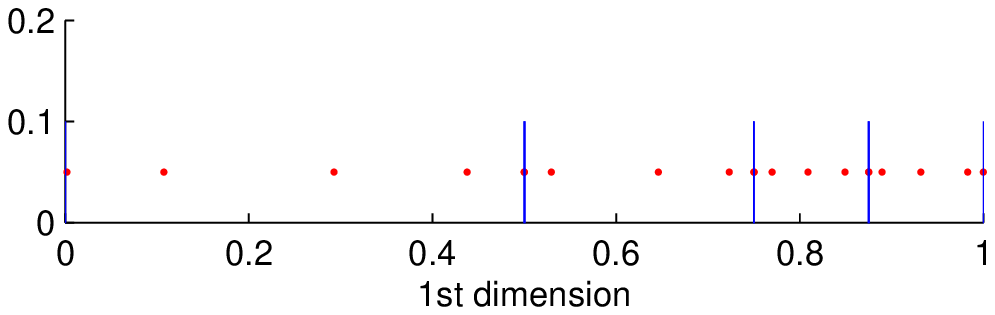}
\label{fig:subfig3}}
\qquad
\subfloat[][$C=1\times 10^{-3}$]{
\includegraphics[width=0.39\textwidth]{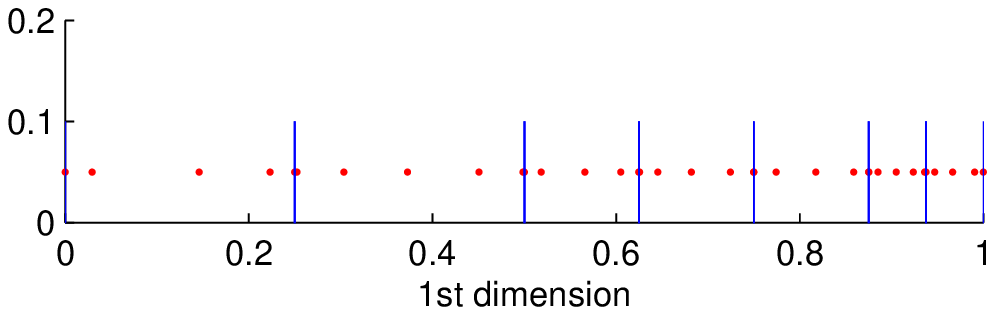}
\label{fig:subfig4}}
\\

\subfloat[][$C=5\times 10^{-4}$]{
\includegraphics[width=0.39\textwidth]{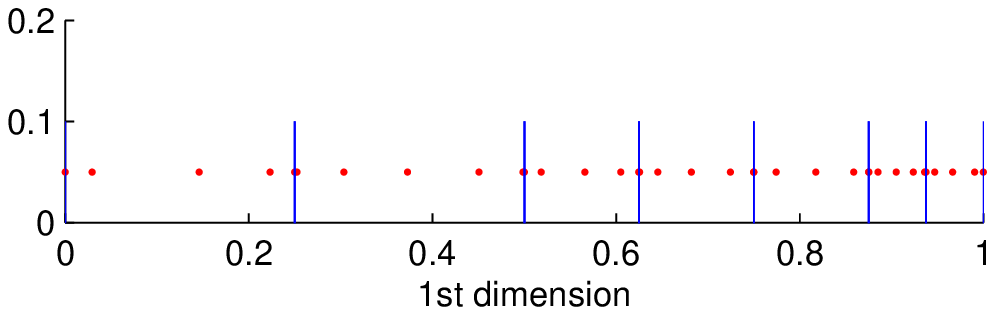}
\label{fig:subfig4}}

\caption{Partition of the random parameter space with sample points utilized for $n=1, n_0=3, \mbox{ and } n_s=5.$}
\label{fig:globfig}
\end{figure}

\begin{figure}
\subfloat[][$C=5\times 10^{-2}$]{
\includegraphics[width=0.39\textwidth]{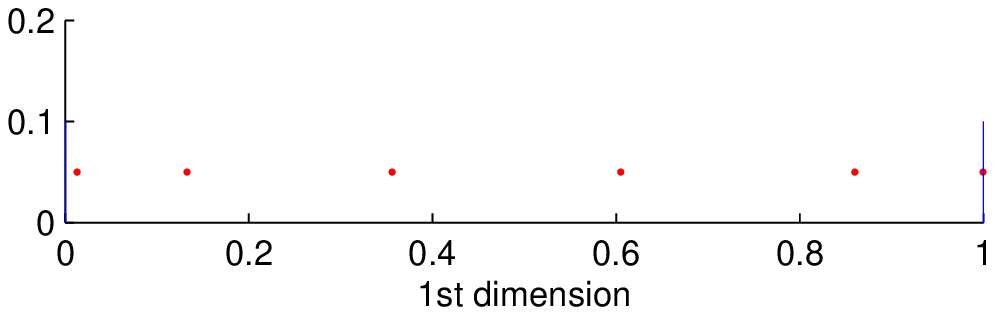}
\label{fig:subfig1}}
\qquad 
\subfloat[][$C=1\times 10^{-2}$]{
\includegraphics[width=0.39\textwidth]{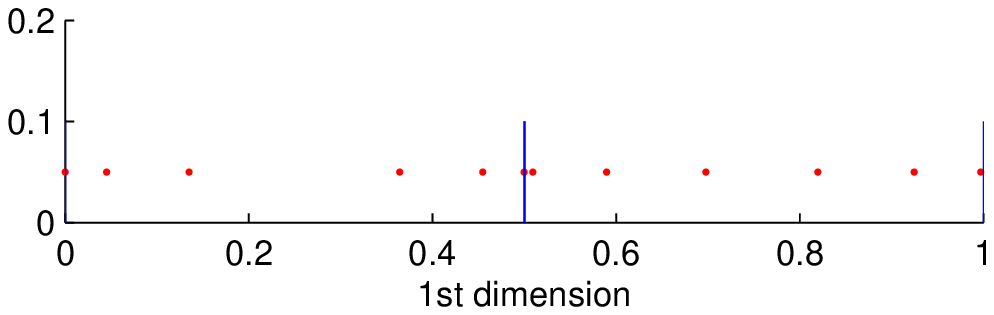}
\label{fig:subfig2}}
\\

\subfloat[][$C=5\times 10^{-3}$]{
\includegraphics[width=0.39\textwidth]{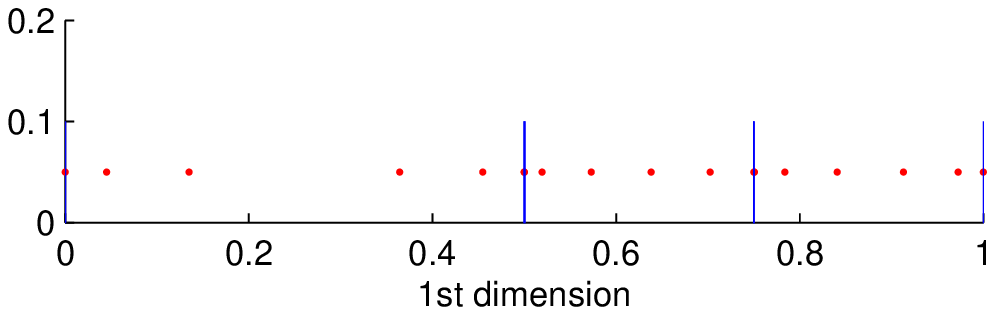}
\label{fig:subfig3}}
\qquad
\subfloat[][$C=1\times 10^{-3}$]{
\includegraphics[width=0.39\textwidth]{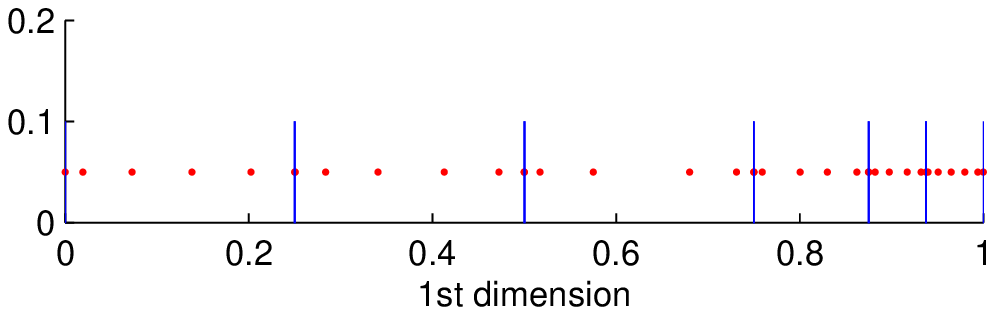}
\label{fig:subfig4}}
\\

\subfloat[][$C=5\times 10^{-4}$]{
\includegraphics[width=0.39\textwidth]{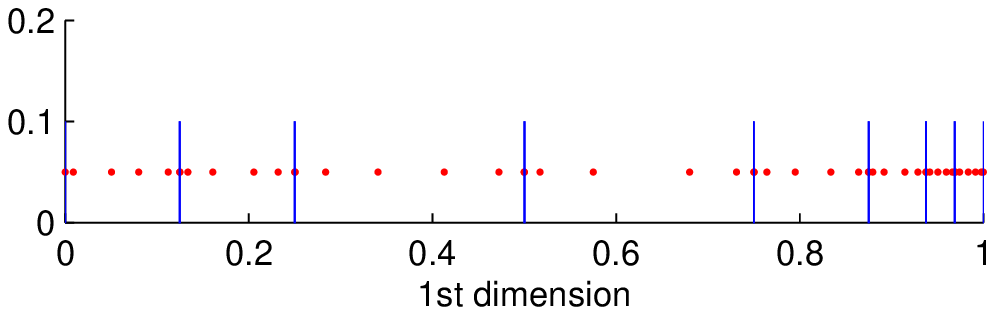}
\label{fig:subfig4}}

\caption{Partition of the random parameter space with sample points utilized for $n=1, n_0=4, \mbox{ and } n_s=6.$}
\label{fig:globfig}
\end{figure}

\begin{figure}
\subfloat[][$C=5\times 10^{-2}$]{
\includegraphics[width=0.5\textwidth]{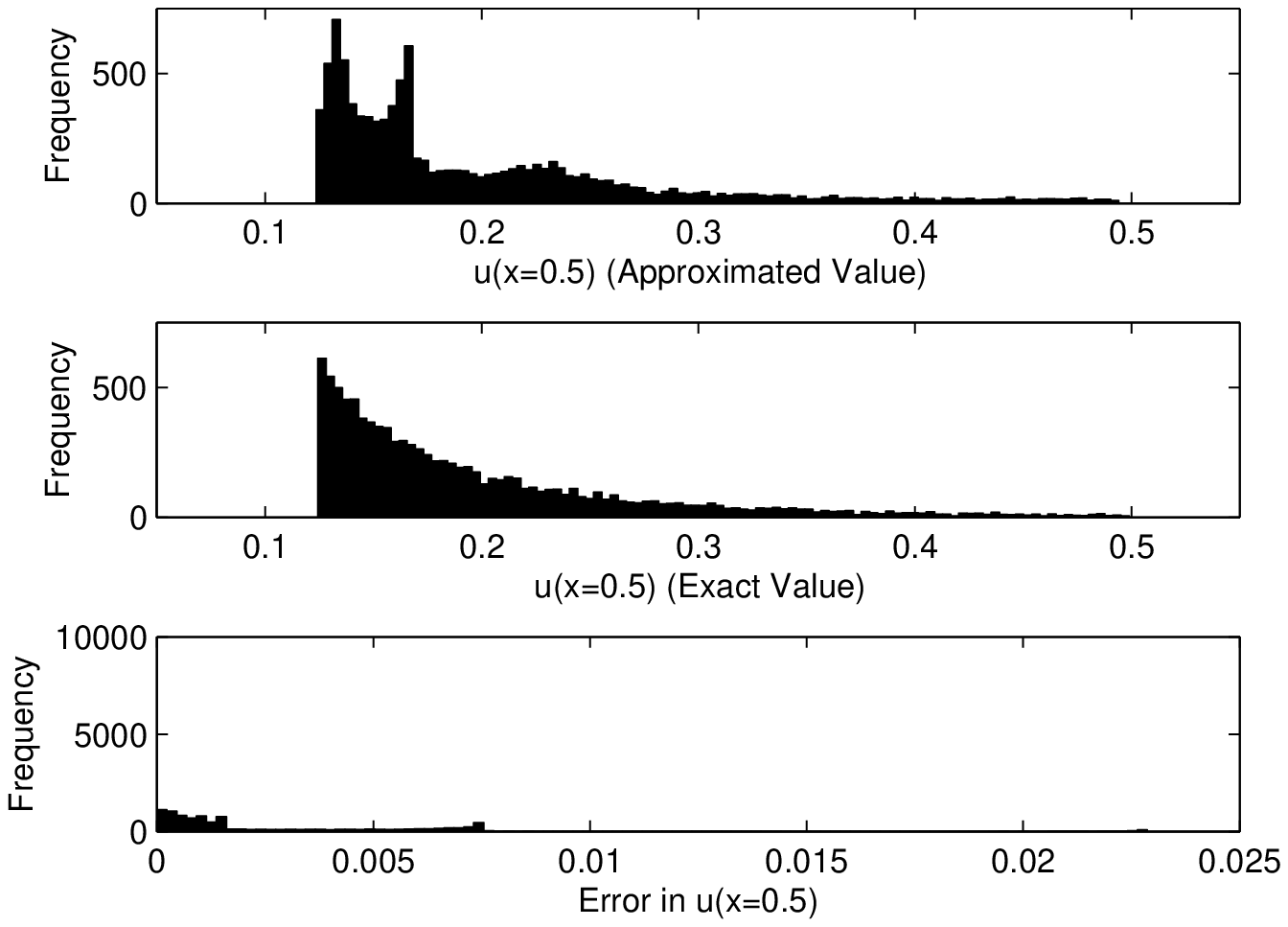}
\label{fig:subfig1}}
\qquad 
\subfloat[][$C=1\times 10^{-2}$]{
\includegraphics[width=0.5\textwidth]{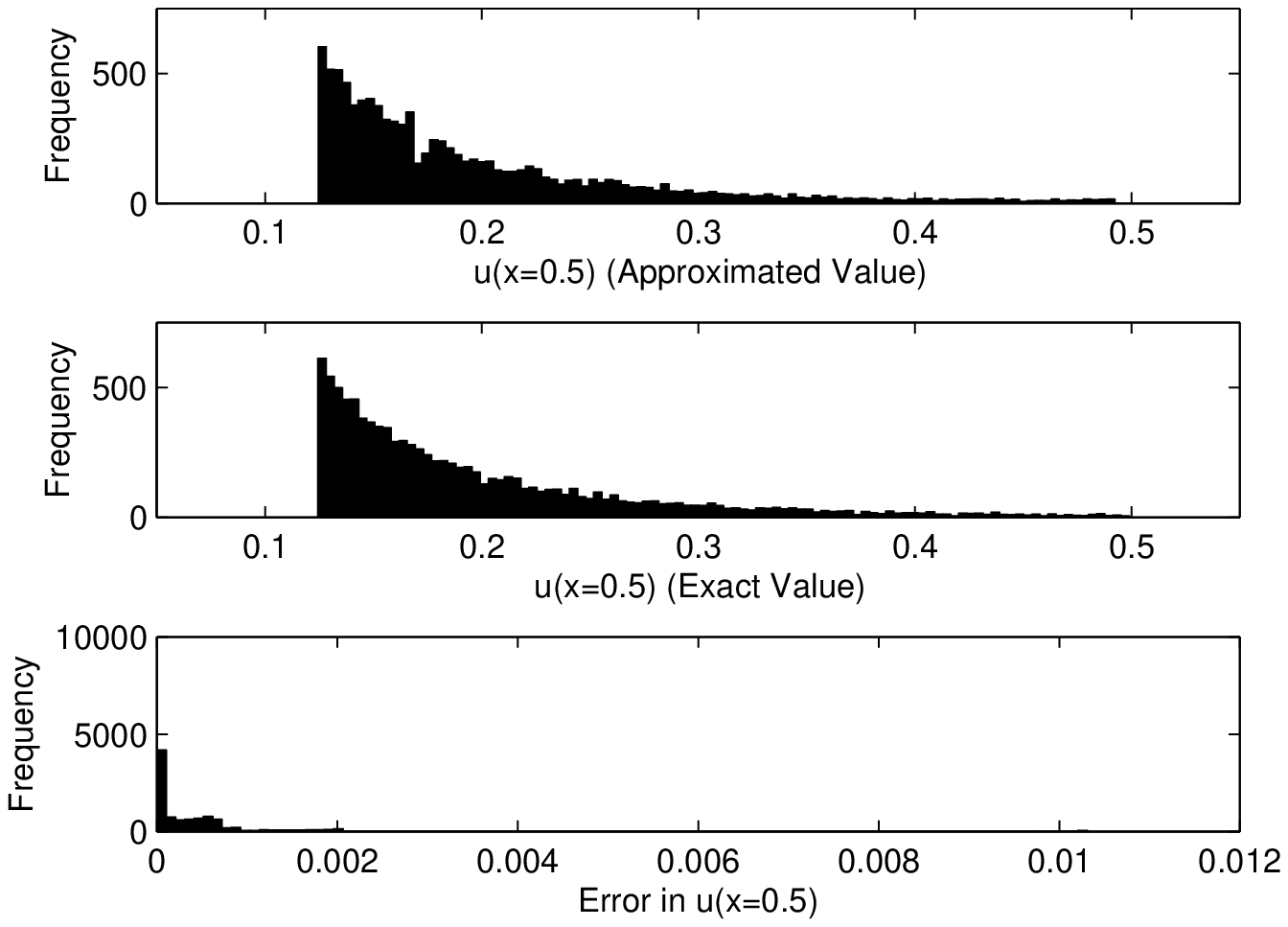}
\label{fig:subfig2}}
\\

\subfloat[][$C=5\times 10^{-3}$]{
\includegraphics[width=0.5\textwidth]{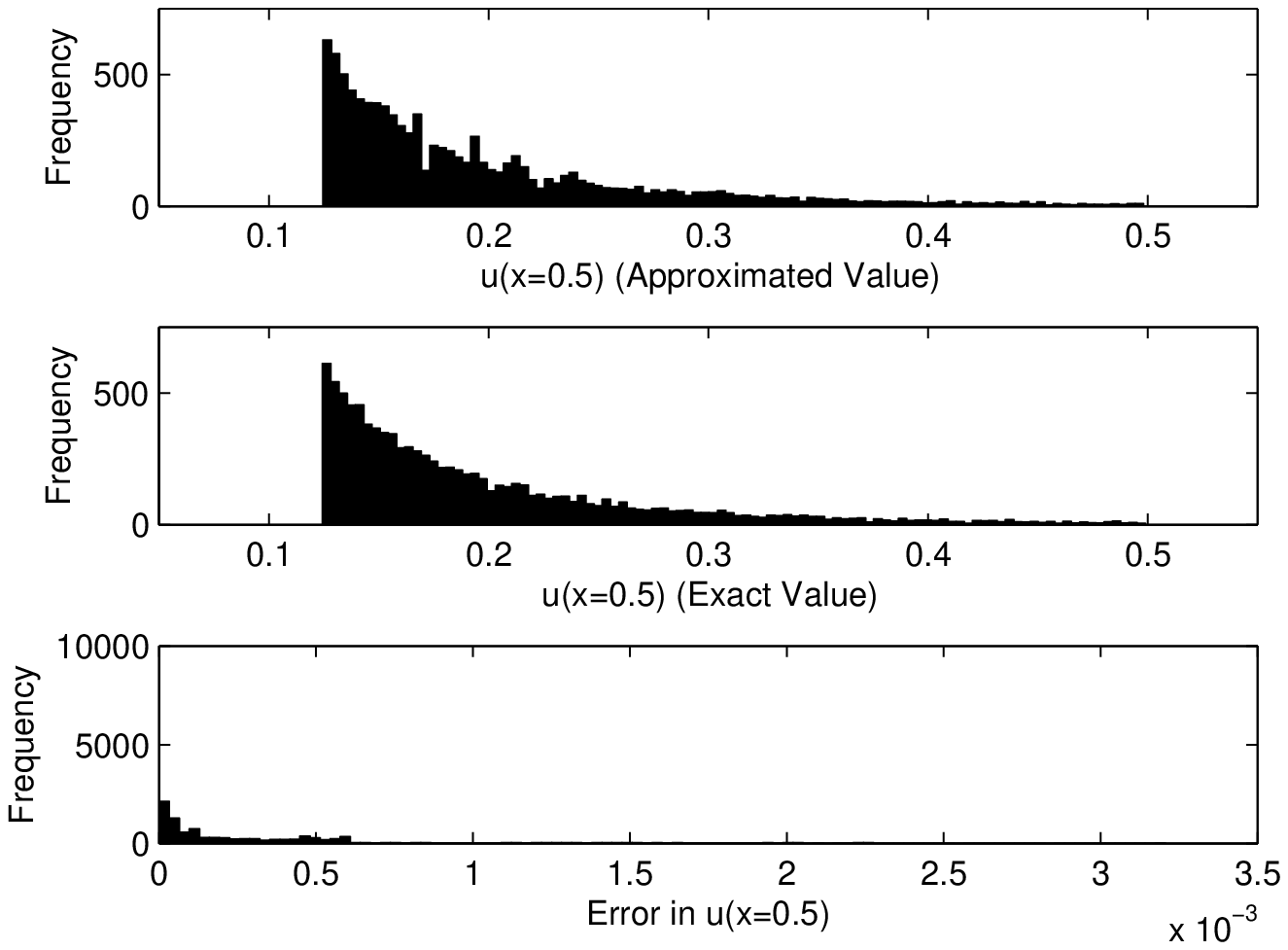}
\label{fig:subfig3}}
\qquad
\subfloat[][$C=1\times 10^{-3}$]{
\includegraphics[width=0.5\textwidth]{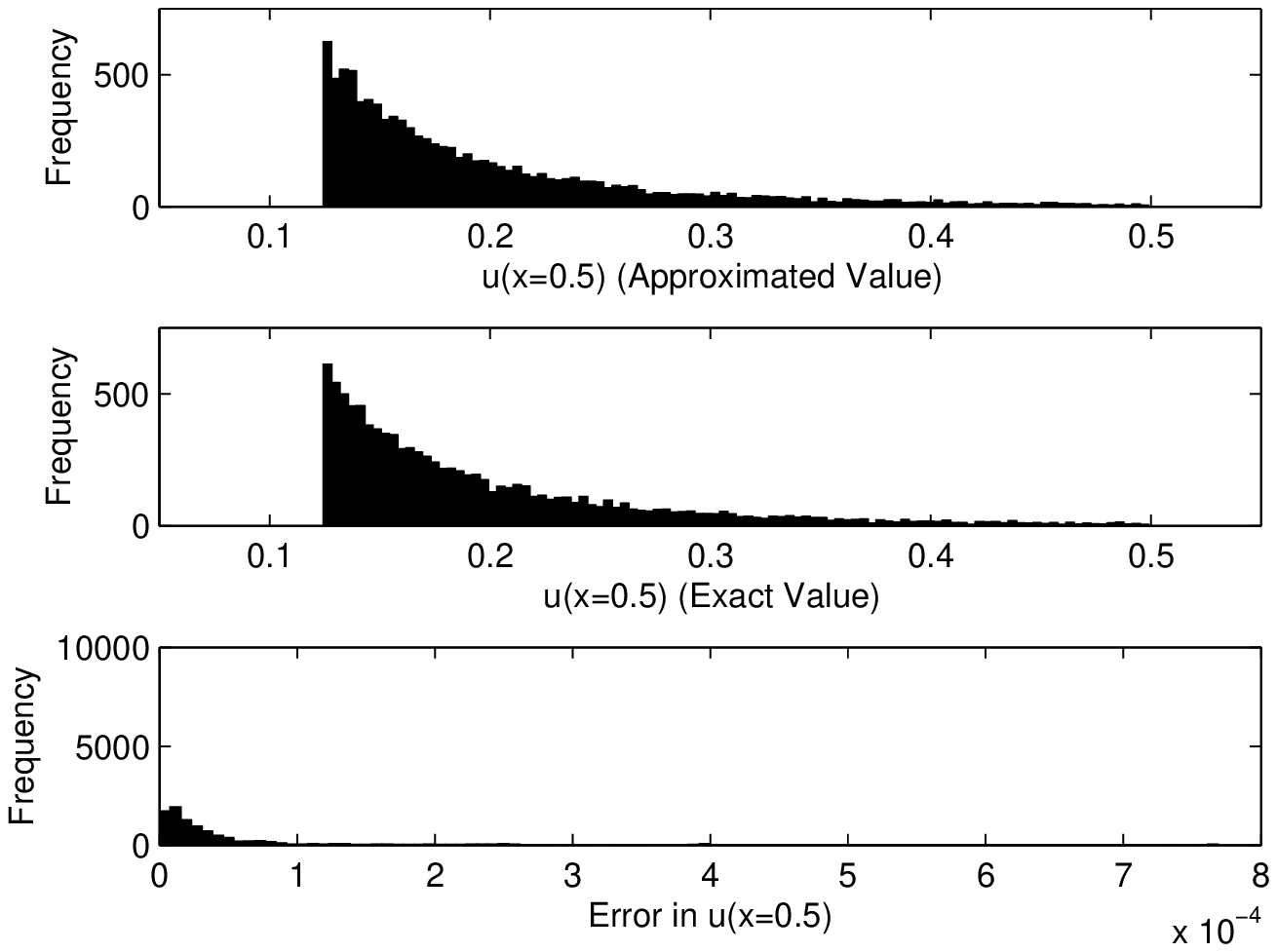}
\label{fig:subfig4}}
\\

\subfloat[][$C=5\times 10^{-4}$]{
\includegraphics[width=0.5\textwidth]{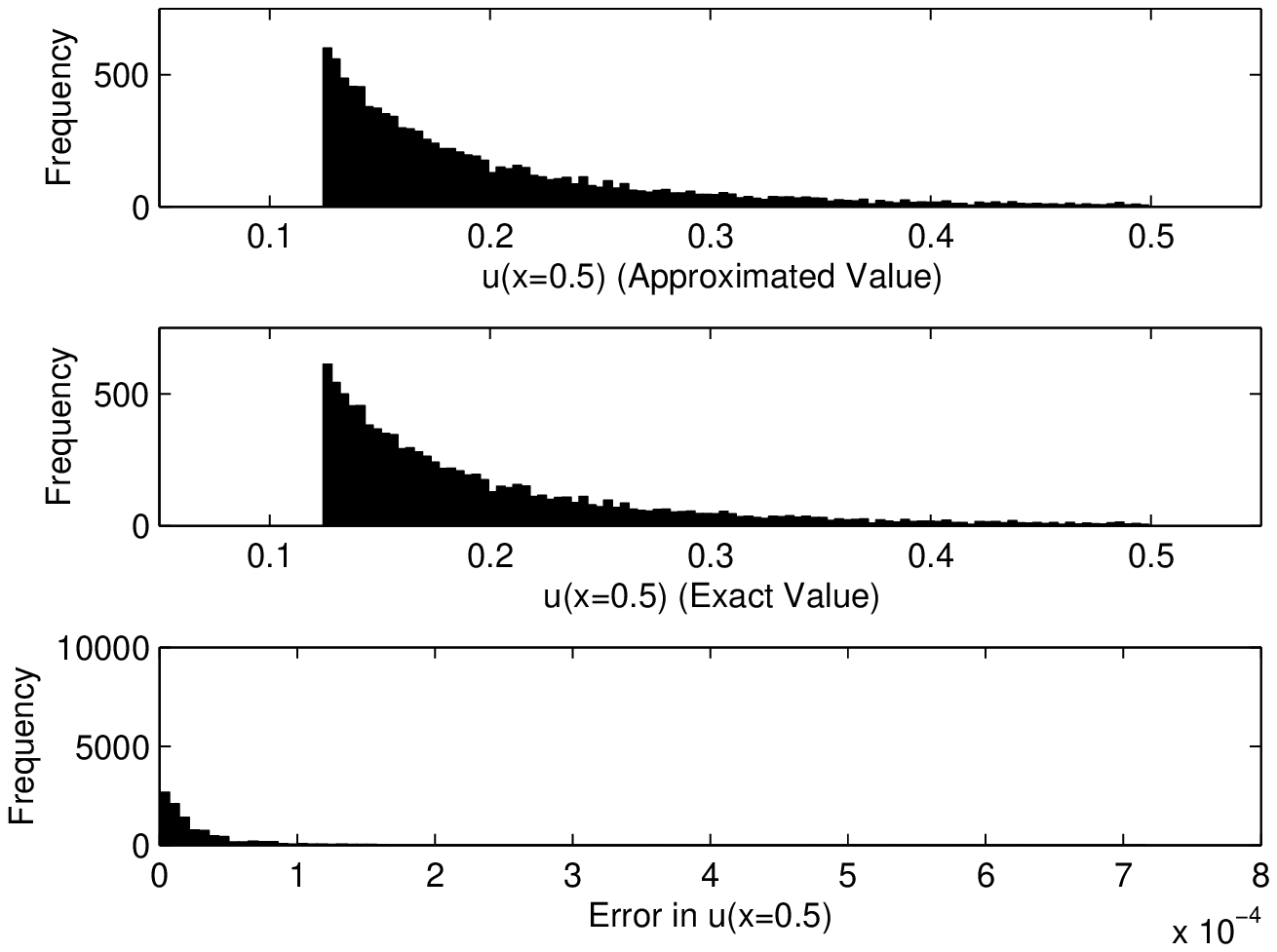}
\label{fig:subfig4}}

\caption{Distribution of function exact values evaluated by Monte Carlo simulation, function approximated evaluated by proposed method, and absolute value of error between them for $n=1, n_0=4, \mbox{ and } n_s=6.$}
\label{fig:globfig}
\end{figure}

\end{document}